 \documentclass[final,5p,times]{elsarticle}

\usepackage[utf8]{inputenc}

\usepackage{graphicx,array}

\usepackage{amsmath}
\usepackage{amssymb}
\usepackage{amscd}
\usepackage{amsfonts}
\usepackage{amsthm}
\usepackage{moredefs}
\usepackage{mathabx}
\usepackage{algorithm}
\usepackage{algpseudocode}
\usepackage{multirow}
\usepackage{color}
\usepackage{rotating}
\usepackage{url}
\usepackage{hyperref}
\usepackage{multimedia}

\theoremstyle{plain}

\newtheorem{Lem}{Lemma}

\newtheorem{Rem}{Remark}

\theoremstyle{definition}

\def\CL{\mathbb{C}^L}

\def\RR{\mathbb{R}}

\def\CC{\mathbb{C}}
\def\ZZ{\mathbb{Z}}
\def\NN{\mathbb{N}}

\newcommand{\prox}{\operatorname{prox}}

\newcommand{\bd}{\mathbf}
\newcommand{\supp}{\operatorname{supp}}

\journal{Applied Mathematics and computation}

\begin{document}

\begin{frontmatter}



\title{Designing Gabor windows using convex optimization}

\author{Nathana\"el Perraudin$^a$, Nicki Holighaus$^b$,\\ Peter L. S\o ndergaard$^c$
and Peter Balazs$^b$
}

\address{$^a$Nathana\"el Perraudin is with the Signal Processing Laboratory 2, \'Ecole polytechnique f\'ed\'erale de Lausanne,  CH-1015 Lausanne, Switzerland \\
$^b$Austrian Academy of Sciences, Wohllebengasse 12--14, 1040 Vienna, Austria \\
$^c$Oticon A/S, 2765 Sm\o rum, Denmark}


\begin{abstract}
\boldmath
Redundant Gabor frames admit an infinite number of dual frames, yet only the canonical dual Gabor system, constructed from the minimal $\ell^{2}$-norm dual window, is widely used. This window function however, might lack desirable properties, e.g. good time-frequency concentration, small support or smoothness. We employ convex optimization methods to design dual windows satisfying the Wexler-Raz equations and optimizing various constraints. Numerical experiments suggest that alternate dual windows with considerably improved features can be found. 
\end{abstract}


\end{frontmatter}




\section{Introduction}

Filterbanks, in particular those allowing for perfect reconstruction (PR), are fundamental and essential tools of signal processing. Consequently, the construction of analysis/synthesis filterbank pairs forms a central topic in the literature, relying on various approaches, like polyphase representation~\cite{ve86} or algebraic methods~\cite{SamadiAS04}. Other methods rely on frame theory~\cite{cos98}, similar to the approach we wish to present. Probably the most widely adopted type of filterbank are modulated cosine and Gabor filterbanks (or transforms)~\cite{gr01,fest98,ma92}, which are closely related.
Gabor transforms, also known as sampled short-time Fourier transforms, provide a uniform time-frequency representation by decomposing a signal into translates and modulations of a single \emph{window function}. They have been used in various applications, among others~\cite{Debbal20071041,roach2008event,Fujita2007153,adler2012audio}, and variations~\cite{zeevi1998multi,Li2015102,badohojave11}. Such filterbanks have a rich structure, are easy to interpret and allow for efficient computation. A substantial body of work exists on the subjects of invertibility of Gabor filterbanks, perfect reconstruction pairs of Gabor windows and window quality, with a strong emphasis on the overcomplete case~\cite{chkiki10,chkiki12}. The ability of the analysis filterbank to separate signal components and the precision of the synthesis operation, after coefficient manipulation, depend crucially on the time and frequency concentration of the windows used. While either the analysis or synthesis window can be chosen almost freely,
 tuned to the desired properties such as optimal time and frequency concentration, choice of the \emph{dual} window is restricted to the set of functions such that a PR pair is obtained. For computational reasons, detailed below, there is 
a 
canonical 
choice for 
the dual window, used almost exclusively. However, this \emph{canonical dual window} might not be optimal with regards to the desired criteria, such as time-frequency concentration or short support, required for high quality processing and efficient computation respectively. 

Therefore, a flexible method to compute optimal (or optimized) dual windows, considering the full set of possible choices and valid for any set of starting parameters, provides a valuable tool for the signal processing community. We obtain such a method by merging considerations from the theory of Gabor frames with the tools provided by modern convex optimization. The optimization framework we present is not limited to concentration or support optimization, but allows optimization with regards to any criterion that can be expressed through a suitable convex functional. Nonetheless, the aforementioned criteria are of universal importance and well-suited to demonstrate the capabilities and limitations of our method, which is why they form the focus of this contribution. 

Since all dual windows perform perfect reconstruction from unmodified Gabor coefficients, the purpose of constructing alternative dual windows might not immediately be obvious beside the minimization of the support. 
However, if the coefficients are modified, e.g. through signal processing procedures such as frame multipliers~\cite{feinow1,xxlmult1,balsto09new}, also known as Gabor filters~\cite{hlawatgabfilt1}, the shape of the dual window plays an important role in the quality and localization of the performed modifications. After processing, a signal is synthesized from the modified coefficients employing a dual Gabor filterbank. 
Let us illustrate the consequences of the window on the synthesis process after modification of the time-frequency representation with a toy example. For this example, we wish to remove an undesirable time-frequency component from a synthetic signal. In Figure~\ref{fig:gabor transform}, we want to remove a localized sinusoid with only minor alteration of the remaining signal. This filtering operation relies on a joint time-frequency representation, since at each time or frequency position, several signal components are active. Although both dual windows, naturally, fail to eliminate the sinusoid completely, they provide visibly different synthesis performance. 
\begin{figure}[!thp]
  \begin{center}
 \begin{minipage}[t]{0.5\textwidth} 
  \includegraphics[height=0.26\textwidth,width=0.48\textwidth]{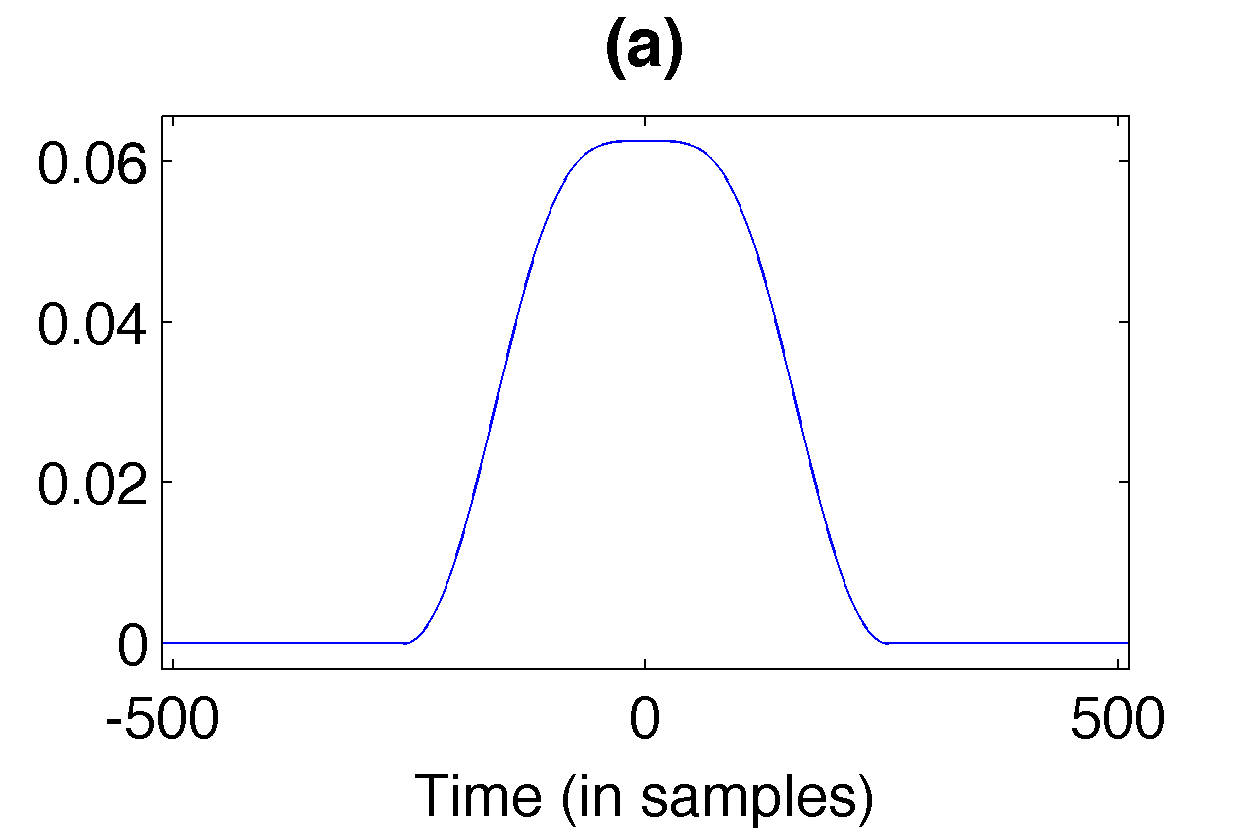}
  \includegraphics[height=0.26\textwidth,width=0.48\textwidth]{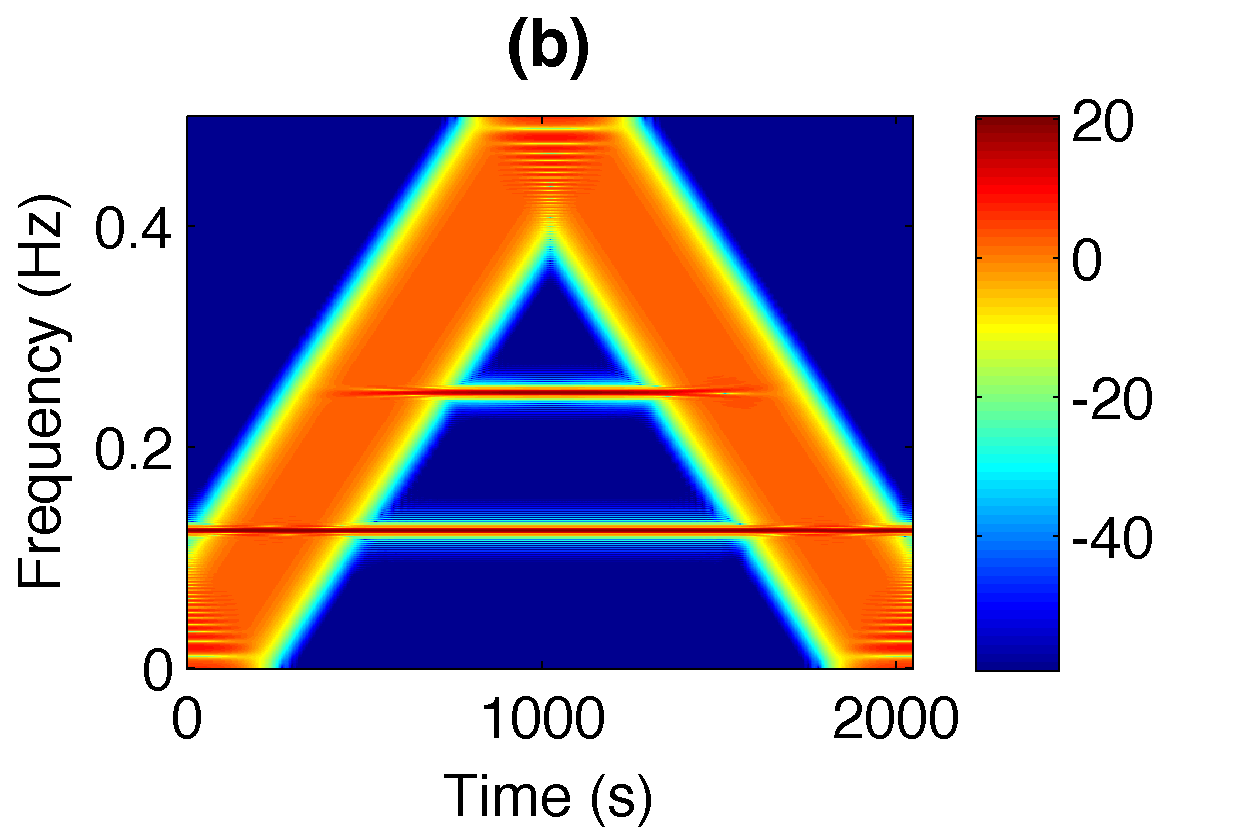}
 \end{minipage}
\par
\begin{minipage}[t]{0.5\textwidth} 
\includegraphics[height=0.26\textwidth,width=0.48\textwidth]{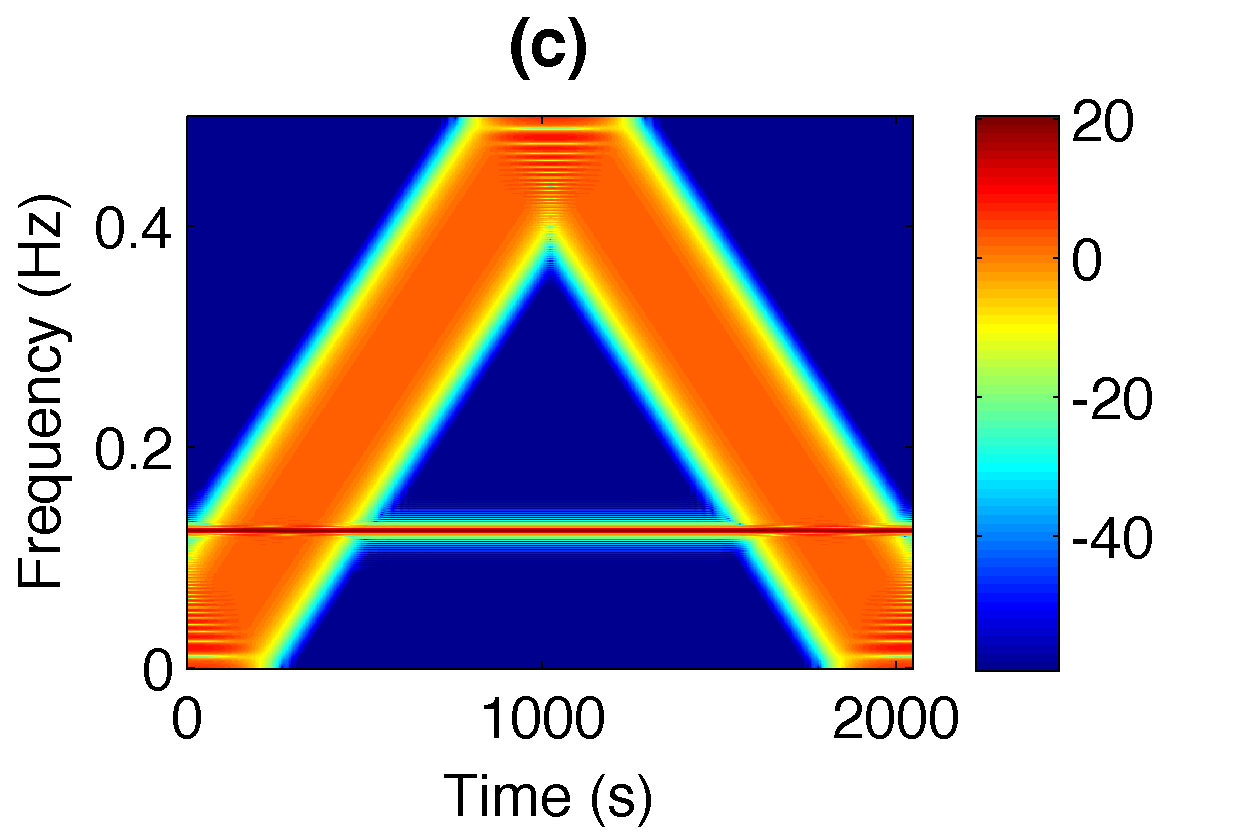}
\includegraphics[height=0.26\textwidth,width=0.48\textwidth]{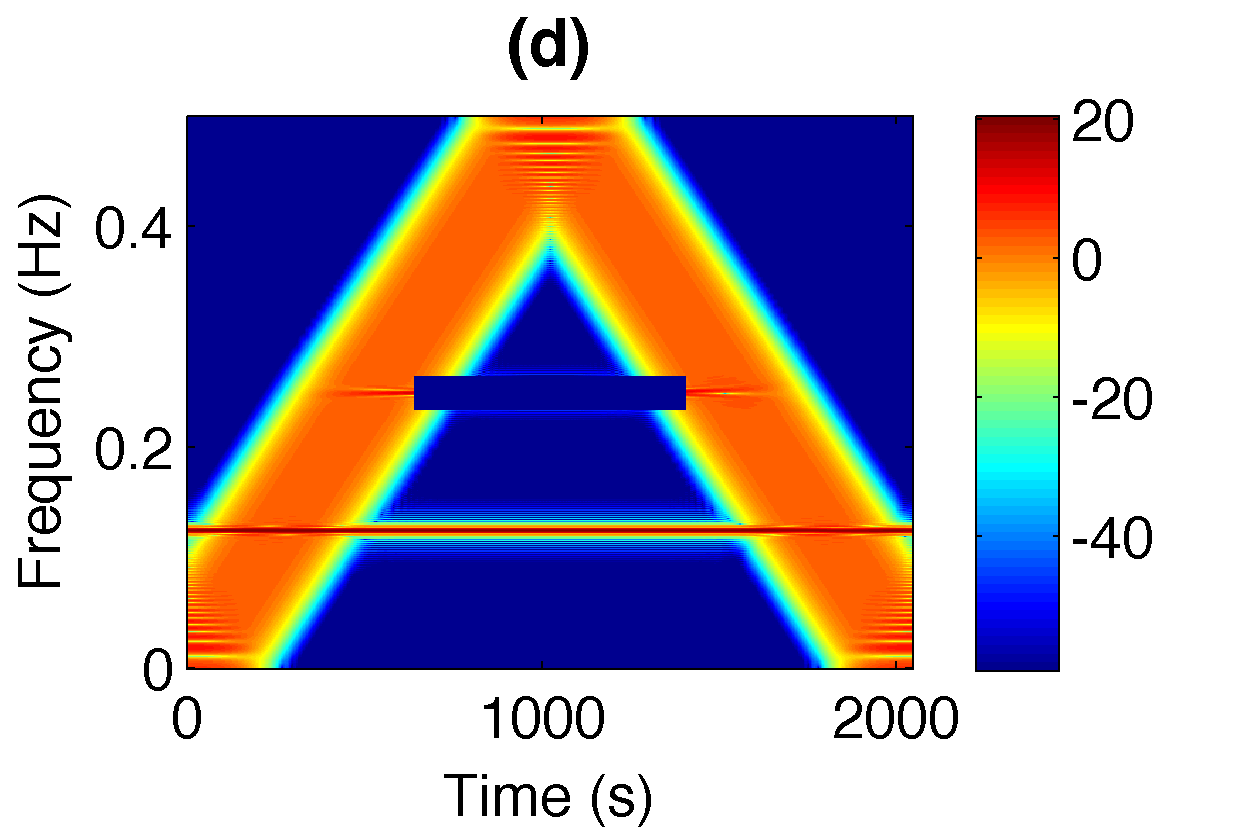}
  \end{minipage}
\par
\begin{minipage}[t]{0.5\textwidth} 
\includegraphics[height=0.26\textwidth,width=0.48\textwidth]{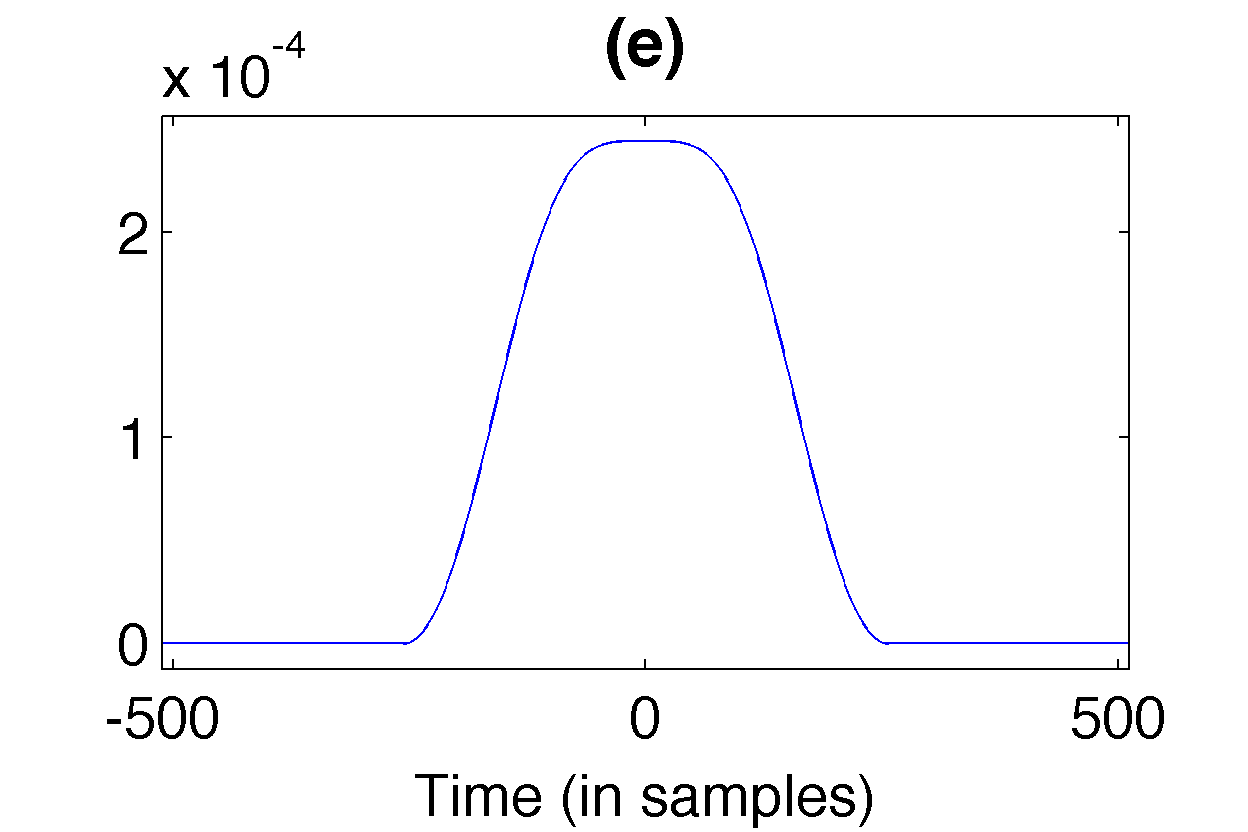}
\includegraphics[height=0.26\textwidth,width=0.48\textwidth]{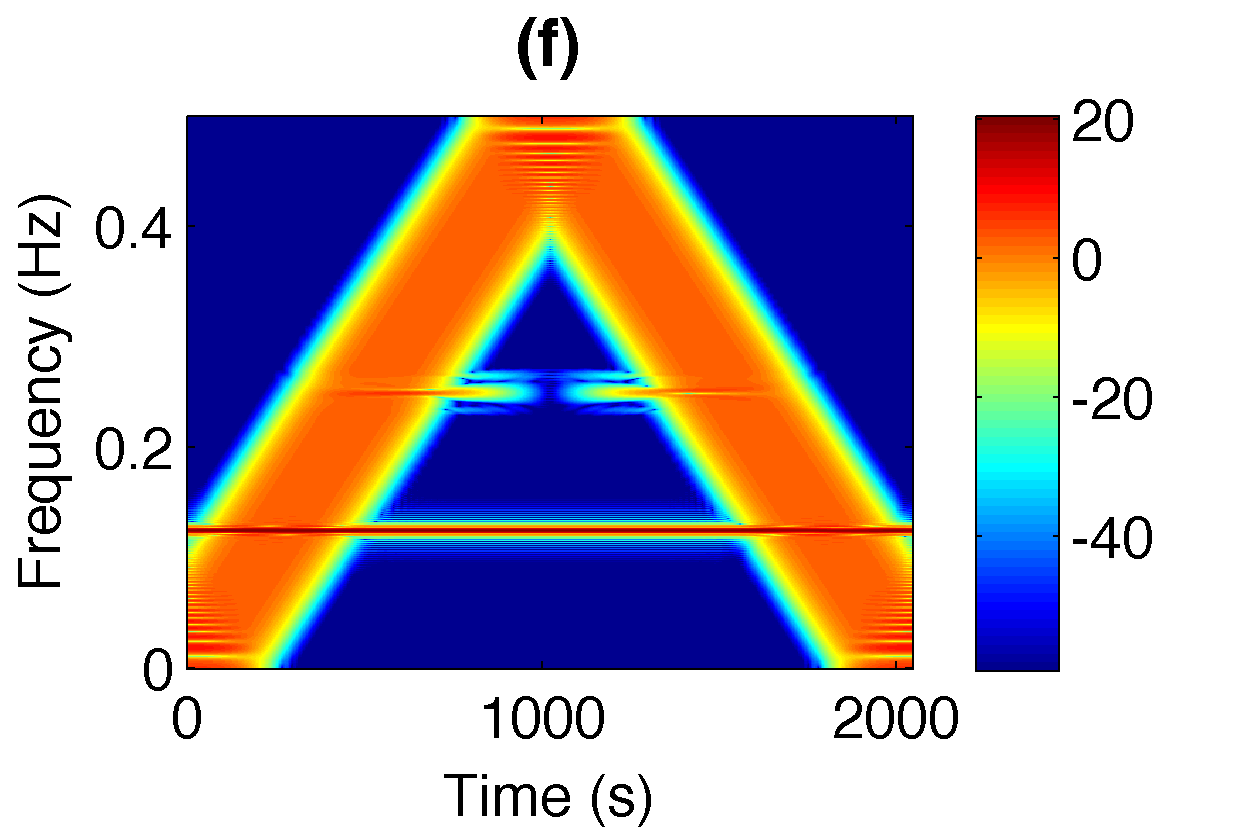}
  \end{minipage}
\par
\begin{minipage}[t]{0.5\textwidth} 
\includegraphics[height=0.26\textwidth,width=0.48\textwidth]{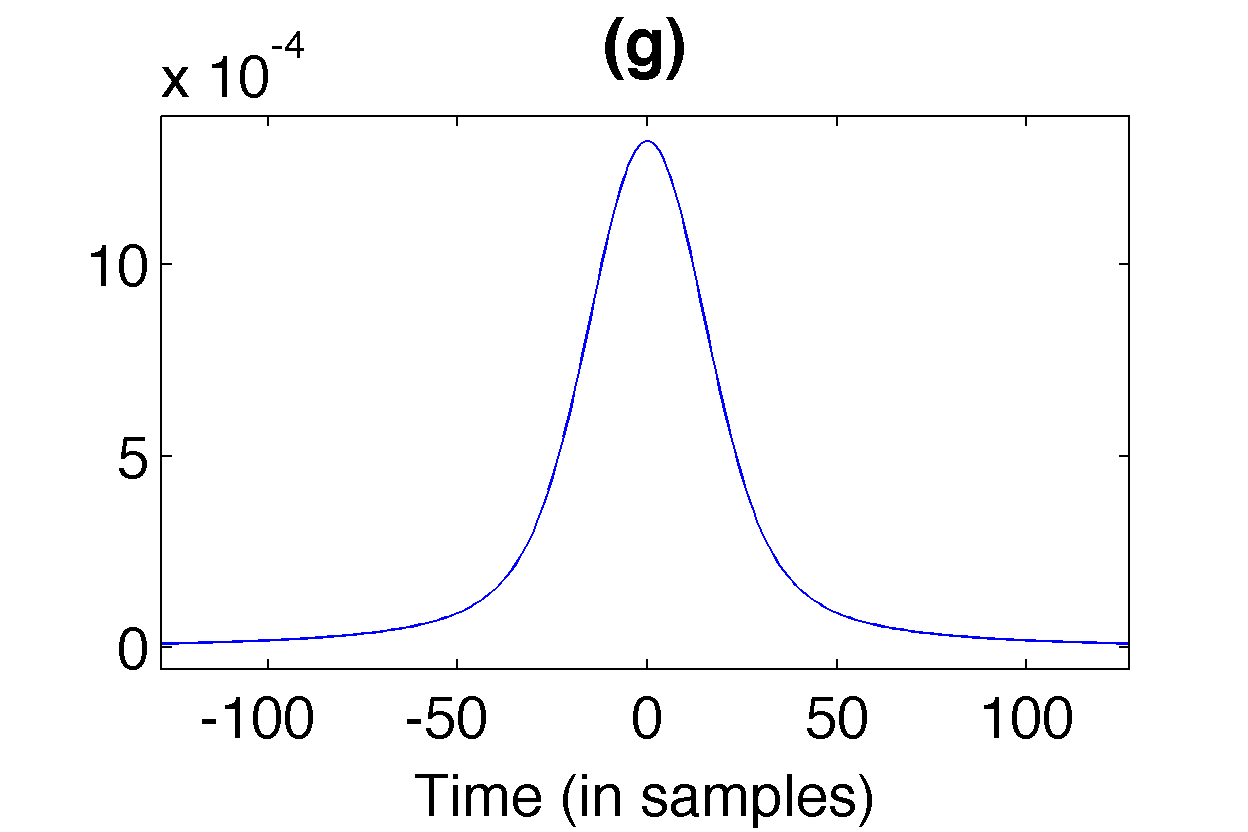}
\includegraphics[height=0.26\textwidth,width=0.48\textwidth]{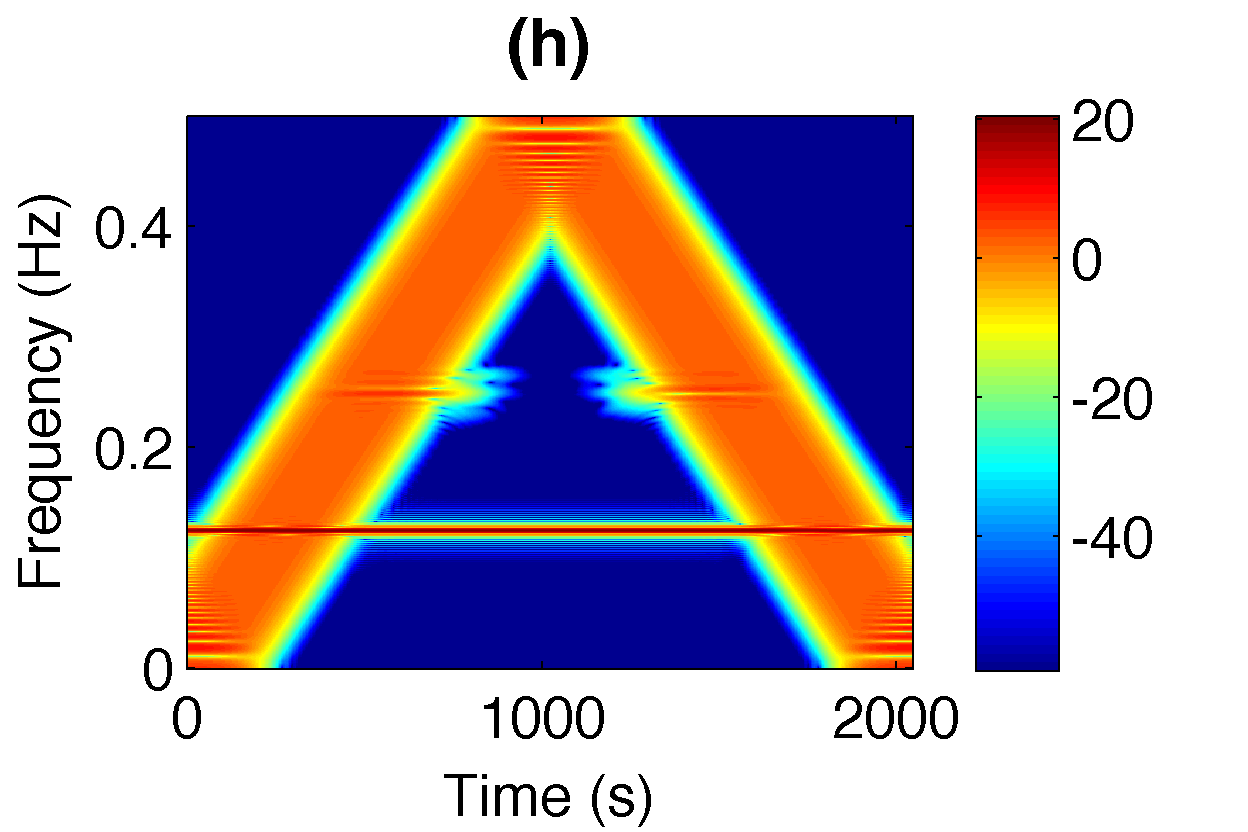}
  \end{minipage}
\end{center}
  \caption{Reconstruction from modified coefficients: (a) analysis window used to compute the spectrogram\protect\footnotemark,
 (b) Spectrogram of a synthetic test signal, (c) Unachievable 'oracle' target spectrogram, (d) modified spectrogram. (e)(g) $2$ different dual windows and (f)(h) spectrograms after synthesis from modified coefficients. Note that the \emph{smearing} effect depends on the concentration of the synthesis window. 
 \label{fig:gabor transform}
 }
\end{figure}

In the given example, the first synthesis operation was performed with the canonical dual window, while an alternative dual window with improved time concentration was computed for the second. For redundant Gabor filterbanks, infinitely many alternative dual windows exist, enabling the synthesis system choice based on the requirements of the desired application~\cite{badokowto13}. Yet, traditional methods utilize the so-called \emph{canonical dual} window, the only dual window that can be obtained directly by applying a linear operator (the inverse frame operator) to the original window.
\footnotetext{The spectrogram is the squared magnitude of a Gabor-type TF representation.}

We propose a convex optimization scheme that selects a dual window 
 such that one or more regularization parameters of the user's choice are optimized. 
 As the solution of a convex optimization problem, the proposed window will be optimal with regards to the selected criteria, provided that the set of admissible dual windows is nonempty.\\

\emph{Innovations of this contribution:}
We propose a general framework to compute Gabor dual windows satisfying individually selected criteria, expressed in a convex optimization problem. For this purpose, we combine frame-theoretical results with modern convex optimization. In particular, we demonstrate how our method can be used to provide dual windows with small support, essential for efficient computation, or with optimized time-frequency concentration. Employing various common measures of time-frequency concentration, our experiments show that these criteria can be considerably improved upon the widely-used canonical dual window. 

Through the selection of various optimization parameters, our method can be used to explore the set of dual Gabor windows, with the aim to better understand the restrictions imposed by the duality condition. 

Finally we introduce a, purely heuristic, method to design tight windows, based on the proposed convex optimization technique. Although the tight window problem is not anymore a convex problem and we cannot provide any convergence guarantee, we observe experimentally that favorable results can be obtained.

The implemented procedure is available as part of comprehensive open-source toolboxes maintained by the authors. An online addendum provides extended computational results and script files reproducing all the presented experiments.\\

\emph{Related work:}
A wealth of research investigating the subject of dual Gabor windows exists. On the theoretical side, state-of-the-art results are available in~\cite{gr01,ch08}, support properties of dual windows are investigated in~\cite{chkiki10,chkiki12,bo99,grst13,christensen2015entire} and some specific constructions for dual window pairs are presented in~\cite{ch06,chki10,chgo12,la09-2}, to name a few. Methods from complex analysis have been used to construct Gabor frames with Gaussian or Hermite windows~\cite{grochenig2007gabor,abreu2012banach,seip1992density}. Filters optimizing classical window quality measures are constructed in~\cite{Harris78,nu81-1}, without frame theoretic considerations. Recent results on phase space covers~\cite{romero2011surgery,romero2012characterization}, based on a collection of Gabor systems, require certain decay of the dual windows, not necessarily provided by the canonical dual. Thus, the study of alternate dual windows is necessary. 

A method for computing dual windows satisfying specific support constraints was proposed by Strohmer~\cite{st98-8}, based on the Moore-Penrose pseudoinverse of a linear equation system describing duality and the support conditions. The method therein allows the use of other regularization constraints, if these can be expressed through a Hermitian positive definite matrix, e.g. weighted $\ell^2$-norms. Already Wexler and Raz~\cite{wexler1990discrete} impose linear constraints to find alternative dual windows, while Daubechies et al.~\cite{daubechies1994gabor} present a formula for finding dual windows that are optimal in a 
modified $L^2$ sense. More similar to our approach, the authors in~\cite{li2012sparse} solve a different convex optimization problem to find sparse dual systems using a weighted $\ell^1$-norm. They also provide a method to obtain compactly supported dual windows. This problem is optimized for the research of sparse windows and not suitable for other constraint such as smoothness. With the method proposed in this paper, similar results can be obtained, but sparse dual windows only form a particular example for its possible applications. The existence of sparse dual frames is also investigated in~\cite{Krahmer2013982}, without assuming a filterbank structure.

Recently, convex optimization in the context of signal processing has grown into a active field of research and in particular proximal splitting methods~\cite{combettes2011proximal,combettes2007douglas,combettes2005signal} have been used to great effect, e.g. in audio inpainting~\cite{adler2012audio,adler2011constrained} and sparse representation~\cite{kowalski2012social}. In those cases, optimization techniques are applied directly to the signal or its time-frequency (TF) representation. In this contribution, we apply optimization techniques to shape the building blocks of the TF representation instead.\\

\emph{Motivation and potential applications:}
Motivated by Figure~\ref{fig:gabor transform} it seems clear that an optimization of the dual window can have significant relevance for acoustical applications. However, to our knowledge, alternate duals are not yet used in that field, while there are already some signal processing investigations using alternate duals for more than a decade, see e.g.~\cite{4156434}.
That is mainly because the acceptance of frame theory by engineers in signal processing is an ongoing process~\cite{boelc1,vettkov1}, and still not fully completed. 
As a consequence, frame theory, and with it the usage of the canonical dual window, has been established as a useful tool in audio and acoustical applications only rather recently~\cite{framepsycho16}. 
While already before the acceptance of frame theory some investigation in the tuning of the re-synthesis stage has been performed (for example in computational auditory scene analysis~\cite{wanbro06}, or speech processing~\cite{quat01}), 
applied scientists often used the same window~\cite{po76} or even the rectangular window in the overlap-add approach~\cite{opp}. 

In this paper we provide intuitive ideas and efficient tools for applied scientists to work with alternate dual windows. We are convinced that many applications will benefit form optimized dual windows, as soon as a coefficient domain modification is attempted. Such modification is ubiquitous in signal processing, e.g. in denoising~\cite{majxxl10}, signal detection~\cite{Hall1984}, time-stretching and pitch-shifting~\cite{zo02,liro13,drmu16,prho17-1}, modification of the spectrogram~\cite{griflim84,nathxxl13}, irrelevance filtering~\cite{Necciari2012146,xxllabmask1}, speech recognition~\cite{NarWan13}, to name a few. Furthermore, jointly optimizing support and smoothness of the synthesis window can provide block-processing algorithms with reduced delay and blocking artifacts, possibly even at lower redundancy than usual. Without a doubt, optimized dual windows  can have numerous applications in the future.\\

\emph{Organization of the paper:}
This work is an extension of~\cite{persampta13}. For the sake of being self-contained, we repeat some results presented therein. 
We begin by recalling the essential background from convex optimization, as well as important concepts from the theory of Gabor frames. In particular, we are interested in the characterization of dual Gabor windows by means of a linear equation system and we show that compactly supported dual window pairs are dual independent of the signal length. The convex optimization problem central to our investigation is introduced in Section~\ref{sec:design}, where we also discuss the considered optimization criteria and their effects in the setting of dual Gabor windows. Finally, Section~\ref{sec:results} presents a set of examples designed to illustrate how to use our method for optimizing or adjusting the time and frequency concentration of the dual window for a given starting window. We compare the results visually, in terms of the optimization criteria and classical measures of window quality. Further numerical experiments demonstrate the construction of smooth dual windows with short support and good frequency 
concentration 
and how the proposed optimization scheme can be employed, albeit heuristically, for finding 
time-frequency concentrated, compactly supported tight windows. 

\section{Preliminaries}

In this contribution, we consider sampled functions, i.e. sequences $f$ in $\ell^{2}(\ZZ)$ or $\CC^L$. The latter is interpreted as the space of $L$-periodic sequences with indices considered modulo $L$.
For such $f$, we refer to the smallest closed interval containing all nonzero values of $f$ as its support, denoted by $\supp(f)$. By $x|y$ we denote that $y/x\in\ZZ$.

Furthermore, we denote \emph{translation} and \emph{modulation} operators
\[
 \bd T_{n}f[l] = f[l-n] \text{ and } \bd M_\omega f[l] = f[l]e^{2\pi i \omega l},
\]
 for $f\in\ell^2(\ZZ)$, $l,n\in\ZZ$ and $\omega\in [0,1)$. Their counterparts on $\CC^L$ are defined in the usual way with indexing modulo $L$.
 

\subsection{Convex optimization and proximal splitting}\label{ssec:formandalg}

In subsequent sections we design Gabor dual windows as the solution to convex optimization problems of the 
form 
\begin{equation}
  \underset{x\in\mathbb{R}^{L}}{\text{minimize}}\sum_{i=1}^{K}f_{i}(x),\label{general_problem_proximal}
\end{equation}
where the $f_{i}$ are convex functions that promote certain features in the solution. Those functions are also referred as \emph{(regularization) priors}. Various methods exist to solve this kind of 
problem for differentiable priors~\cite{hager2006survey}. However, proximal splitting methods~\cite{combettes2011proximal} only require the $f_i$ to be lower semi-continuous, convex and proper\footnote{A proper function $f$ on a domain $C$ is function satisfying $f(x)>-\infty, \forall x \in C$ and $\exists x \in C \text{ such that } f(x)<\infty$} functions, thus allowing to increase the design freedom. Those methods solve Equation~\eqref{general_problem_proximal} by iteratively applying the proximity operator
$$\prox_{f}(y):=\mathop{\operatorname{arg~min}}\limits _{x\in\mathbb{R}^{L}}\left\{ \frac{1}{2}\|y-x\|_{2}^{2}+f(x)\right\}$$ to each prior $f_i$. More information and convergence results can be found 
in~\cite{rockafellar1976monotone,martinet1972determination,combettes2011proximal}. 

Restriction of the optimization to a convex subset $\mathcal C$ of $\RR^L$, e.g. the set of dual Gabor windows, is achieved by selecting the indicator function 
\begin{equation}
i_{\mathcal C}:\mathbb{R}^{L}\rightarrow\{0,+\infty\}:x\mapsto\begin{cases}
0,\hspace{0.25cm} & \text{if}\hspace{0.25cm}x\in \mathcal C\\
+\infty\hspace{0.25cm} & \text{otherwise.}
\end{cases},\label{eq:indfun}
\end{equation}
as prior. Its proximity operator is given by the orthogonal projection on $\mathcal C$.

\subsection{Gabor systems, frames and dual windows}\label{ssec:gabframes}

We define the Gabor system 
\begin{equation}
\mathcal{G}(g,a,M):=\left(g_{m,n}=\bd \bd M_{m/M} T_{na}g\right)_{n\in\ZZ,~m=0,\ldots,M-1}
\end{equation}
for a given window $g\in\ell^{2}(\ZZ)$, a hop size $a\in\mathbb{Z}$ and a number of frequency bins $M \in\mathbb{Z}$.

For any signal $f\in\ell^{2}(\ZZ)$, the \emph{Gabor coefficients} (or Gabor analysis) with respect to $\mathcal{G}(g,a,M)$ are given by 
\begin{equation}\label{eq:gabana}
(\bd{G}f)[m+nM]=\langle f,g_{m,n}\rangle=\sum_{l\in\ZZ}f[l] \, \overline{g_{m,n}[l]},
\end{equation}
with the analysis operator $\bd{G}$ given by the infinite matrix
$\bd{G}[m+nM,l]:=\bd{G}_{g,a,M}[m+nM,l]:=\overline{g_{m,n}[l]}$.

\emph{Gabor synthesis} of a coefficient sequence $c\in\ell^2(\ZZ)$ with respect to $\mathcal{G}(g,a,M)$ is performed by
\begin{equation}\label{eq:gabsyn}
f_{syn}[l]=(\bd{G}^{\ast}c)[l]=\sum_{n\in\ZZ}\sum_{m=0}^{M-1} c[m+nM] \,g_{m,n}[l],
\end{equation}
where $\bd{G}^{\ast}$ denotes the transpose conjugate of $\bd{G}$. 
In $\CC^L$, a Gabor transform with $a=1$ and $M=L$ is also known as (full) \emph{short-time Fourier transform} (STFT).
This highly overcomplete setup allows for straightforward inversion using the synthesis operator, i.e. $f = \bd{G}^{\ast}\bd{G}f$~\cite{gr01}. Otherwise, 
we require that $\mathcal{G}(g,a,M)$ forms a stable, (over-)complete system satisfying
\begin{equation}
A\|f\|_{2}^{2}\leq\|\bd{G}f\|_{2}^{2}\leq B\|f\|_{2}^{2},\ \text{for all}\ f\in\ell^{2}(\ZZ),
\end{equation}
for some $0<A\leq B<\infty$, i.e. a \emph{Gabor frame}~\cite{ch08}. In that case, every signal $f\in\ell^{2}(\ZZ)$ can be written as 
\begin{equation}\label{eq:gabsynpr}
f=\bd{G}^{\ast}_{g,a,M}c
\end{equation}
for some coefficient sequence $c\in\ell^{2}(\ZZ)$. A frame is \emph{tight}, if $A=B$ is a valid choice. Then $h=g/A$ is
a dual window. 
The frame property guarantees the existence of a dual Gabor frame $\mathcal{G}(h,a,M)$ such that 
$f= \bd{G}^{\ast}_{h,a,M}\left(\bd{G}_{g,a,M}f\right)=\bd{G}^{\ast}_{g,a,M}\left(\bd{G}_{h,a,M}f\right)$
holds for all $f\in\ell^2(\ZZ)$. Hence, $c$ in~\eqref{eq:gabsynpr} can be chosen to be the Gabor coefficients
with respect to $\mathcal{G}(h,a,M)$. 

If $\mathcal{G}(g,a,M)$ is redundant, then the \emph{dual window} $h\in\ell^2(\ZZ)$ is not unique. 
Instead, the space of dual windows equals the solution set of the 
\emph{Wexler-Raz (WR) equations}~\cite{wexler1990discrete,ltfatnote006}, that 
characterize the dual Gabor windows for $\mathcal{G}(g,a,M)$. They are given
by 
\begin{equation}
\begin{split} \frac{M}{a}\left\langle h,g[\cdot-nM]e^{2\pi im\cdot/a}\right\rangle =\delta[n]\delta[m]  \\
\text{ or } \bd{G}_{g,M,a}h = [a/M,0,0,\ldots]^T, \label{eq_wexler_raz} \end{split}
\end{equation}
for $m=0,...,a-1,\ n\in\ZZ$. Here, $\delta$ denotes the Kronecker delta. Note that, while $\bd{G}_{g,a,M}$ is overcomplete, $\bd{G}_{g,M,a}$ is underdetermined and admits
infinitely many solutions, whenever $a<M$.

The WR equations form the central step towards the formulation of the Gabor dual problem 
in the context of convex optimization. Any function in this set facilitates perfect reconstruction from unmodified coefficients,
but some are better suited for synthesis from processed coefficients than others, see Figure~\ref{fig:gabor transform}. 
From now on, we will denote by $\mathcal{C}_{dual}$ the solution set of the nontrivial WR equations, forming the basic constraint of the considered optimization problem.

The \emph{canonical dual window}, defined via the pseudoinverse of the analysis operator $\bd{G}_{g,a,M}$, is the only widely used dual. It can be computed efficiently, see e.g.~\cite{xxlfei1,ltfatnote007}. 
We note that the canonical dual window $\gamma$ minimizes the $\ell^2$-norm as well as the $\ell^2$-distance to $g$ among all duals, see e.g~\cite[Prop. 7.6.2]{gr01}. 
Unless certain very specific conditions are satisfied, the canonical dual is infinitely long~\cite{bo99}, preventing finite time synthesis. The most prominent setup that provides a compactly supported canonical dual
is the \emph{painless case}, i.e.  when the length of $g$ is less or equal than the number of channels $M$. Therefore the setup
where the length of the window equals the number of channels is omnipresent in signal processing, to the point where these two numbers are sometimes not
distinguished. For integer redundancy, the conditions in~\cite{bo99} are even equivalent to the painless case.\\

\textbf{Gabor dual windows beyond the canonical dual:}
  A considerable amount of research on alternative Gabor dual windows has been conducted, mostly concerned with finding dual pairs of windows with compact support. Such results often consider special configurations of analysis window~\cite{chki10,la09-2} and/or Gabor parameters~\cite{ch06,chkiki10,chkiki12}. While compactly supported duals play a central role in this contribution, our method admits further design freedom and does not impose constraints on the analysis window or Gabor parameters. To ensure efficient computation, it is crucial to establish the independence of the duality conditions from the signal length $L$, for compactly supported pairs of dual windows. This property, while widely accepted in the community, seems not to have found its way into the literature explicitly. Since it forms a central point of our argument, we will now state the result including a short proof. We now assume the existence of finite intervals $I_g$, $I_h$ such that $\supp(g)\subseteq I_g$ 
for the analysis window and $\supp(h)\subseteq I_h$ for the solution dual window.

\begin{Lem} \label{sec:dualfin1}
  Let $I_g$, $I_h$ be intervals of length $L_g$ and $L_h$ with nonempty intersection. For any Gabor system $\mathcal G(g,a,M)$ with $\supp(g)\subseteq I_g$ and any $h\in\ell^2(\ZZ)$ with $\supp(h) \subseteq I_h$, all but $a\lceil\frac{L_{g}+L_{h}}{M}\rceil$ of the WR equations are trivally satisfied. Moreover, if $\langle h, g \rangle_{\ell^2}=a/M$, then the following are equivalent:
  \begin{itemize}
   \item[(i)] $g,h$ are Gabor dual windows on $\ell^2(\ZZ)$ for $a,M$,
   \item[(ii)] For any $L > L_g + L_h$ with $a,M\mid L$, $g_{fin},h_{fin}$, defined by $g_{fin}[l] = \sum_{k\in\ZZ} g[l - kL]$ and $h_{fin} = \sum_{k\in\ZZ} g[l - kL]$ for $l = 0,\ldots,L-1$, are Gabor dual windows
   on $\CC^L$ for $a,M$.
  \end{itemize}
  Moreover, (i)$\Rightarrow$(ii) holds for any $L \geq L_g,L_h$ with $a,M\mid L$.
\end{Lem}
\begin{proof}
  By assumption there are $n_0,n_1\in\ZZ$, with $n_0\leq 0\leq n_1$ such that $I_{h}\cap(I_{g}+nM)\neq\emptyset$ for $n_0\leq n\leq n_1$ and $I_{h}\cap(I_{g}+nM)=\emptyset$ for every other $n\in\ZZ$. In particular, 
  $n_1-n_0 \leq \lceil\frac{L_{g}+L_{h}}{M}\rceil-1$. Therefore $\langle h, \bd M_{ma^{-1}} \bd T_{nM}g \rangle = 0$ for all $n\in\ZZ$ s.t. $n<n_0$ or $n>n_1$, proving that at most $a\lceil\frac{L_{g}+L_{h}}{M}\rceil$ of the WR equations are not trivial. Now let $L\in\NN$ such that $L > L_g + L_h$ and $M\mid L$. It is easily seen that 
  \begin{equation}\label{eq:truncWR}
   \langle h_{fin}, \exp(2\pi i m\cdot /a)\bd T_{nM}g_{fin} \rangle_{\CL}
   = \langle h, \bd M_{ma^{-1}} \bd T_{nM}g \rangle_{\ell^2} 
  \end{equation}
  holds for all $m = 0,\ldots,a-1$ and $n_0\leq n\leq n_1$, proving (ii)$\Rightarrow$(i). Note that translation on the left side of the equation is circular and $L > L_g + L_h$ guarantees that the sums defining $g_{fin},\ h_{fin}$ possess only a single nonzero term each. To prove (i)$\Rightarrow$(ii), observe that $L > L_g + L_h$ implies $\langle h_{fin}, \exp(2\pi i m\cdot /a)\bd T_{nM}g_{fin} \rangle_{\CL} = 0$ for $n_1< n < L/M - n_0$.
  For the final part, assume for now that $I_g,\ I_h$ are centered around $0$. 
  If $L \geq L_g,L_h$ and $a,M\mid L$, then the sums defining $g_{fin},\ h_{fin}$ possess only a single nonzero term each still and
 \begin{align*}
   \lefteqn{ \delta[n]\delta[m] = \left\langle h_{fin},g_{fin}[\cdot-nM]e^{2\pi im\cdot/a}\right\rangle_{\CL} }\\
   & = \begin{cases}
         \left\langle h,\bd M_{m/a} \left(\bd T_{nM} \, g + \bd T_{(nM+L)}\, g\right)\right\rangle_{\ell^2} & \text{ if } nM \leq L/2, \\
         \left\langle h,\bd M_{m/a}  \left(\bd T_{nM} \, g + \bd T_{(nM-L)}\, g\right)\right\rangle_{\ell^2} & \text{ if } nM > L/2.
       \end{cases}
\end{align*}
If either (or both) of $I_g,\ I_h$ are not centered at $0$, the result is obtained by a suitable index shift in $n$. 
\end{proof}

Note that the condition $\langle h, g \rangle_{\ell^2}=a/M$ can be easily fulfilled be multiplying $h$ with a scalar factor, as long as $g$ and $h$ are not orthogonal to each other.

\textbf{Compactly supported duals by truncation:} 
In 1998,\linebreak Strohmer proposed a simple algorithm
for the computation of compactly supported dual windows when a compactly supported analysis  
window is given~\cite{st98-8}. The algorithm, which we will refer as the \emph{truncation method},
requires no additional restrictions to the analysis system, similar to our own approach. 
The truncation method is based on the fact that a support constraint 
on the dual window is equivalent to deleting the corresponding columns in the
WR matrix~\eqref{eq_wexler_raz} and computing only 
the values that are possibly nonzero.
The resulting equation system is then solved by computing the 
pseudoinverse, obtaining the least-squares solution. While the resulting
windows satisfy the duality conditions, they are not very smooth and
indeed show some discontinuity-like behavior, see Figure~\ref{fig:Experiments-FIR}(e,f).
One of the goals of this contribution is the improvement of these
undesirable effects.

Strohmer's method is not restricted to support constraints, but can be adjusted for the direct computation of 
a dual window $h\in\mathcal{C}_{dual}$ that minimizes $\|\bd{R}h\|_2$, for some Hermitian positive definite matrix $\bd{R}$. 
However,~\cite{st98-8} does not explore this possibility beyond the proposition of weighted $\ell^2$-norm optimization and
the method remains more restrictive than a general convex optimization formulation.

 
\section{Design of optimal dual windows}\label{sec:design}
  For a given Gabor frame $\mathcal{G}(g,a,M)$ the construction of a suitable Gabor dual window supported on an interval $I_h$ can be accomplished by solving 
  \begin{equation}\label{eq:minimizedualsupp}
    \mathop{\operatorname{arg~min}}\limits _{x \in \mathcal{C}_{\text{dual}}\cap \mathcal{C}_{\text{supp}}}\sum_{i=1}^{K}f_{i}(x),
  \end{equation}
  where $\mathcal C_{\text{dual}}$ is the set of dual windows, $\mathcal C_{\text{supp}}$ the set of all functions in $\ell_2(\ZZ)$ supported on $I_h$ and $f_i$ are priors that promote certain features in the solution. 
  In practice, each $f_i$ is weighted by a regularization parameter $\lambda_i>0$ for tuning the quantitative relations between the priors. Moreover, we only consider real-valued, symmetric windows $g$ and real-valued solutions. Those two supplementary constraints are not mandatory for optimization. However, they are used in most applications. For real-valued $g$, the canonical dual window is guaranteed to be real-valued, as a direct consequence of the Walnut representation of the Gabor frame operator~\cite{ch08}. Hence, the set of real-valued dual windows is a nonempty affine subspace of all dual windows.
  Note that the constraint $x\in\mathcal C_{supp}$ can be dropped if a solution only for a specific finite dimensional setup is required. However, when a dual window for $\ell_2(\ZZ)$ or independent from the signal length $L$ is desired, the finite support constraint is mandatory. Beyond this consideration however, reduced support $L_h \ll L$ is often desired to improve the efficiency and to shorten the processing delay. Another minor decrease in complexity can be obtained by reducing the number of frequency channels $M$, while keeping the redundancy $M/a$ fixed, thus favoring non-painless configurations.
  
  The process of selecting and tuning the priors $f_i$ is very flexible and therefore heavily dependent on the intended application, which is reminiscent of the situation for the search of the \emph{optimal} window. Here, we will mainly investigate the optimization of several classical measures of time, frequency and TF concentration. This problem is of particular importance, since joint TF concentration (or equivalently TF smoothness) is crucial for the minimization of artifacts, when performing local modification of TF coefficients in processing application. A list of the priors we consider is provided in Table~\ref{tab:regs} and their effect is discussed in the next section.
  
\begin{Rem}
  For a solution to Equation~\eqref{eq:minimizedualsupp} to exist, obviously $\mathcal C_{dual}\cap \mathcal C_{supp} \neq \emptyset$ is required. It is known~\cite{st98-8} that the WR equations are linearly independent. The same can easily be seen for the equations describing the support set $\mathcal C_{supp}$. However, when jointly considering both equation systems, we have observed linear dependencies for nonrandom analysis windows. Linear dependencies can theoretically lead to unsolvable systems or additional degrees of freedom. However, in practice, we have only observed the latter and controlling the number of equations is usually sufficient for ensuring solvability, but might not be optimal in the sense of minimality. An investigation of this issue is planned for a later contribution.
\end{Rem}

Simulations were performed using the LTFAT~\cite{ltfatnote015} and
the UNLocBoX Matlab toolbox~\cite{perraudin2014unlocbox}. A reproducible research addendum with additional material and MATLAB scripts that 
reproduce the presented results is available in \url{https://lts2.epfl.ch/rrp/gdwuco/}. We refer to this address as the \emph{webpage}.

\subsection{Functionals and proximal operators} \label{sec:prox}
In order to tune the solution of a convex optimization problem~\eqref{eq:minimizedualsupp} towards the properties we desire,
we have to select \emph{priors} $f_i$ that promote these properties. In this contribution, we mostly consider priors that are fairly standard in optimization, or simple extensions of such priors. Their various effects are quite well known, e.g. $\ell^1$ optimization favors solutions with a few large values while an $\ell^2$ prior favors a more even spread of the energy, but considerable limitations are imposed by the duality constraint. Although the set of Gabor dual windows is characterized by the WR equations, their implications in terms of window shape, localization, decay etc. remain largely unexplored. Therefore a short discussion of relevant priors, expected effects and their actual effect in our context seems worthwhile.

All the examples provided in this section were computed with an Itersine\footnote{The Itersine window $g(t) = \sin(0.5\pi\cos(\pi t)^2)\chi_{[-1/2,1/2]}$, where $\chi_{I}$ is the characteristic function of $I$, is designed to form a tight frame in the painless case with half overlap. This is equivalent to the sum of the squared modulus of the translated windows summing to a constant, a property that is retained for any appropriately sampled version.} analysis window with $L_g=60$ $L=240$, $a=15$ and $M=120$, without support constraints. This setup, in particular its high redundancy, allows us to shape the dual windows rather freely for different objective functions, therefore producing characteristic examples. The window is shown in Figure~\ref{fig:demo_iter}. The canonical dual (not depicted), equals the window up to scaling.

\begin{figure}[ht!]
\begin{center}
\includegraphics[height=0.136\textwidth,width=0.23\textwidth]{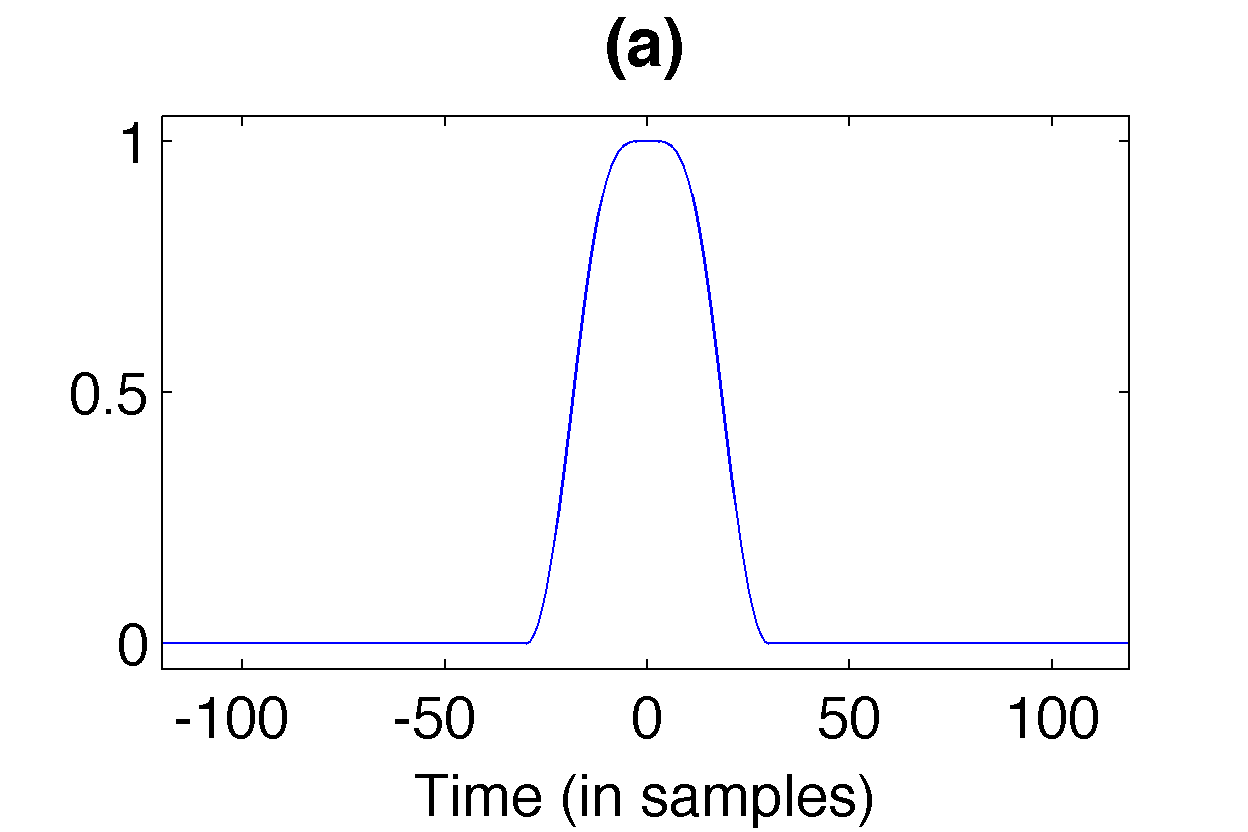}
\includegraphics[height=0.136\textwidth,width=0.23\textwidth]{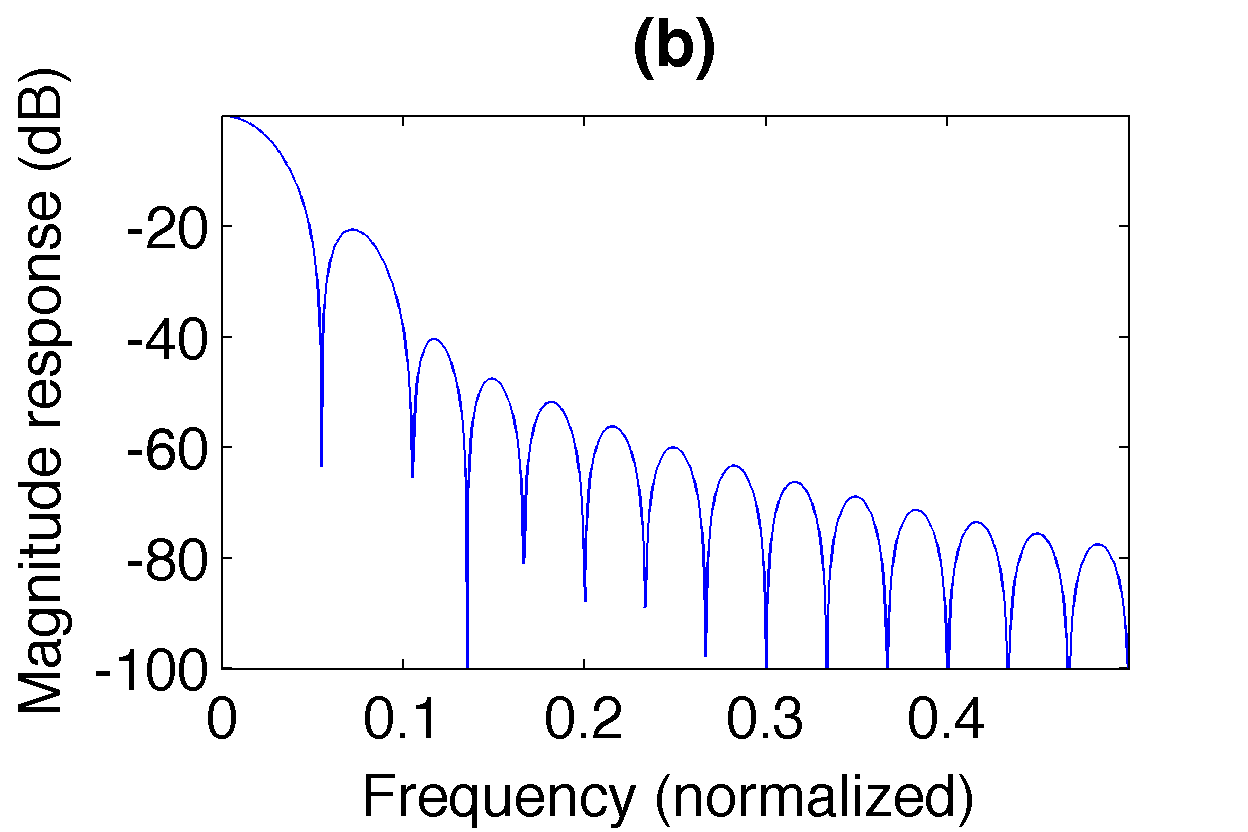}
\end{center}
\caption{\label{fig:demo_iter} The Itersine window and its magnitude frequency response (in dB).}
\end{figure}

$\ell^1$-norm minimization is usually considered, whenever (approximate) sparsity, i.e. a small number of (significant) non-zero values is desired. As the convex relaxation of the $\ell^0$ minimization problem, it is equivalent or at least close to sparsity optimization under certain conditions~\cite{li2012sparse,chen1998atomic}. 
In general, these conditions are not satisfied by Equation~\eqref{eq:minimizedualsupp}. 
Nevertheless, the restrictions imposed by the WR equations usually allow a solution with few large values. Such solutions are favored by the $\ell^1$-norm prior. However, small $\ell^1$-norm alone does not imply clustering of the large values, i.e. a solution supported on a short interval. With the window $g$ concentrated around zero, it is expected however, that any dual window necessarily has non-negligible values around zero. Hence, the optimization of $\ell^1$-norm while enforcing the duality constraint provides a localized dual, see Figure~\ref{fig:prior_l1} (a)(b). In the presented experiment, the $\ell^1$ solution possesses only $15$ values above $-80$~dB (relative to the maximum amplitude) on an interval around $0$, only half the number of WR equations $\frac{L}{M}a=30$. However, other configurations have provided solutions with few significant values spread over a larger interval, see the webpage.

The proximity operator of the $\ell^1$ prior is computed by soft-thresholding:
$$\text{soft}_\mu(y)=\text{sgn}(y) \left(|y|-\mu \right)_+$$ 
where $(\cdot)_+=\max(\cdot,0 )$. For compactly supported dual windows, strict bandlimitation is clearly not feasible. Therefore, when applied in the Fourier domain, the $\ell^1$ prior cannot achieve a truly sparse solution, but promotes a small number of significant values. In many cases, the result is similar to actual concentration measures, compare Figure~\ref{fig:prior_l1}(c)(d) and Figure~\ref{fig:demo_variance2}(a)(b).

\begin{figure}[ht!]
\begin{center}
\includegraphics[height=0.136\textwidth,width=0.23\textwidth]{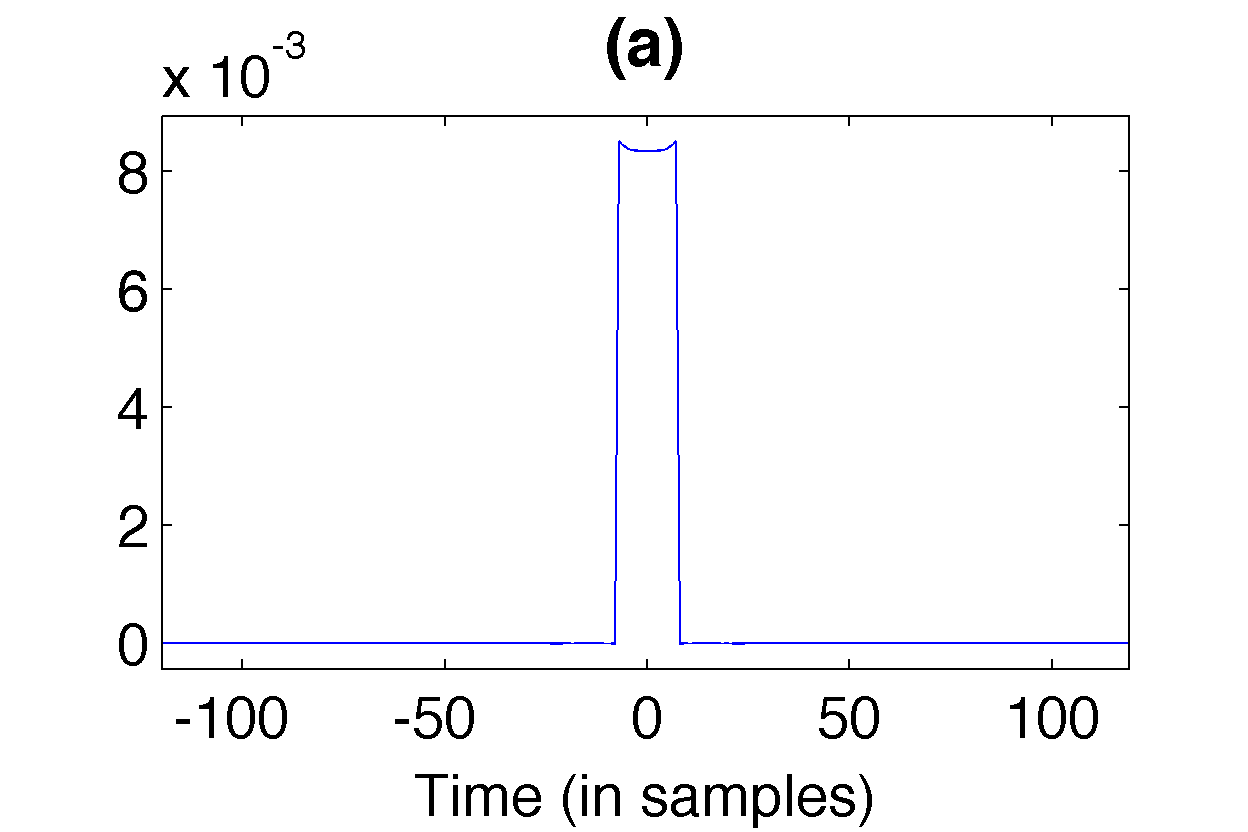}
\includegraphics[height=0.136\textwidth,width=0.23\textwidth]{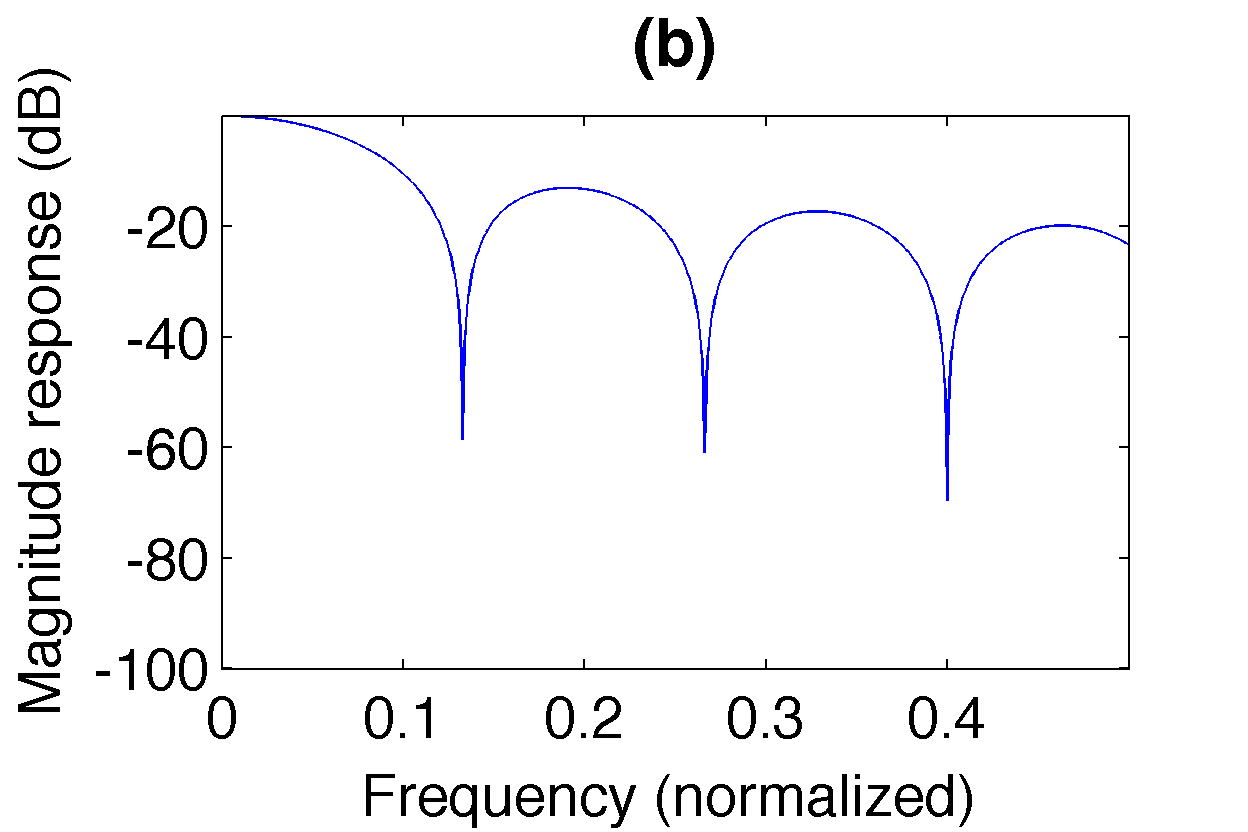}\\
\includegraphics[height=0.136\textwidth,width=0.23\textwidth]{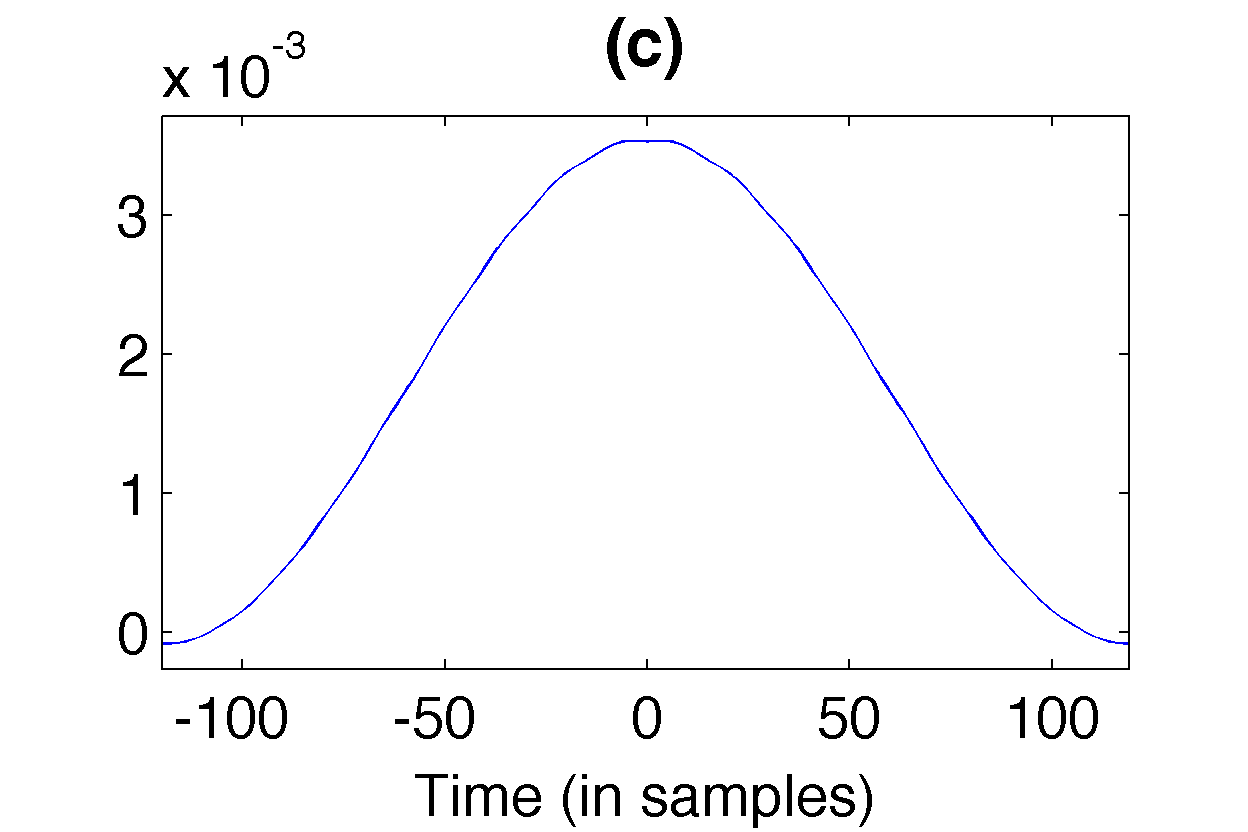}
\includegraphics[height=0.136\textwidth,width=0.23\textwidth]{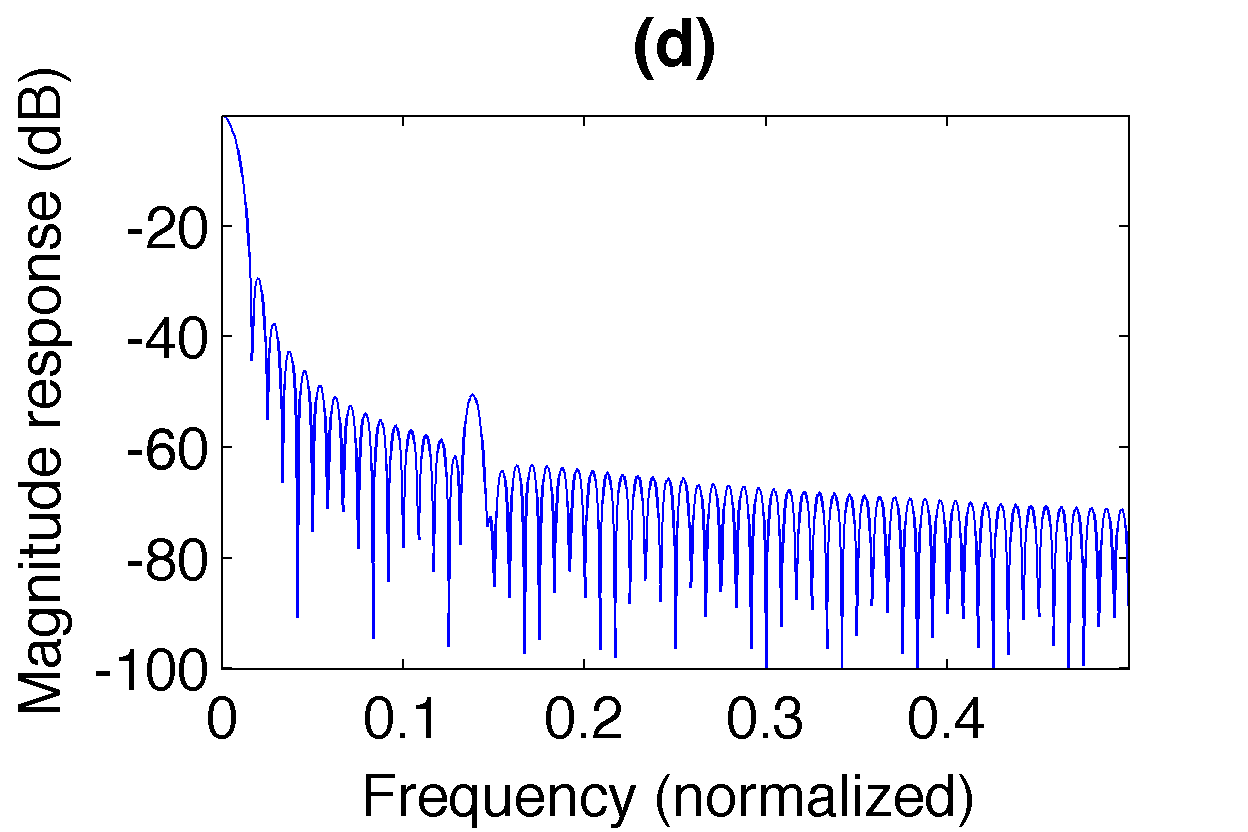}
\caption{Result of $\ell_1$ optimization. Priors: (a)(b) $\|x\|_1$ (c)(d) $\|\mathcal{F}x\|_1$}\label{fig:prior_l1}
\end{center}
\end{figure}

An $\ell^2$ prior will, in our context, always lead to the canonical dual Gabor window. In general, this prior will affect the values in a more proportional way over the whole signal range. 
It is traditionally used as a data fidelity term, i.e. the solution is expected to be close, in the $\ell^2$-norm sense, to a given estimate. 
The associated objective function is not only convex, but also smooth, admitting gradient descent approaches for minimization.\\

\emph{a) Concentration inducing functions:}
Our main objective in the following section will be the search for a Gabor dual window with optimized/modified TF concentration. Therefore we recall a number of different concentration measures.
Inspired by the famous Heisenberg inequality, the most natural way to impose localization is to optimize the variance of the signal $x\in\RR^L$ or more precisely, its modulus:
\[
\text{var}(|x|)= 1/\sqrt{L} \sum_{i=-L/2}^{L/2-1}(i-\overline{|x|})^2|x_i|
\]
with $\overline{|x|}=\sum_{i=-L/2}^{L/2-1}i|x_i|$ being the center of gravity. As we consider symmetric windows $\overline{x}=0$, we can simplify this expression to: $\text{var}(|x|)= 1/\sqrt{L} \sum_{i=-L/2}^{L/2-1}(i)^2 |x_i|$. In that case\footnote{If the center of gravity is not fixed to $0$, the variance is not a weighted $\ell_1$ norm anymore and its optimization is not straightforward. For a symmetric prototype $g$, it is reasonable however to assume the center of gravity of the dual window to coincide with $\overline{|g|}$. This expectation is confirmed by the results we obtain.}, the variance turns out to be a weighted $\ell^1$-norm with quadratic weight $w^2$, $w := \frac{1}{\sqrt{L}}\left[-L/2,\ldots,L/2-1\right]$. Compared to $\ell^1$ minimization, this prior additionally penalizes values far from the origin, inducing concentration. The proximity operator of $\text{var}(|x|)$ is a variation of the $\ell^1$ proximity operator and computed by weighted soft thresholding. 
And example is shown in Figure~\ref{fig:demo_variance}(a)(b).

We also consider the variance of the energy of the signal: $\text{var}(|x|^2)$, for symmetric windows equal to a weighted $\ell^2$ norm with linear weight $w$: $\text{var}(x^2)=\|w \cdot x\|_2^2$. Explicit computation of the proximity operator leads to 
\begin{equation} \label{eq:prox_var2}
\prox_{\gamma \text{var}(x^2) }(y) = \frac{1}{1+2\gamma w^2} y,
\end{equation}
i.e. multiplication with a function that decays quadratically away from zero, see Figure~\ref{fig:demo_variance}(c)(d).

A closely related concentration measure is smoothness in frequency, as measured by the gradient of the Fourier transform $\|\nabla \mathcal{F} x\|_2^2$. Indeed, the resulting proximity operator has almost the same form: 
\begin{equation} \label{eq:prox_grad_fourier}
\prox_{\gamma \|\nabla \mathcal{F} x\|_2^2 }(y) = \frac{1}{1+2\gamma \psi} y
\end{equation}
with $\psi[l] = 2-2 \cos \left( \frac{2\pi l}{L} \right)$. Since $\psi[l] \approx C l^2$ for small $l$ and values away from $0$ are strongly attenuated, the priors $\text{var}(|x|^2)$ and $\|\nabla \mathcal{F} x\|_2^2$ often lead to similar results. Both functions induce concentration by attenuation of values far from the origin. Examples are shown in Figure~\ref{fig:demo_variance}(e)(f).

Concentration in frequency is easily achieved through $\text{var}(|\mathcal{F}x|)$, $\text{var}(|\mathcal{F}x|^2)$ or $\|\nabla x\|_2^2$. The respective proximity operators are obtained simply by conjugating the proximity operators discussed above with the (inverse) Fourier transform. For examples, see Figure~\ref{fig:demo_variance2}.

\begin{figure}[ht!]
\begin{center}
\includegraphics[height=0.136\textwidth,width=0.23\textwidth]{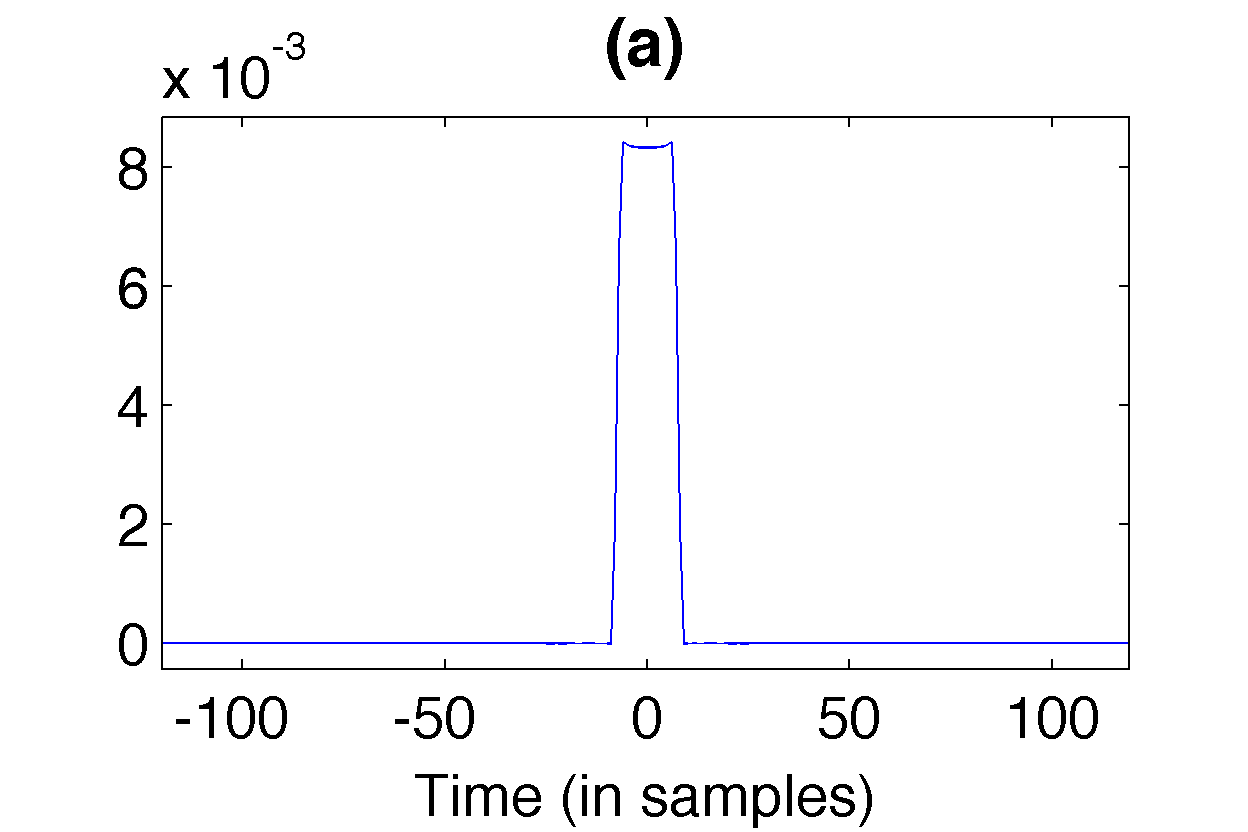}
\includegraphics[height=0.136\textwidth,width=0.23\textwidth]{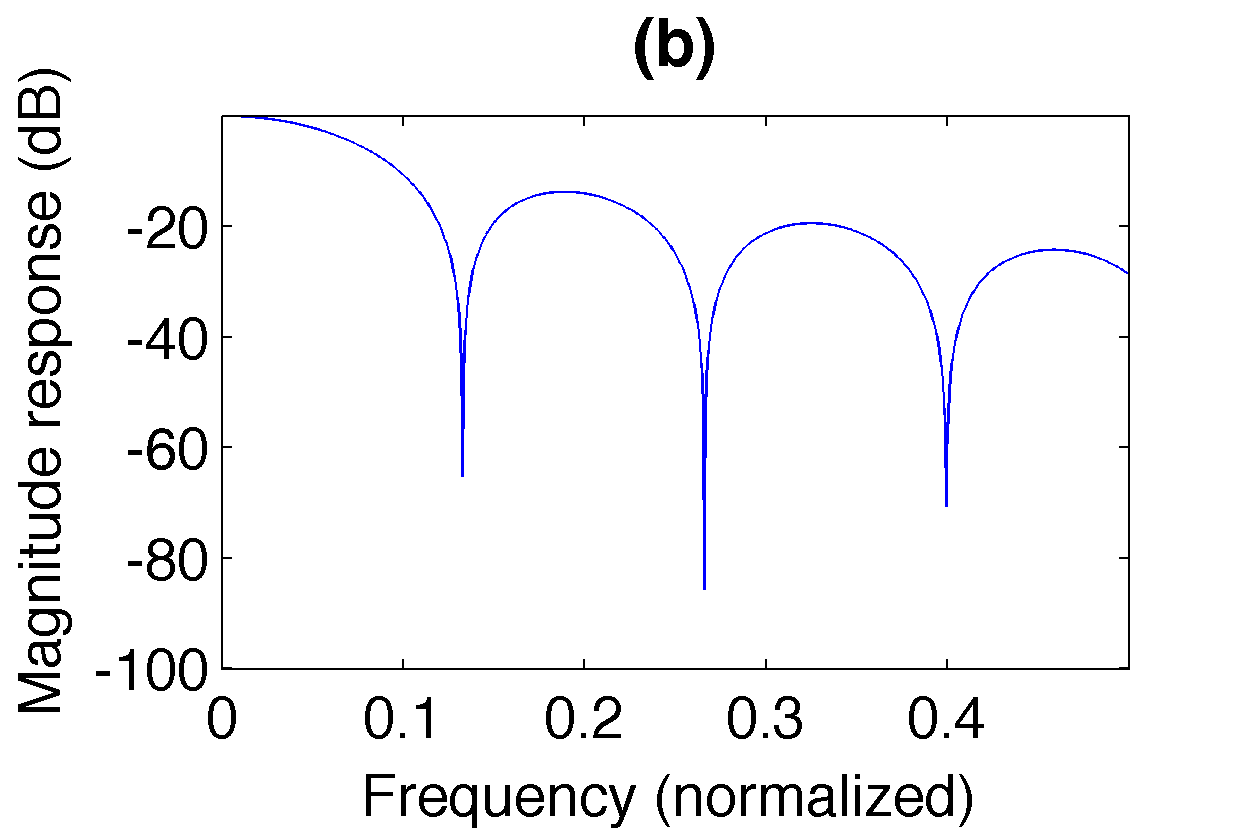}\\
\includegraphics[height=0.136\textwidth,width=0.23\textwidth]{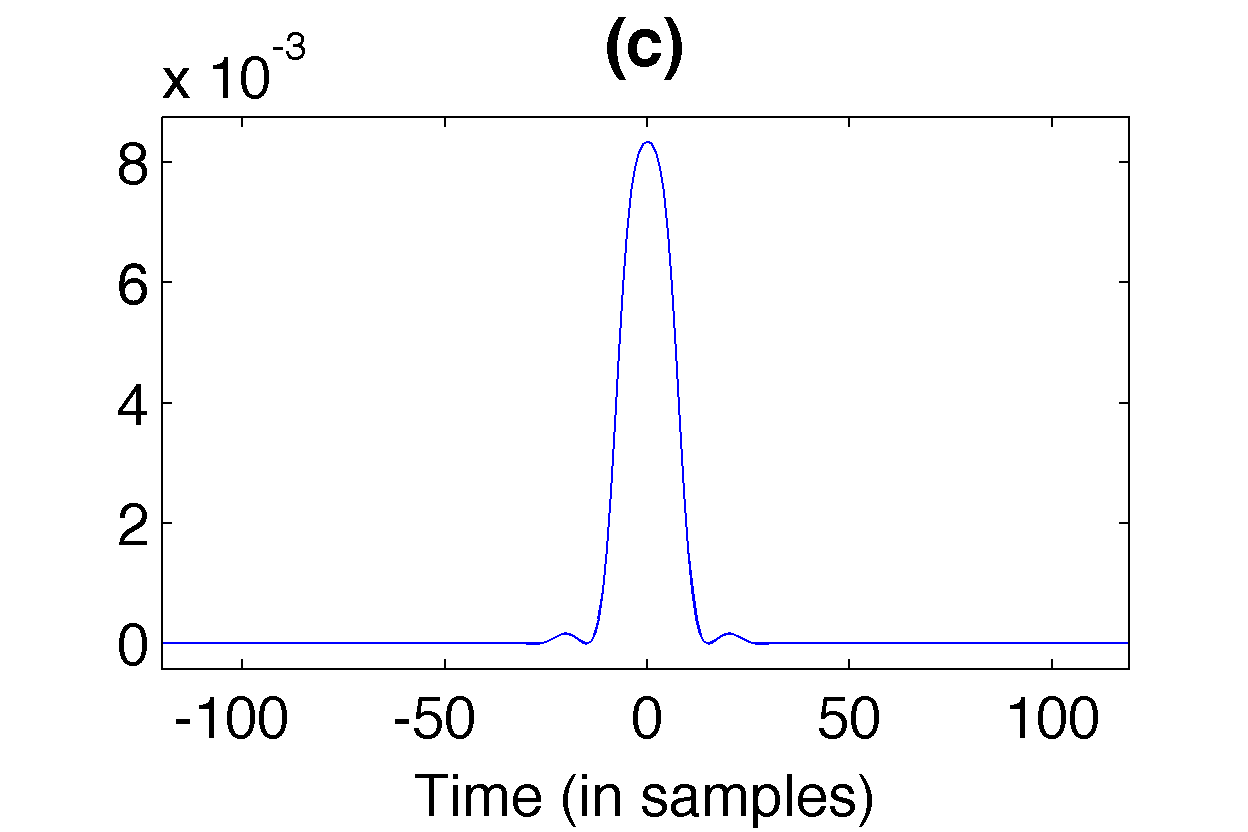}
\includegraphics[height=0.136\textwidth,width=0.23\textwidth]{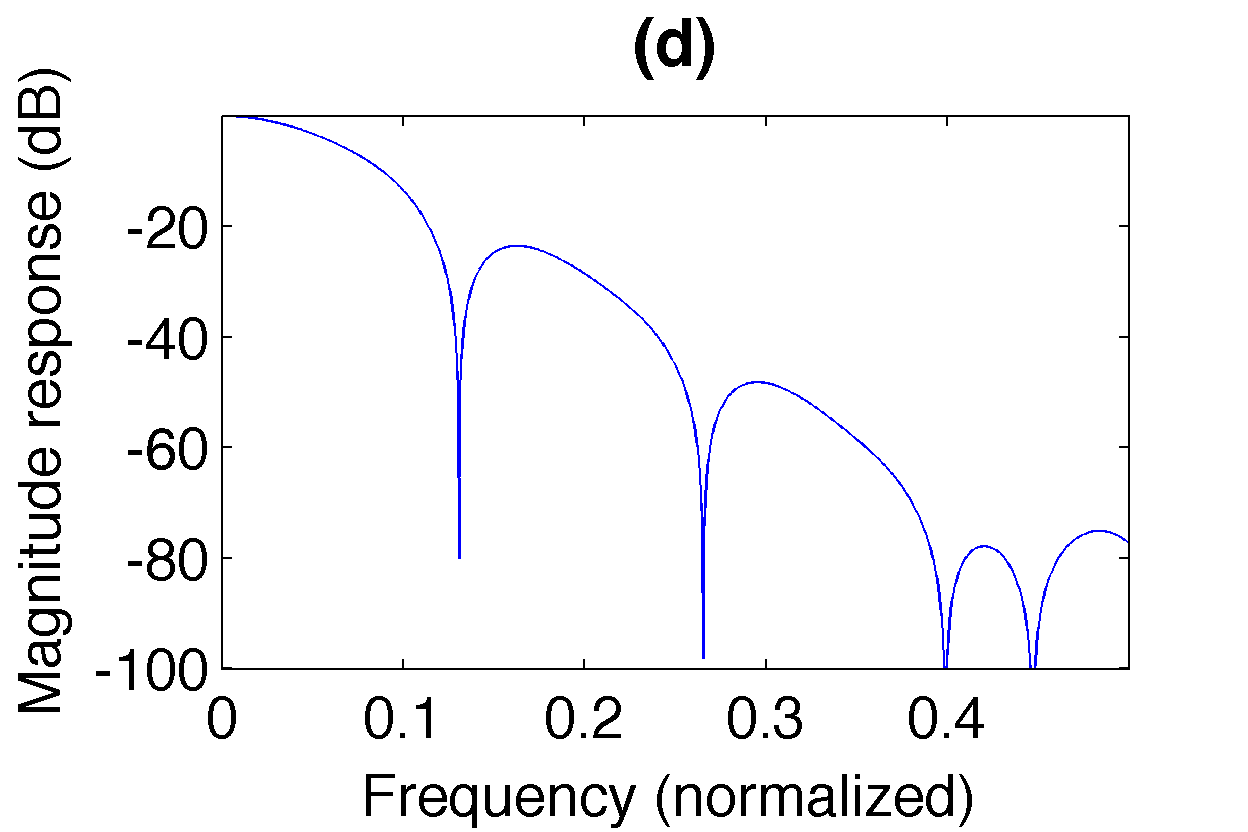}\\
\includegraphics[height=0.136\textwidth,width=0.23\textwidth]{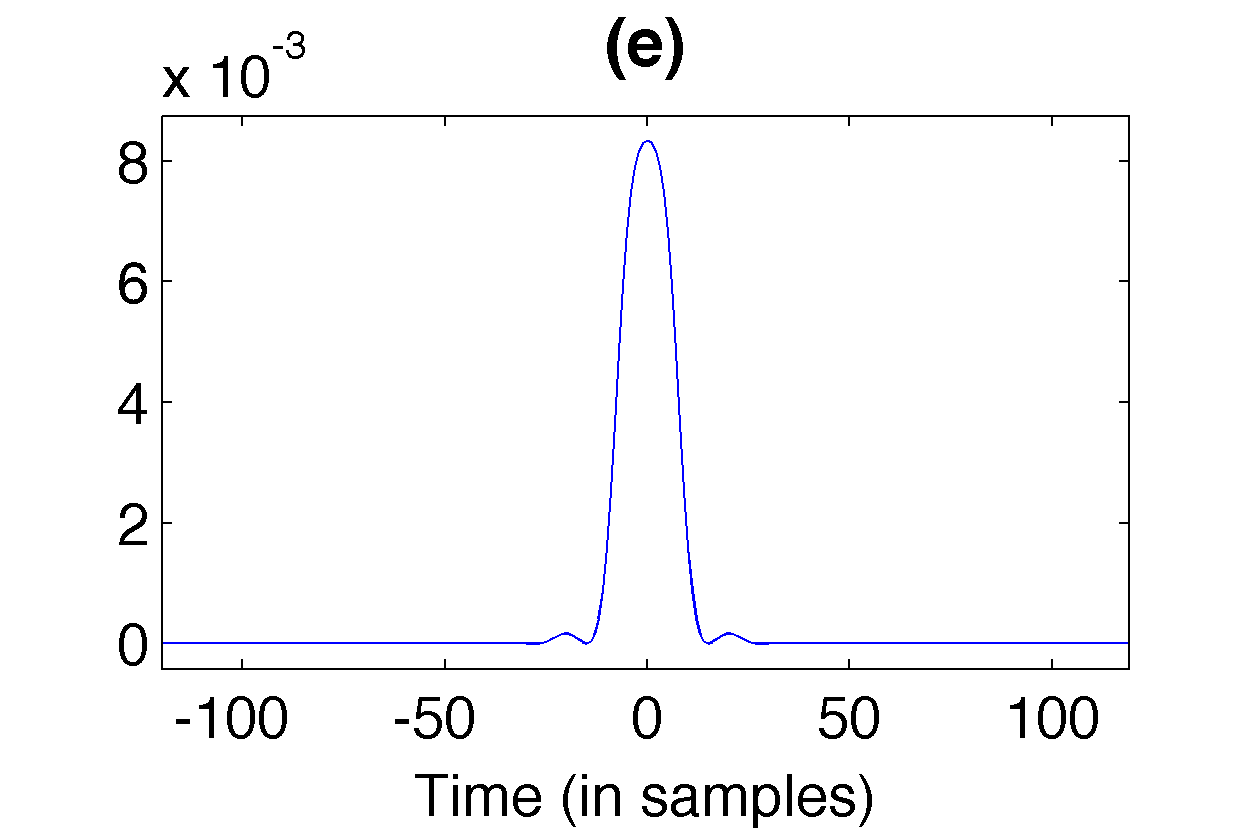}
\includegraphics[height=0.136\textwidth,width=0.23\textwidth]{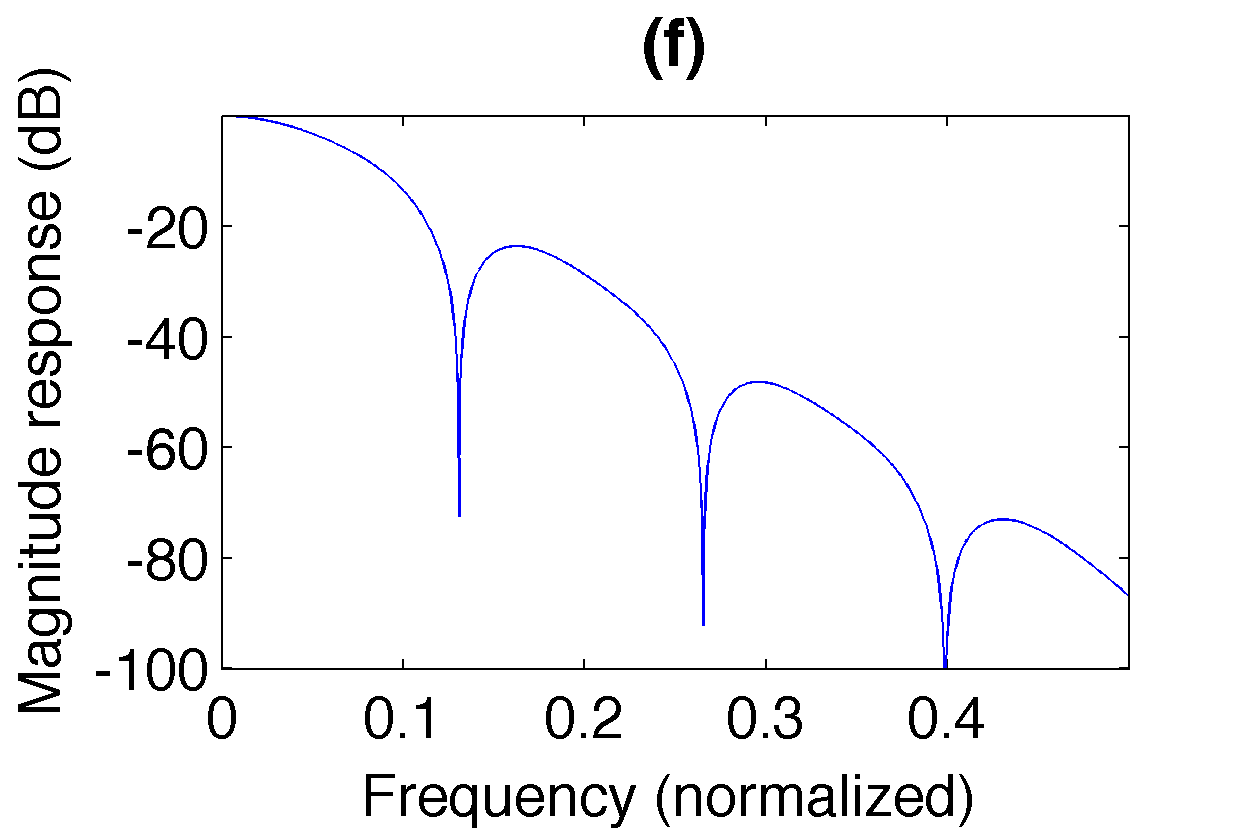} 
\end{center}
\caption{\label{fig:demo_variance} 'Time-optimized' dual windows and their magnitude frequency response (in dB). Priors: (a)(b) $ \text{var}(|x|) $. 
(c)(d) $ \text{var}(|x|^2)$. (e)(f) $\|\nabla \mathcal{F}x\|_2^2$}
\end{figure}

\begin{figure}[ht!]
\begin{center}
\includegraphics[height=0.136\textwidth,width=0.23\textwidth]{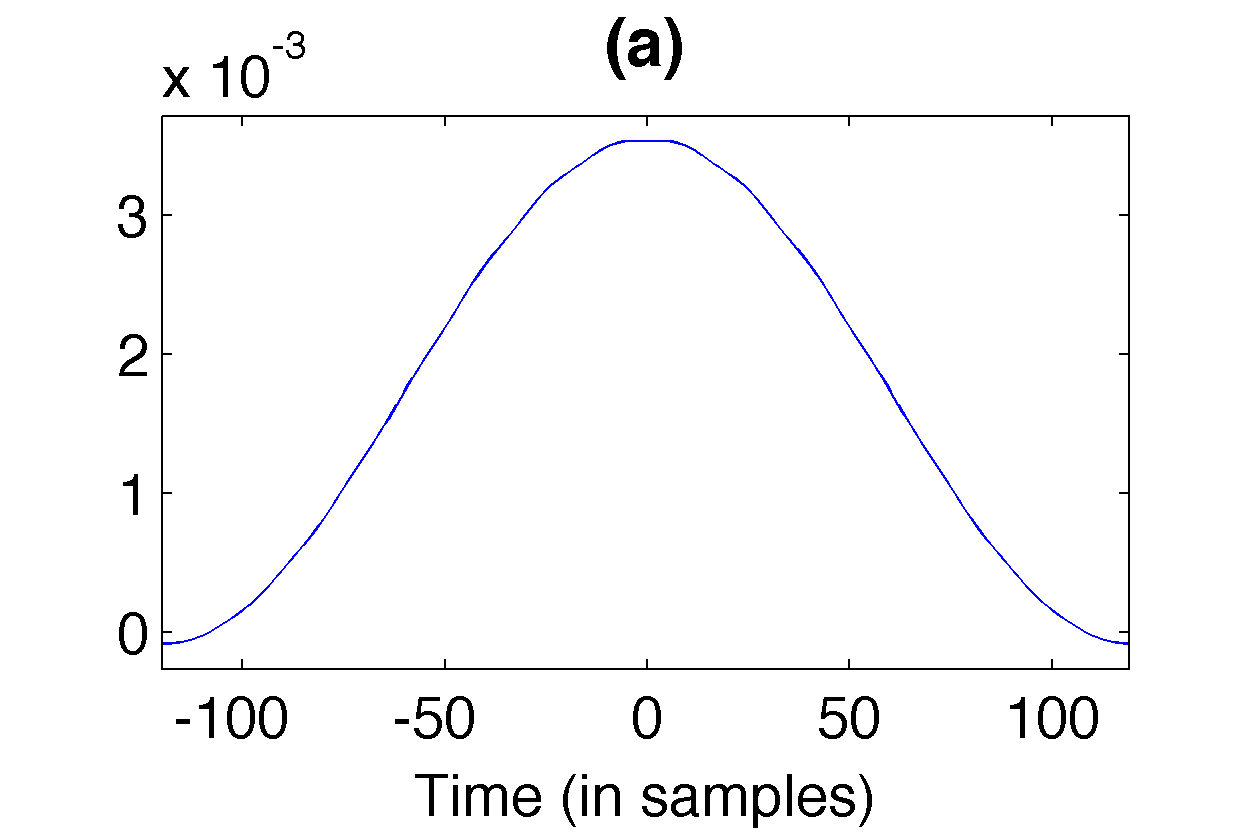}
\includegraphics[height=0.136\textwidth,width=0.23\textwidth]{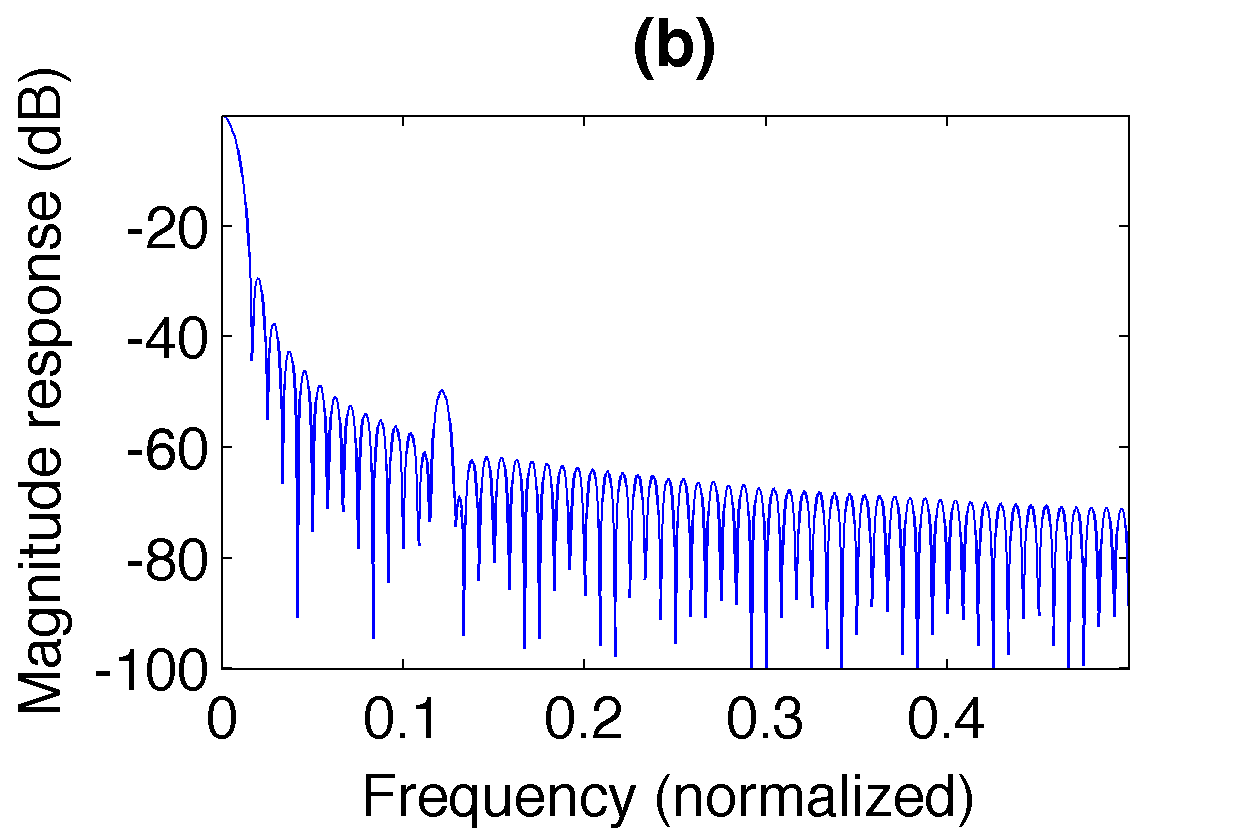}\\
 \includegraphics[height=0.136\textwidth,width=0.23\textwidth]{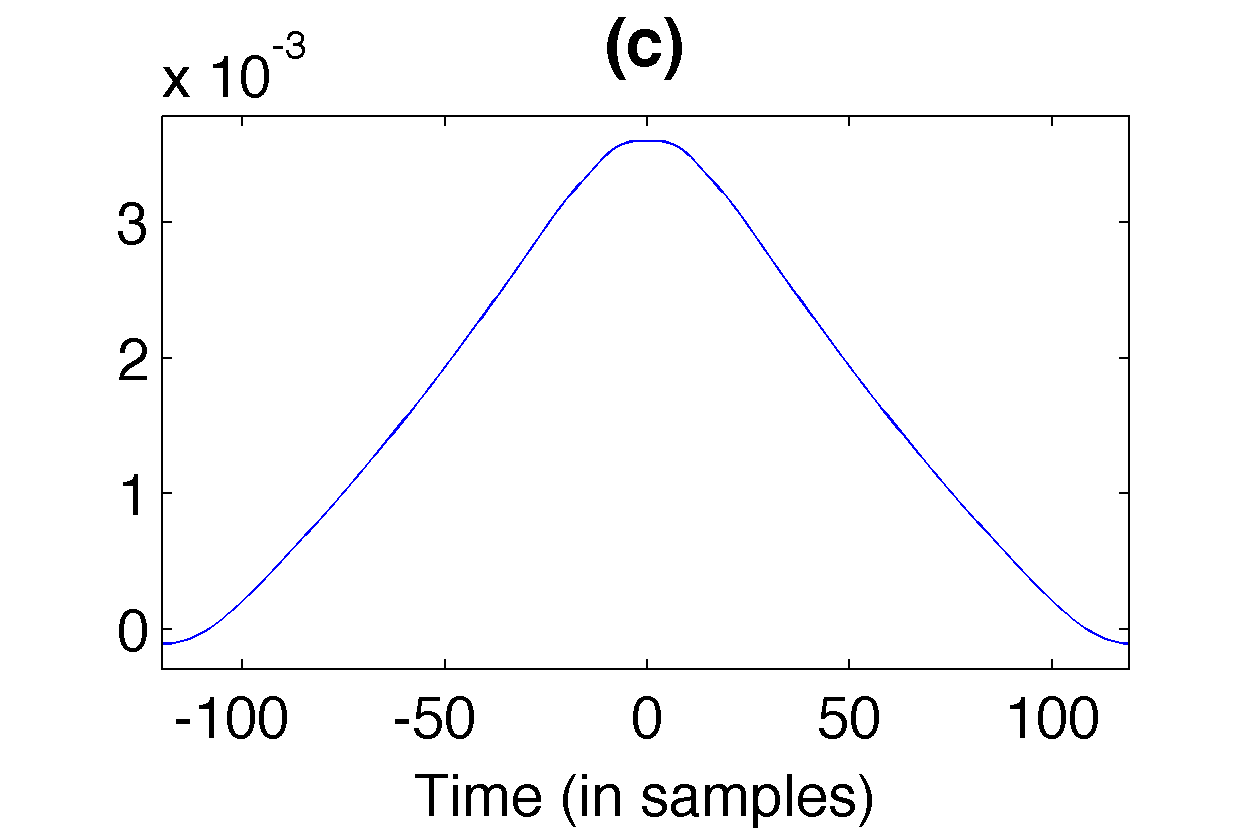}
 \includegraphics[height=0.136\textwidth,width=0.23\textwidth]{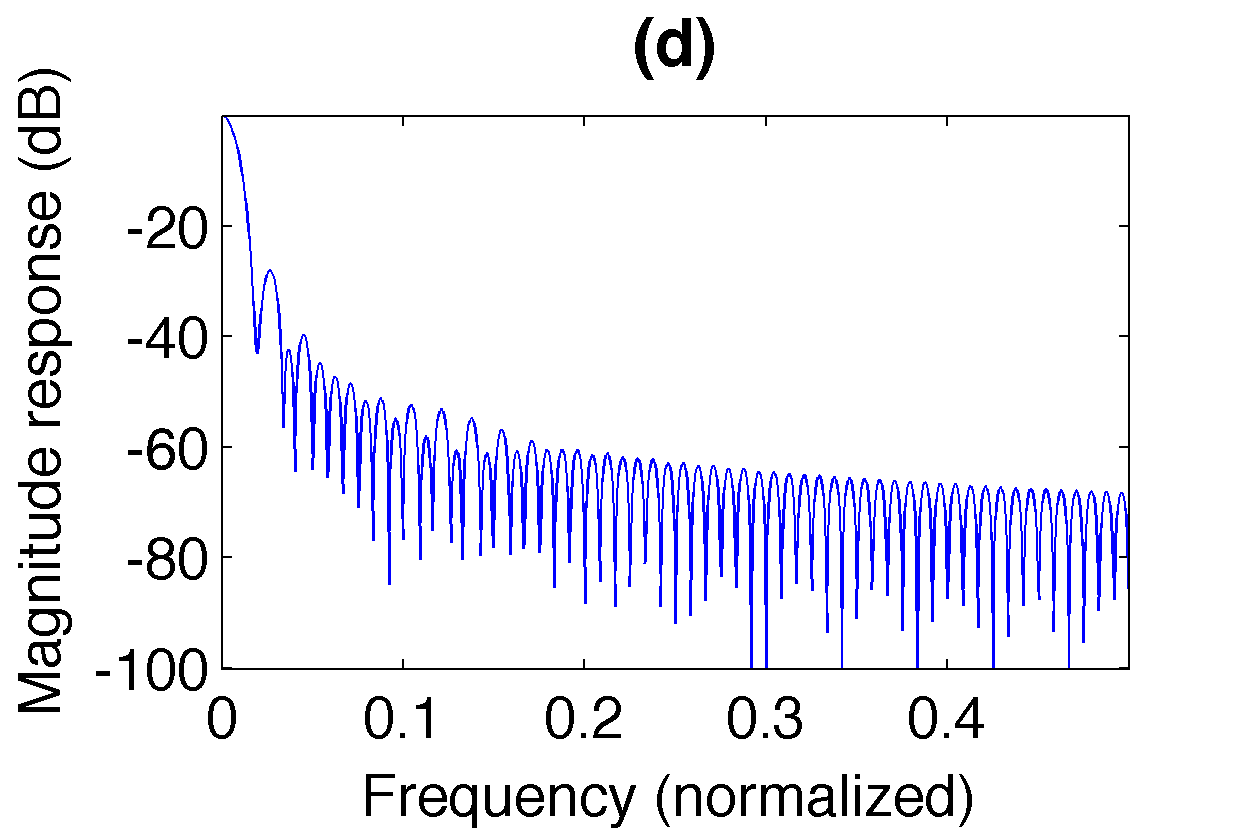}\\
 \includegraphics[height=0.136\textwidth,width=0.23\textwidth]{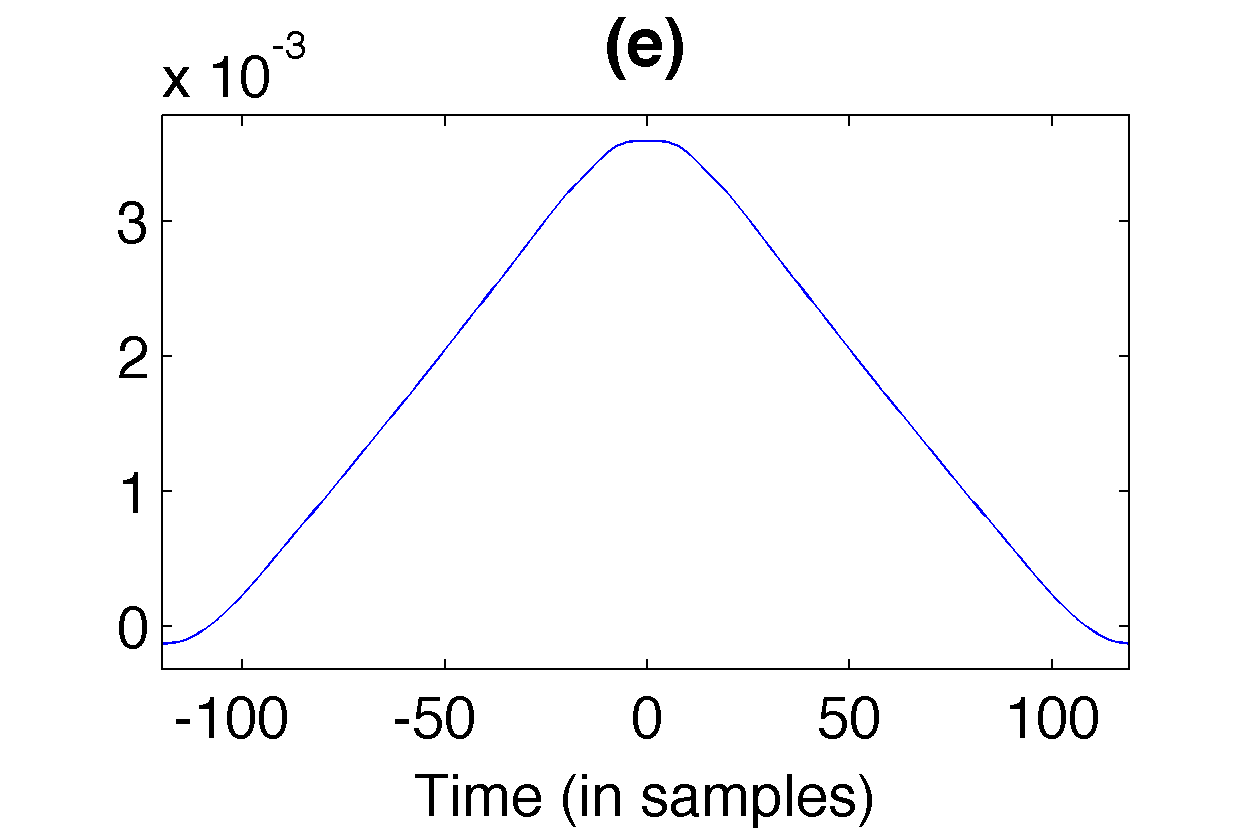}
 \includegraphics[height=0.136\textwidth,width=0.23\textwidth]{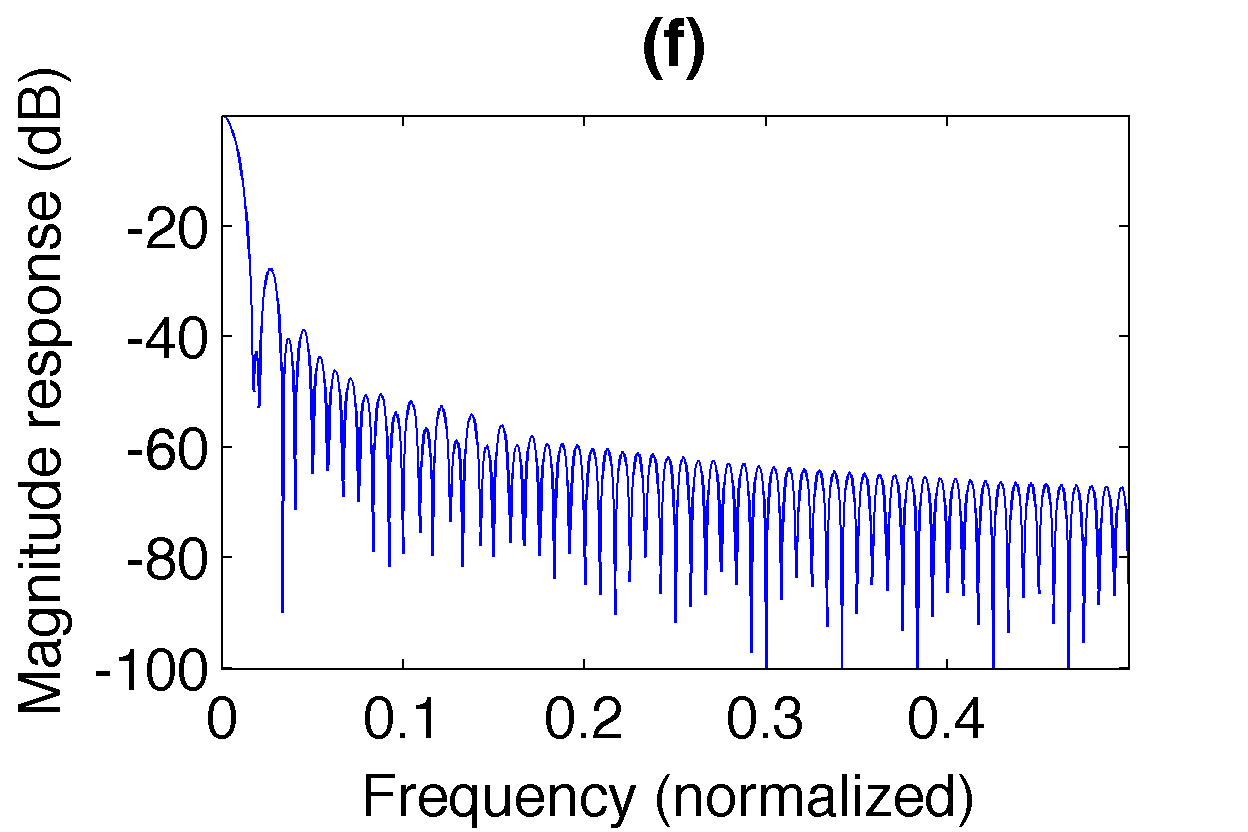} 
\end{center}
\caption{\label{fig:demo_variance2} 'Frequency-optimized' dual windows and their magnitude frequency response (in dB). Priors: (a)(b) $ \text{var}(|\mathcal{F}x|) $. 
(c)(d) $ \text{var}(|\mathcal{F}x|^2)$. (e)(f) $\|\nabla x\|_2^2$}
\end{figure}

\emph{b) Concentration in time and frequency:}
For simultaneous concentration in time and frequency, we can consider jointly the time- and frequency-domain variants of the priors discussed above. Alternatively, we use a single cost functions providing concentration in both domains at once. In TF literature, modulation space norms, i.e. $\ell^p$-norms on the short-time Fourier coefficients are frequently used to measure joint TF localization, see e.g.~\cite{gr01,fe81-2}. In particular $\|x\|_{S_0}=\| G_{g,1,L} x\|_1$, where $g$ is a Gaussian function, is considered as quality measure for window functions. 
They are two different ways to motivate this claim. First, the $S_0$-norm is an inverse measure of concentration. It is limited by an uncertainty principle demonstrated by Lieb in~\cite{lieb2002integral} and generalized to the discrete setting in~\cite[Proposition 2]{feichtinger2012method} and in~\cite[Theorem 3]{perraudin2016global}. Hence, minimizing the $S_0$-norm tends to reduce uncertainty and improve the overall concentration. 
Second, the $S_0$-norm prior is nothing but an $\ell^1$-norm prior on the time-frequency representation of $x$. Hence, similar to the classical $\ell^1$-prior, it is expected to promote functions with few large values and a significant magnitude drop-off outside of these values. However, we know that time-domain concentration implies frequency-domain smoothness and vice-versa. Therefore, we can assume that optimizing the $S_0$-norm provides a function with good joint time-frequency concentration and smoothness. Again, similar to the $\ell^1$-case, optimizing $S_0$-norm alone does not not guarantee concentration around the origin (or any single TF location). However, non-negligible values around the origin are required for duality and thus, the combination of the duality constraint and $S_0$-norm optimization provides dual windows with excellent localization around zero in both domains, see Figure~\ref{fig:prior_s0}(a)(b).

Compared to the previously discussed priors, $S_0$-norm optimization is considerably more expensive. Since we are not aware of an explicit solution to the $S_0$ proximity operator, we propose its computation via an iteration based on ADMM~\cite{boyd2011distributed}. The number of required ADMM steps per PPXA (parallel proximal algorithm) iteration is low and scales well with $L$ (usually 3-4 steps provided sufficient precision), but each substep requires the computation of one full STFT and one inverse STFT, with a complexity of $\mathcal{O}(L^2 \log(L))$ each. 

In some cases, concentration can be further increased and a desired trade-off between time- and frequency-concentration can be established by a weighted $S_0$-norm prior. The proximity operator is realized similar to the unweighted case. Figure~\ref{fig:prior_s0}(c)(d) shows an example using the circular weight
\[
W[m,n]=\ln \left( 1 + w^2[n]+w^2[m] \right),
\]
using the weight $w$ as defined above.
While other weights are clearly feasible, the weight above has been tuned to yield good results in our experiments and is also used for Exp. 1.

\begin{figure}[ht!]
\begin{center}
\includegraphics[height=0.136\textwidth,width=0.23\textwidth]{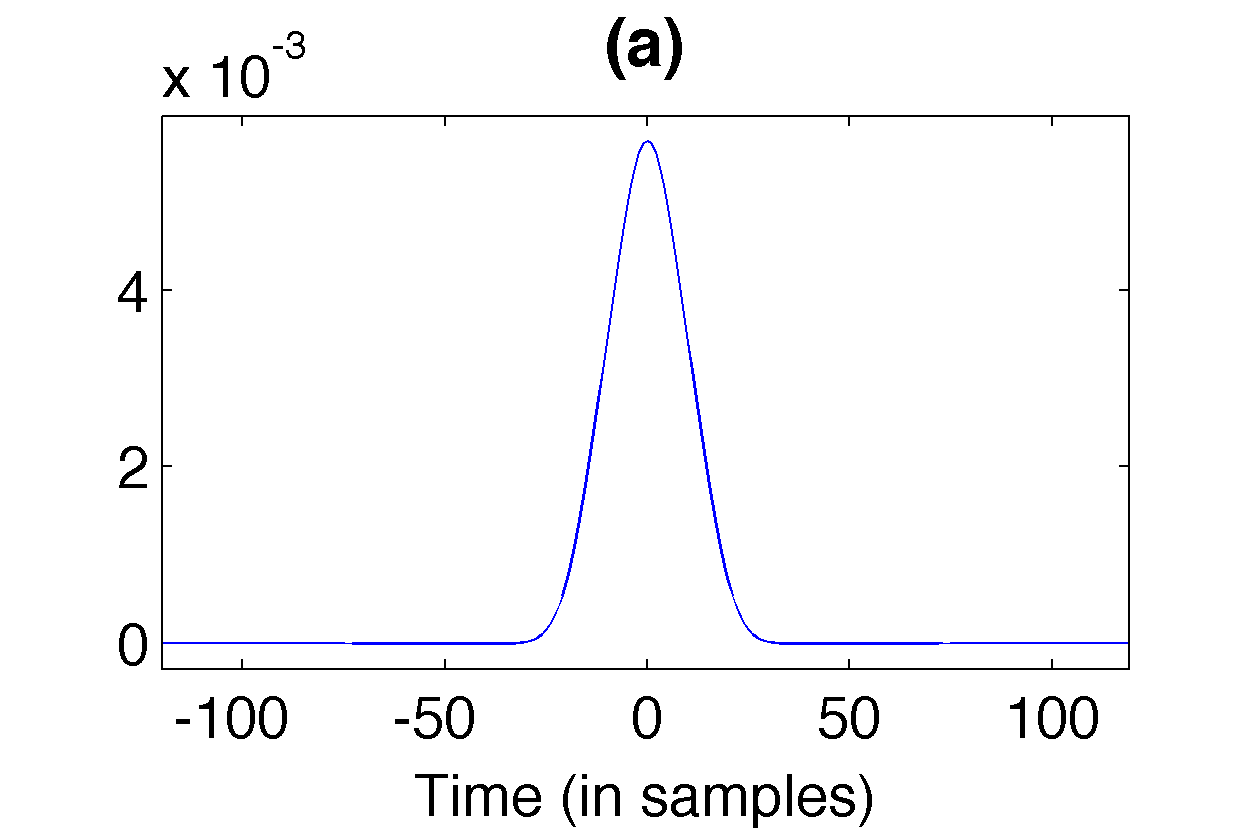}
\includegraphics[height=0.136\textwidth,width=0.23\textwidth]{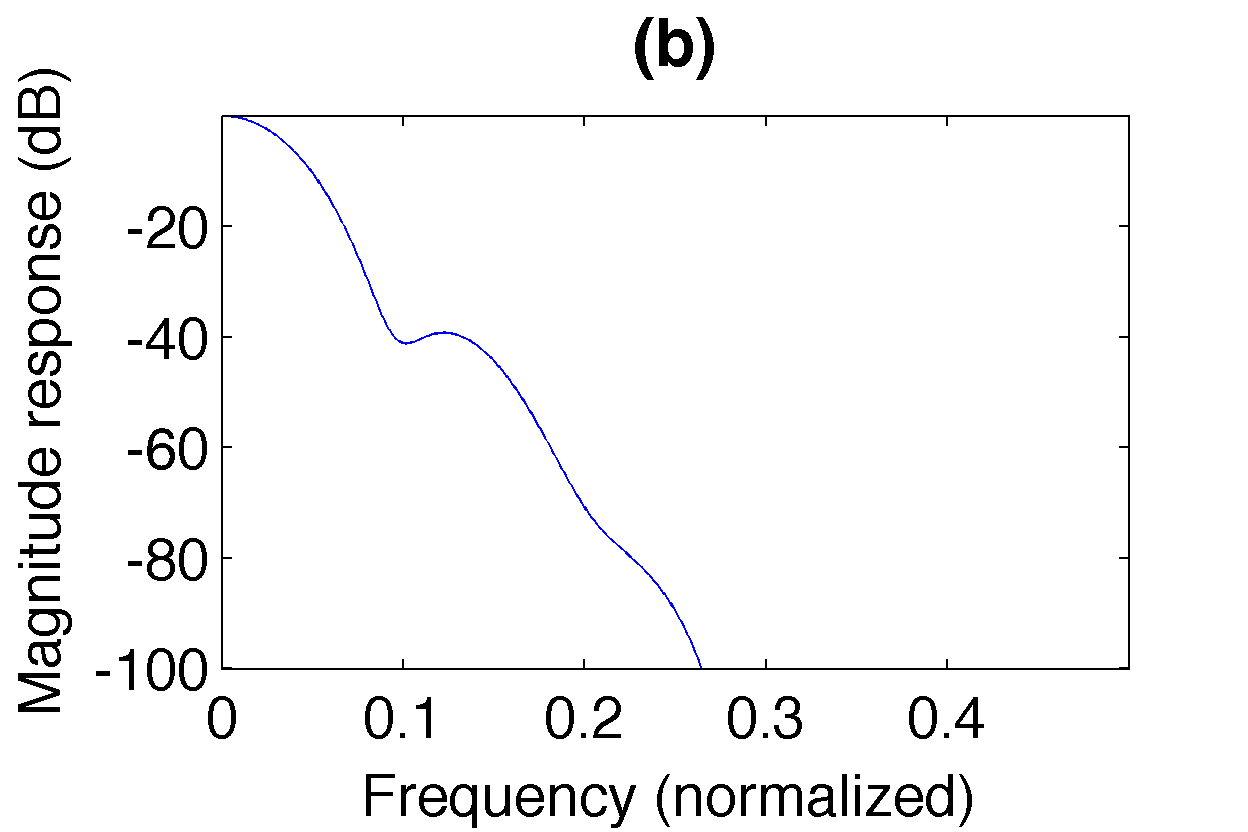}\\
\includegraphics[height=0.136\textwidth,width=0.23\textwidth]{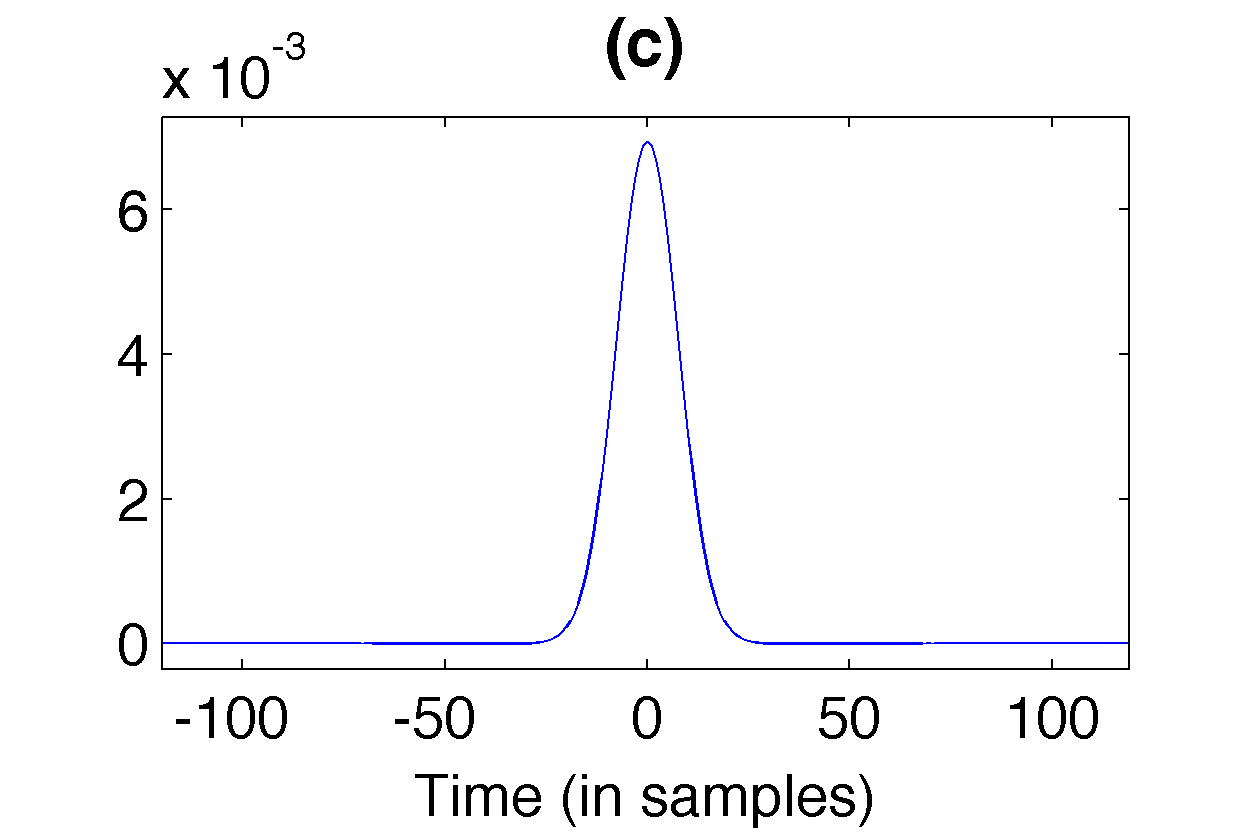}
\includegraphics[height=0.136\textwidth,width=0.23\textwidth]{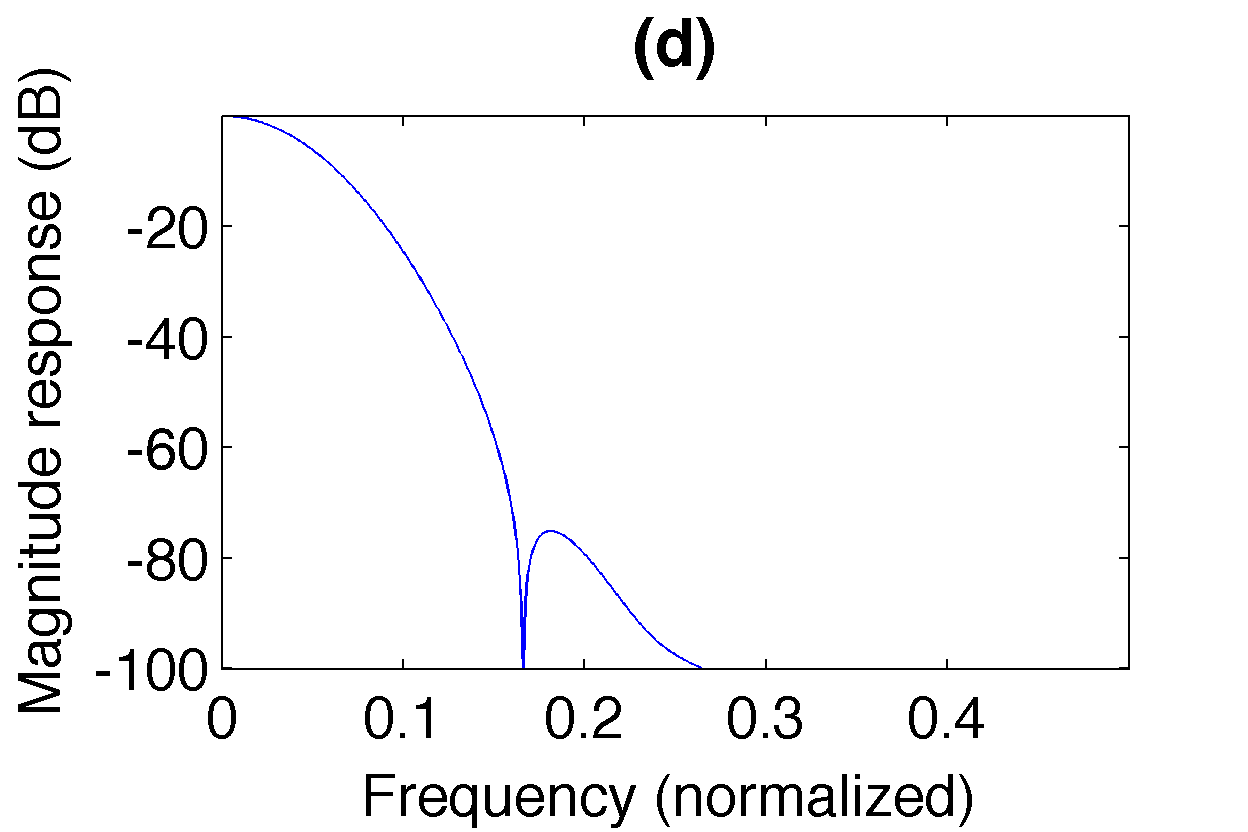}
\end{center}
\caption{Result of optimization function $\|x\|_{S_0}=\| G_{g,1,L} x\|_1$ (a)(b) and its weighted version (c)(d)}\label{fig:prior_s0}
\end{figure}

\emph{c) Other cost functions:}
The list of possible cost functions is vast and full exploration of the possibilities of convex optimization in window design is far beyond the scope of any single contribution. As a rather academic example, we propose a free design approach that selects the dual Gabor window closest to the linear span of a model window $g_{sh}$, i.e. we find 
\[
 \mathop{\operatorname{arg~min}}\limits _{x \in \mathcal{C}_{\text{dual}}}\|x - P_{\langle g_{sh}\rangle}x\|_2^2,
\]
where $\langle g_{sh}\rangle$ is the linear span of $g_{sh}$. The solution is computed by a POCS (projection onto convex set)~\cite{combettes1993foundations} algorithm. Due to the examples academic nature, we were not concerned with convergence time. Examples using a sine wave and a dirac pulse as model window are presented in Figure~\ref{fig:other_priors}(a)(b) and (c)(d)(d').

\begin{figure}[!ht]
\begin{centering}
\includegraphics[width=0.23\textwidth]{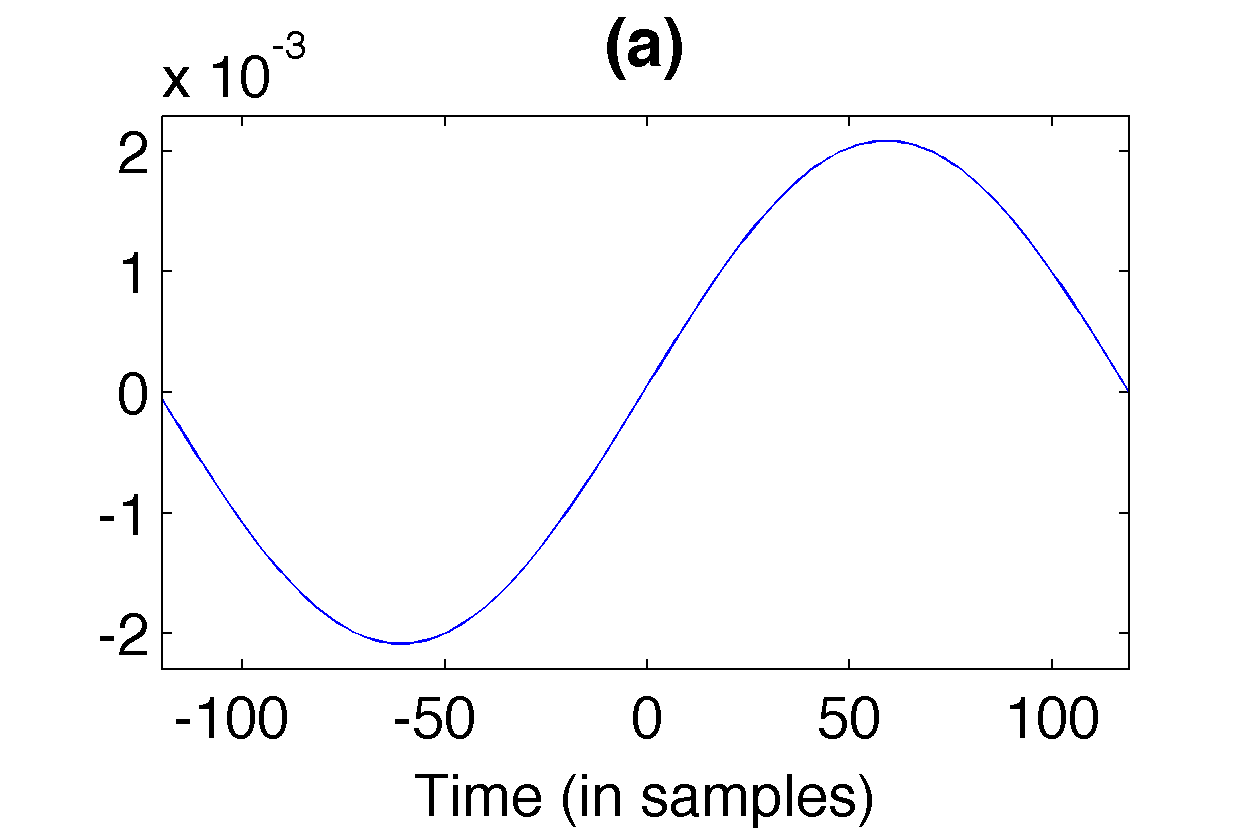}
\includegraphics[width=0.23\textwidth]{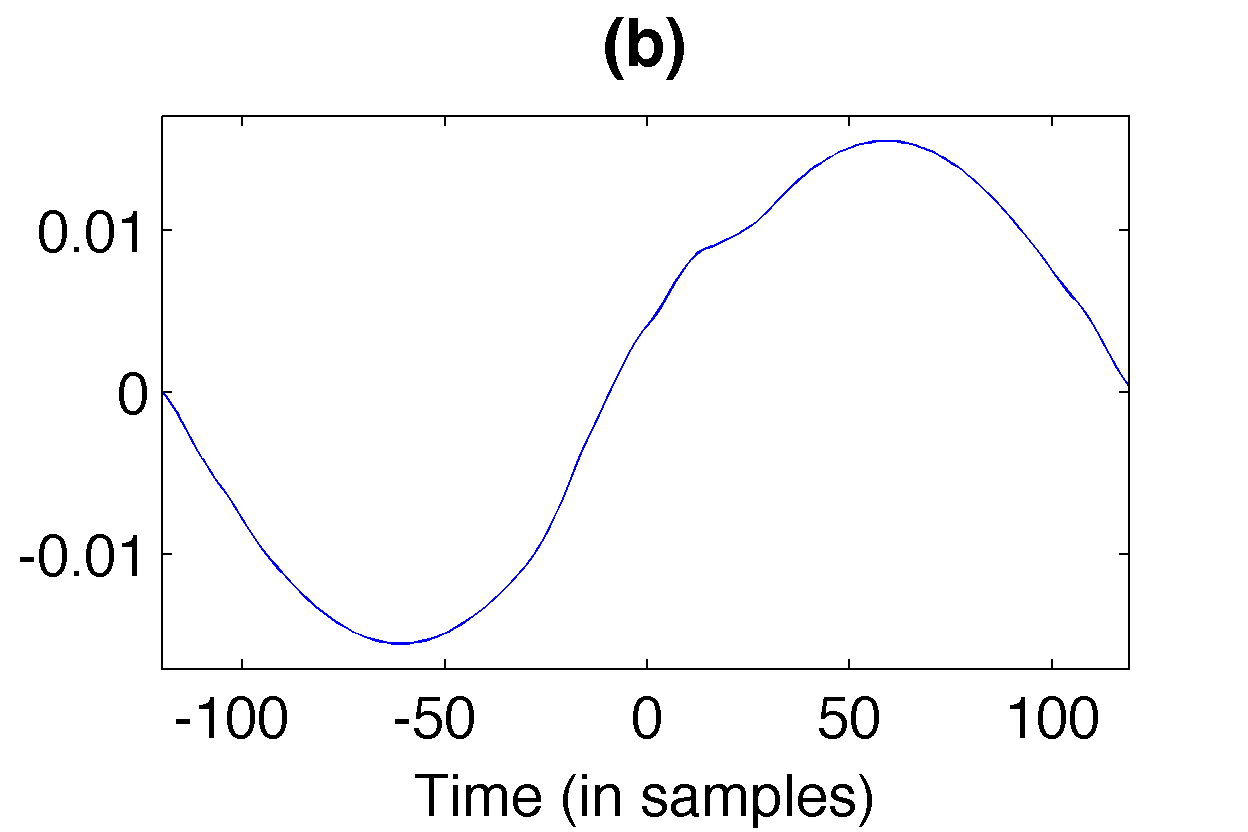}\\
\includegraphics[width=0.23\textwidth]{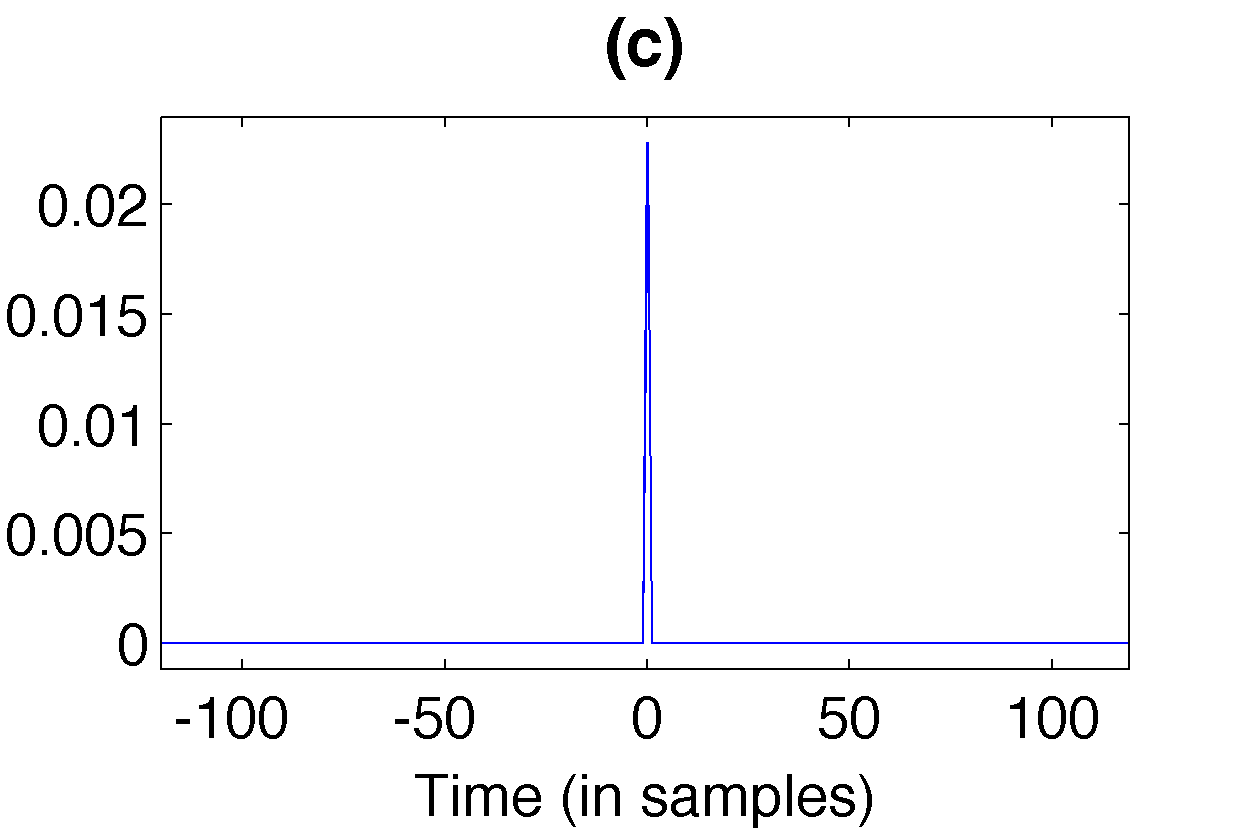}
\includegraphics[width=0.23\textwidth]{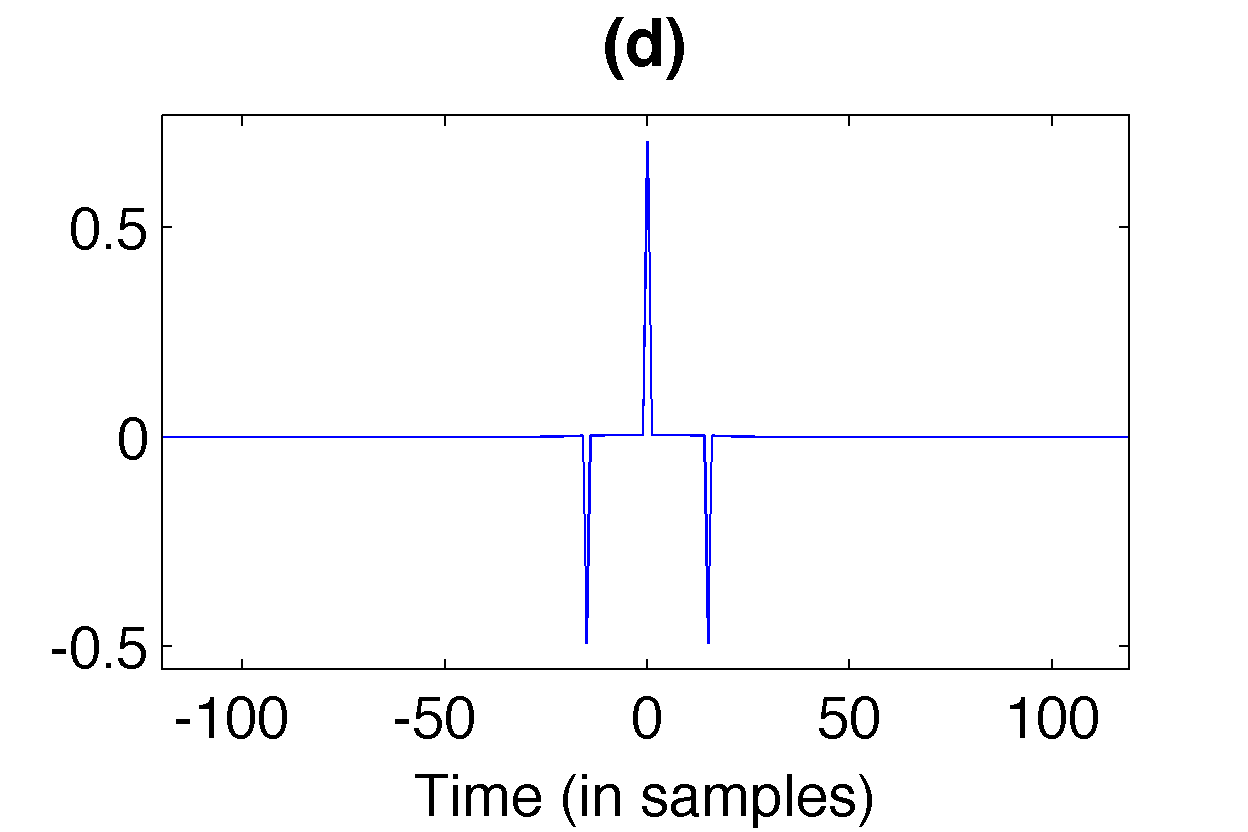} \\
\includegraphics[width=0.23\textwidth]{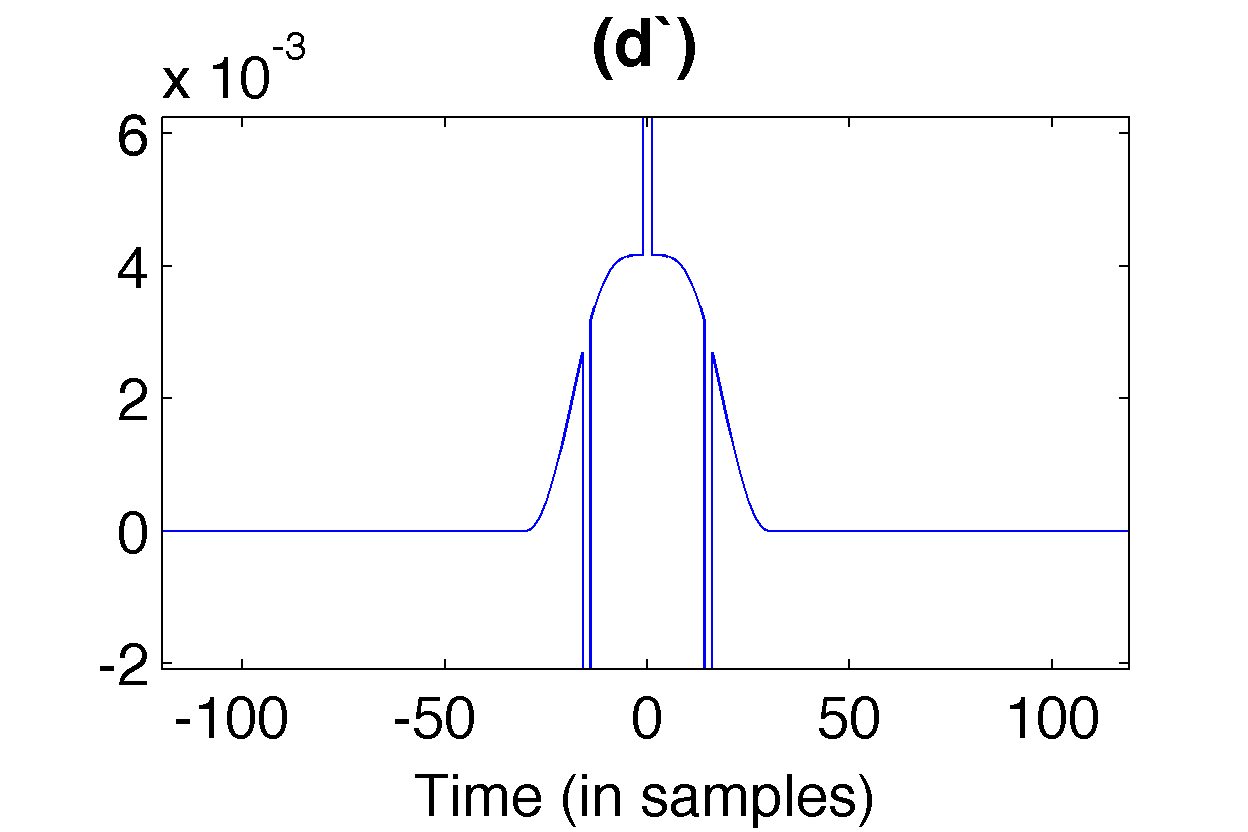} 

\par\end{centering}

\caption{Result of optimization function $\|x - P_{\langle g_{sh}\rangle}x\|_2^2$. $g_{sh}$ (left) and the result of optimization (right):
(a)(b) A sine. (c)(d) An impulse (Dirac). Note that the solution window (d) is actually composed of a smooth bump function in addition to the 3 clearly visible impulses. We provide a zoom-in in (d'). .}\label{fig:other_priors}
\end{figure}

\begin{table}[ht]
\begin{center}
{\small
\hfill{}
\caption{Summary of important priors}
\label{tab:regs} 
\begin{tabular}{|l|p{1.7in}|r|}
\hline 
 Function  &  Effect on the signal & Complexity \\
\hline 
$\|x\|_{1}$  &  sparse representation in time  & $L$ \\
$\|\mathcal{F}x\|_{1}$  & sparse representation in frequency & $L \log(L)$ \\
$\|\nabla x\|_{2}^{2}$  & smoothen in time / concentrate in frequency &  $L \log(L)$\\
$\|\nabla\mathcal{F}x\|_{2}^{2}$ &  smoothen in frequency / concentrate in time &  $L$\\
$\|x\|_{S0}$  &  Concentrate in both time and frequency & $L^2 \log(L)$\\
$\|x\|_{2}^{2}$  &  spread values more evenly/ toward the canonical dual &  $L$ \\
$\text{var}(|x|)$  &  Concentrate the signal in time & $L $ \\
$\text{var}(x^2)$  &  Concentrate the signal in time & $L$\\
$\text{var}\left(\mathcal{F}x\right)$  & Concentrate the signal in frequency &  $L \log(L)$\\
$\text{var}\left((\mathcal{F}x)^2\right)$  &  Concentrate the signal in frequency & $L \log(L)$ \\
$i_{\mathcal{C}}(x)$  &  force $x\in\mathcal{C}$ &  $L^2$ [$L^3$] \\ \hline 
\end{tabular}
}
\hfill{}
\end{center}
\end{table}

\emph{A note on implementation and complexity:} Various methods exist for solving the optimization problems formulated throughout the paper, e.g. generalized forward backward~\cite{raguet2011generalized} or SDMM \cite{setzer2010deblurring}~\cite{combettes2011proximal}, some of which might prove more efficient than PPXA \cite{combettes2007douglas} which we employ. However, optimizing computational complexity is, for several reasons, not at the center of this contribution. First, in contrast to continuously varying the analysis window and transform parameters, a single (or few different), predetermined transform configuration is usually used repeatedly, allowing for off-line computation of the dual window in advance. Second, our main concern is the construction of pairs of dual windows such that duality is satisfied independent of the signal length $L_s$ by imposing support constraints, see Section~\ref{ssec:gabframes}. Therefore, the complexity depends mainly on the support 
size of the dual 
windows and not of 
the signal size $L_s$, see Lemma~\ref{sec:dualfin1}. Third, most of the applications use windows of size smaller than $10^4$ samples, where the implementation used provides quick convergence, except for the $S_0$-norm prior, which might require several minutes of computation time.

The complexity of the overall algorithm depends mostly on the computation of the proximal operator of the selected priors. Indeed, the most expensive prior operator is often the bottleneck of the optimization algorithm (whereas the choice of the algorithm itself rather influences the total number of iterations). For very large $L$, the algorithm will be limited by the projection onto the dual set satisfying the WR systems of equations  $Gx = \delta$. The solution of the projection
\[
\mathop{\operatorname{arg~min}}\limits_{x}\|x-x_{0}\|_{2}^{2}\hspace{1em}\mbox{s. t. }Gx=\delta,
\]
is given by
\[
x=x_{0}-G^{*}\left(G G^{*}\right)^{+}\left(G x_{0}- \delta\right).
\]
The projection itself has quadratic complexity, but the pseudoinverse of the matrix $G G^{*}$ has to be computed before the start of the iteration process. This particular computation scales with $L^3$. Due the restricted use of very long windows, this is no significant limitation. The complexity of the other proximal operators is provided in Table~\ref{tab:regs}. 
Among them, we do observe that the optimization of the (weighted) $S_0$-norm is much more expensive. As a results, it might be preferable to optimize separable time-frequency measures instead, leading to a complexity of $O(L\log(L))$ instead of $O(L^2\log(L))$. 


\section{Numerical experiments}\label{sec:results}
In the following sections, we present several experiments regarding dual Gabor windows optimizing joint TF concentration. Such windows are well-suited both as analysis and synthesis windows, reducing cross-component interference in the Gabor coefficients or increasing the precision of processing operations, respectively. It is widely known that windows of Gaussian type $g(t) = e^{-ct^2}$ optimize a wide variety of
concentration measures, such as the Heisenberg uncertainty (product of variances), $\|g\|_{S_0}/\|g\|_2$ and many more~\cite{rito13}. The same can be said about its discretized counterpart. However, Gaussian windows are not compactly supported in either domain and any attempt to make them so has a detrimental effect on their TF concentration. Compact support in either domain generally introduces infinite support and oscillation in the other, but low amplitude values (in comparison to the desired processing precision) can often be considered irrelevant. The concept of a \emph{good} window is subjective, depending not only on the application, but on the user as well. 

We now present $3$ experiments where we compute various \emph{optimal} dual windows under different assumptions and restrictions. Exp. 1 provides a comparison of the effect of the concentration measures discussed in Sec.~\ref{sec:prox}, when applied jointly in time and frequency. Exp. 2 demonstrates, by tuning the optimization parameters, that the set of dual windows allows surprising freedom when choosing the trade-off between time and frequency concentration. In those two first experiments we impose only a weak support constraint (i.e the support of the dual is significantly longer than the support of the analysis window). In contrast to the previous experiments, Exp. 3 considers a situation beyond the painless case, i.e. the canonical dual window has long (possibly infinite) support. We construct a smooth dual window $h$ supported $I_g$, the support set of $g$ and compare our result to that provided by the truncation method.

\textbf{Experiment 1 - Optimizing TF concentration:} 
In this experiment, we simultaneously optimize TF concentration with regards to the previously introduced measures. This is either achieved by a single prior on the TF representation of the dual window $h$ ($\|x\|_{S_0}$, $\|x\|_{S_0,w}$), or by applying the time and frequency versions of one prior, with equal weights ($\|\nabla x\|_2^2$, $\text{var}(|x|)$, $\text{var}(|x|^2)$), i.e.
\[
\mathop{\operatorname{arg~min}}\limits_{x \in \mathcal{C}_{\text{dual}}\cap \mathcal{C}_{\text{supp}}} f(x) + f(\mathcal{F}x).
\]
The time step $\gamma$ of the algorithm has been tuned experimentally for each prior, to yield good convergence speed and precision. For this experiment, we have chosen a Tukey window with a transition area ratio\footnote{The Tukey window with transition area ratio $r\in[0,1]$ is given by $g(t) = 1$ for $t\in[-(1-r)/2,(1-r)/2]$, $g(t) = 0.5+0.5\cos(\pi (2t-r+1)/2r)$ for $t\in[-1/2,-(1-r)/2]$, $g(t) = 0.5+0.5\cos(\pi (2t+r-1)/2r)$ for $t\in[(1-r)/2,1/2]$ and zero else.} of $3/5$ (see Figure~\ref{fig:Experiment1-ana}), with $L_g = 240$, $a = 50$ and $M=300$. The support of the dual window candidates was restricted to $L_h = 600$. 

\begin{figure}[!thp]
\begin{centering}
\includegraphics[width=0.24\textwidth]{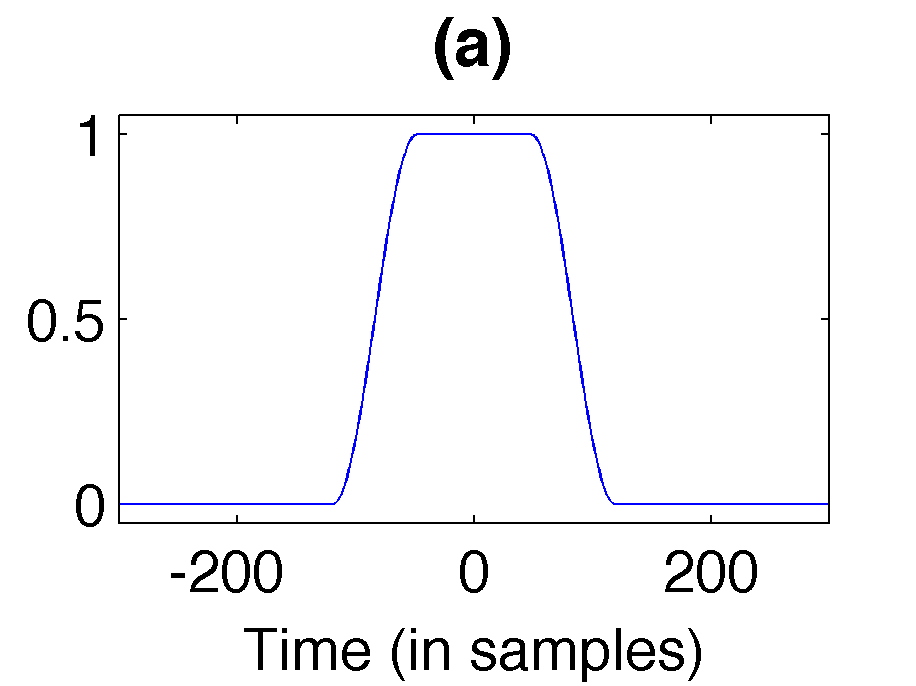}
\includegraphics[width=0.24\textwidth]{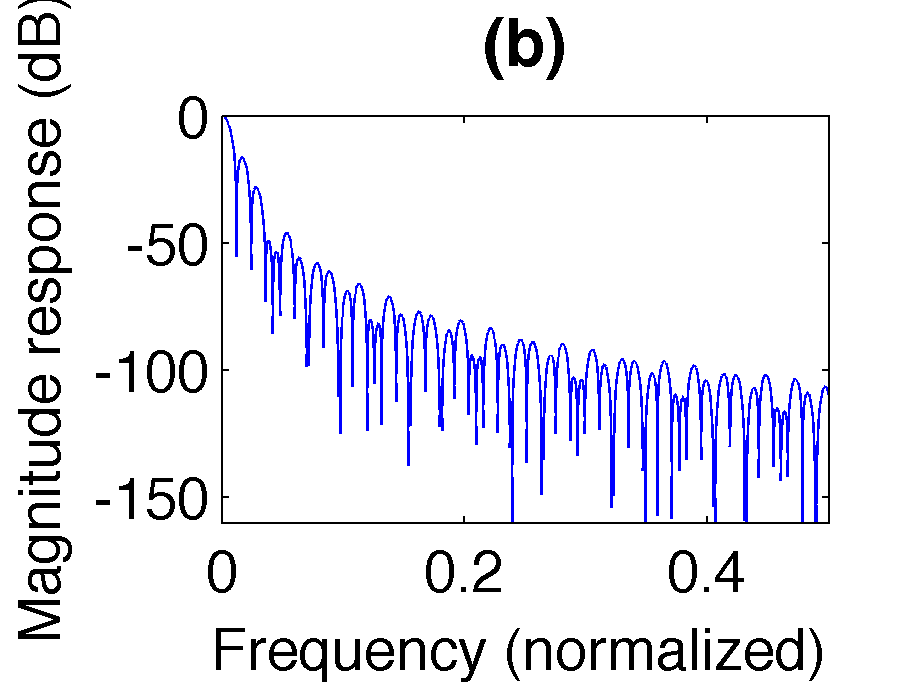}
\par\end{centering}
\caption{The analysis Tukey window and its magnitude frequency response (in dB).}\label{fig:Experiment1-ana}
\end{figure}

Figure~\ref{fig:Experiment1-duals} shows time and frequency representations, as well as the ambiguity function\footnote{In $\CC^L$, the ambiguity function is the STFT of a function with regards to itself: $G_{f,1,L}f$.} of the results. We see  Figure~\ref{fig:Experiment1-duals}(a)(b)(c) that all the criteria provide (visually) nicely concentrated dual Gabor windows, improved over the canonical dual.
In particular, we see that the gradient and energy variance optimal windows are very similar around the origin, whereas the latter shows worse decay. The variance induces the best concentration of large values in an almost rectangular TF area, but similar to the energy variance, no good decay is achieved. Both $S_0$ and weighted $S_0$ priors perform very well, but the weight induces a more symmetric TF concentration and slightly better decay.

Note that, while the shape of this experiment's results is quite characteristic, their quality cannot be representative for any arbitrary setup. In fact the results are highly dependent on the quality of the original Gabor frame, its redundancy and the choice of support constraint. As expected, the optimization effect is smaller, the better the starting setup is. For additional experiments with other starting windows and/or without support constraint, we refer to the webpage.

To further emphasize the similarities and differences between the window functions and to assess their quality, we computed for all the results every concentration measure discussed in Sec.~\ref{sec:prox}, see Table~\ref{tab:ex1-crit1}. Besides demonstrating nicely that the various solutions actually provide the best result in their respective criteria, the table underlines the similarity of gradient and energy variance optimization. In total, except for the canonical dual and $S_0$-norm solution, all results are reasonably close. The different shape of the $S_0$-norm solution is easily explained by the fact that it is the only measure that does not penalize high-energy contributions away from the origin. Furthermore, in Table~\ref{tab:ex1-crit2}, the following classical measures for window quality are presented: $-3$~dB width (time), main lobe width (frequency), side lobe attenuation (frequency, in dB) and side lobe decay (in dB). Please note that the lack of an underlying continuous function for the 
dual windows prevents us from determining the side lobe decay rate from the degree of smoothness of the window. Therefore, as a rough approximation, we compute the ratio between the largest sidelobe and the largest of the final $3$ side lobes below the Nyquist frequency instead. Notably, the variance optimization concentrates the main energy contribution in the smallest TF area as indicated in the third column, while the gradient optimization provides arguably the most balanced solution and the best decay properties, followed by the considerably more expensive $S_0$ and weighted $S_0$ solutions.

\begin{figure}[!thp]
\begin{centering}
\includegraphics[width=0.15\textwidth]{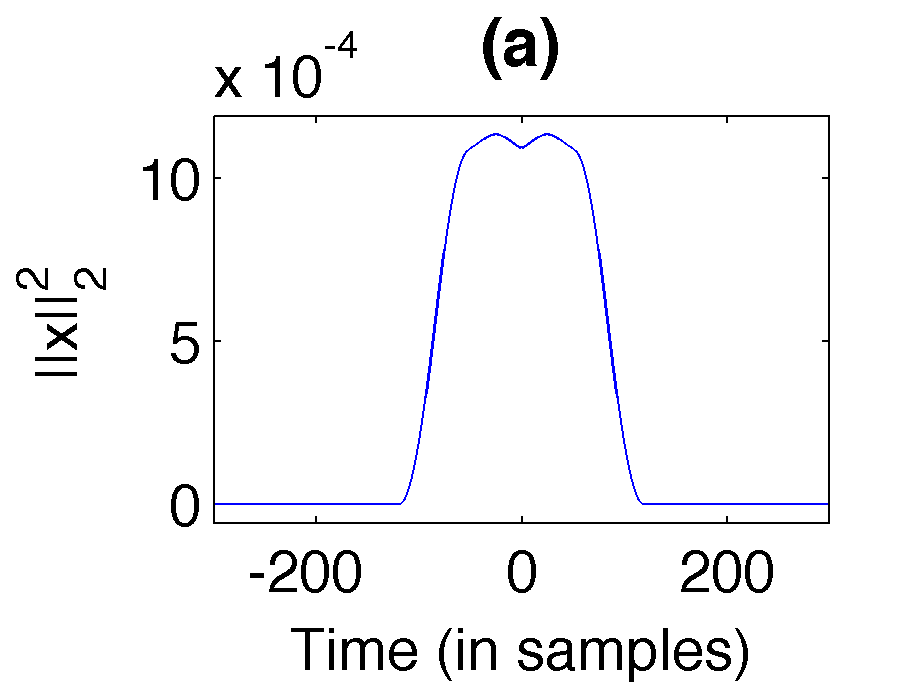}
\includegraphics[width=0.15\textwidth]{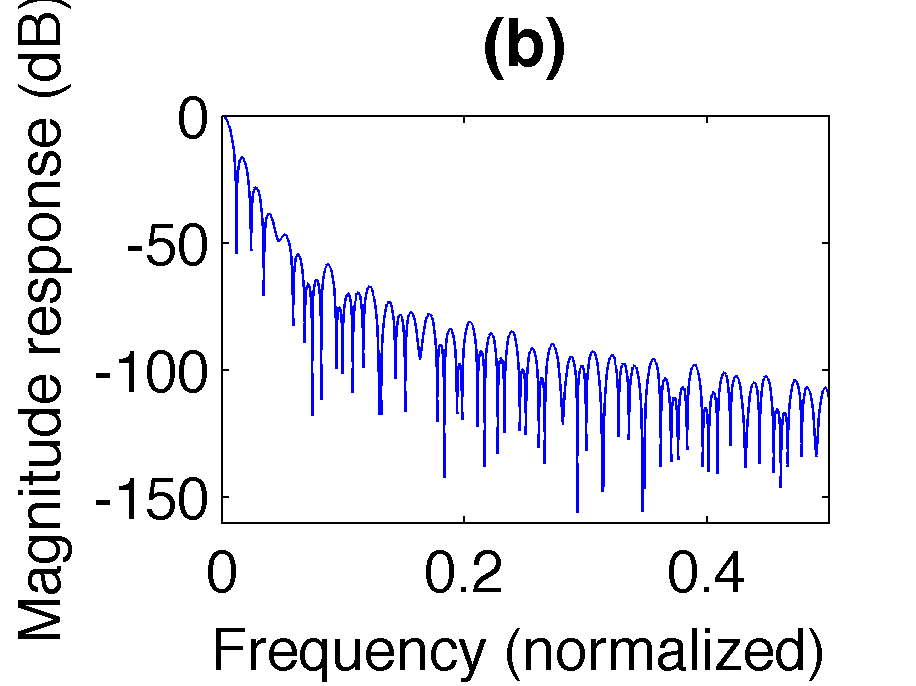}
\includegraphics[width=0.15\textwidth]{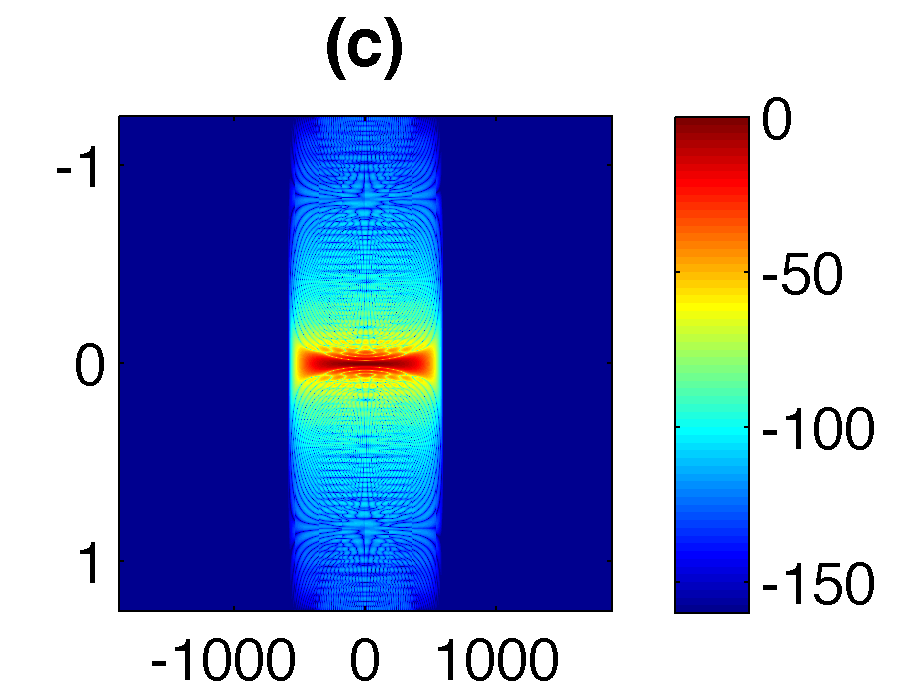}\\

\includegraphics[width=0.15\textwidth]{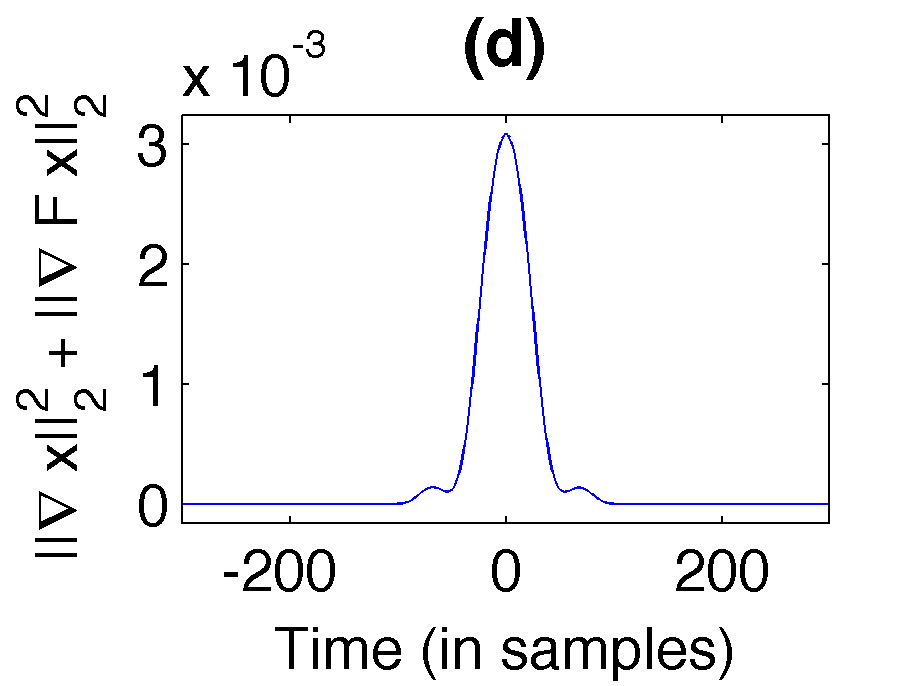}
\includegraphics[width=0.15\textwidth]{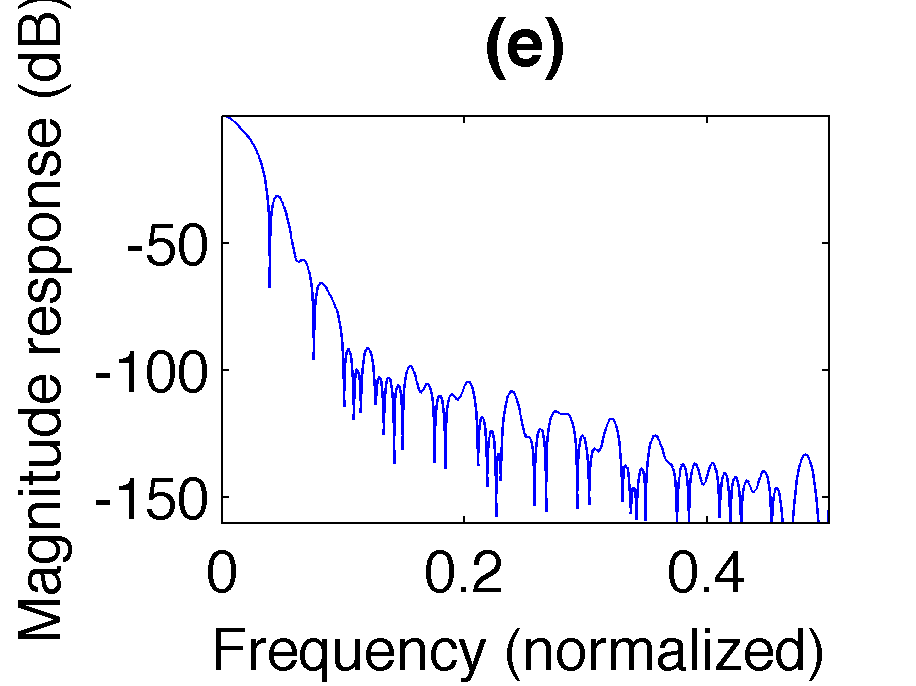}
\includegraphics[width=0.15\textwidth]{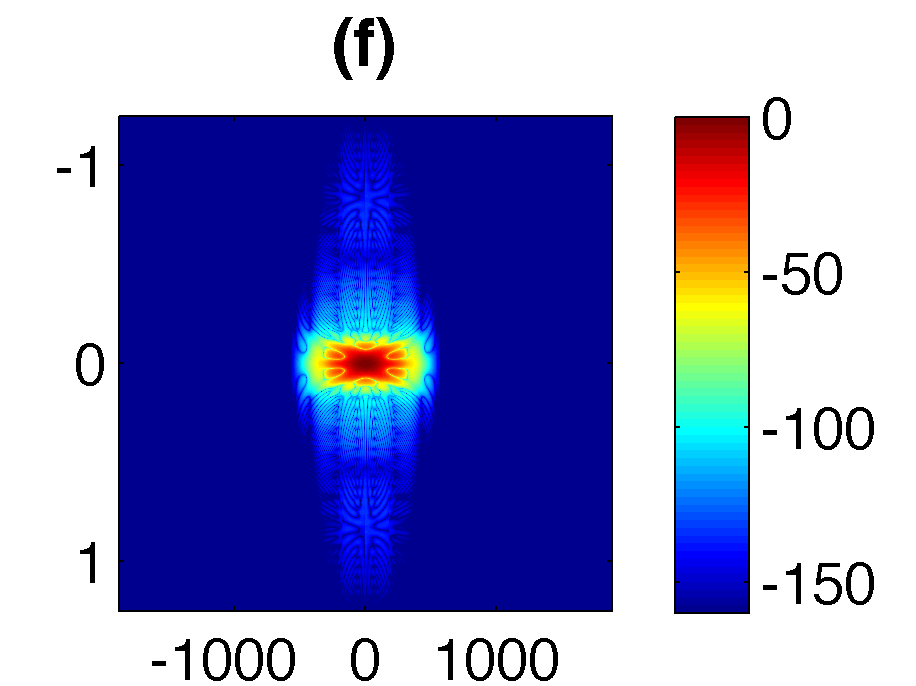}\\

\includegraphics[width=0.15\textwidth]{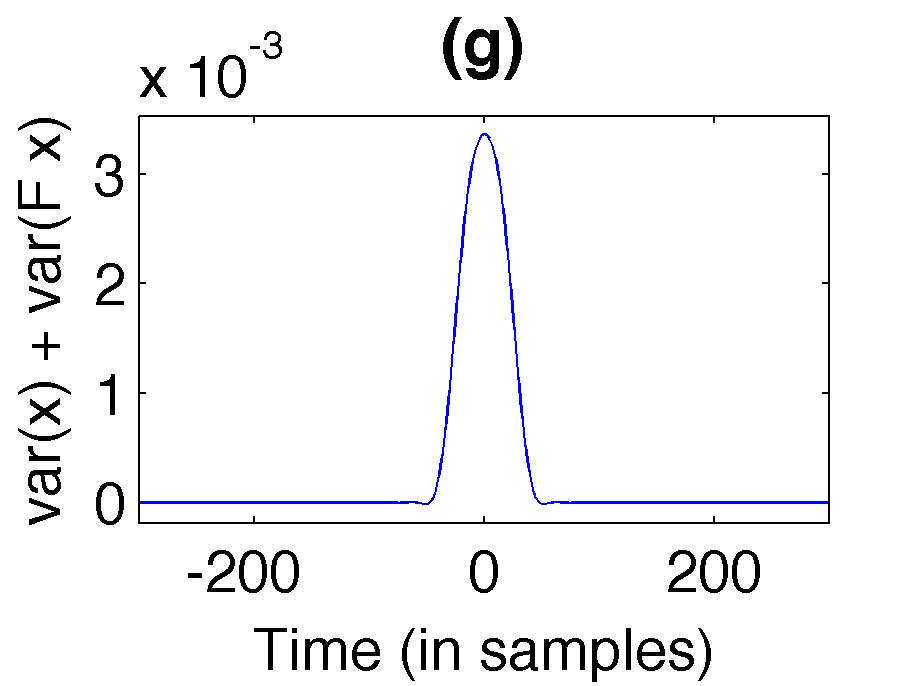}
\includegraphics[width=0.15\textwidth]{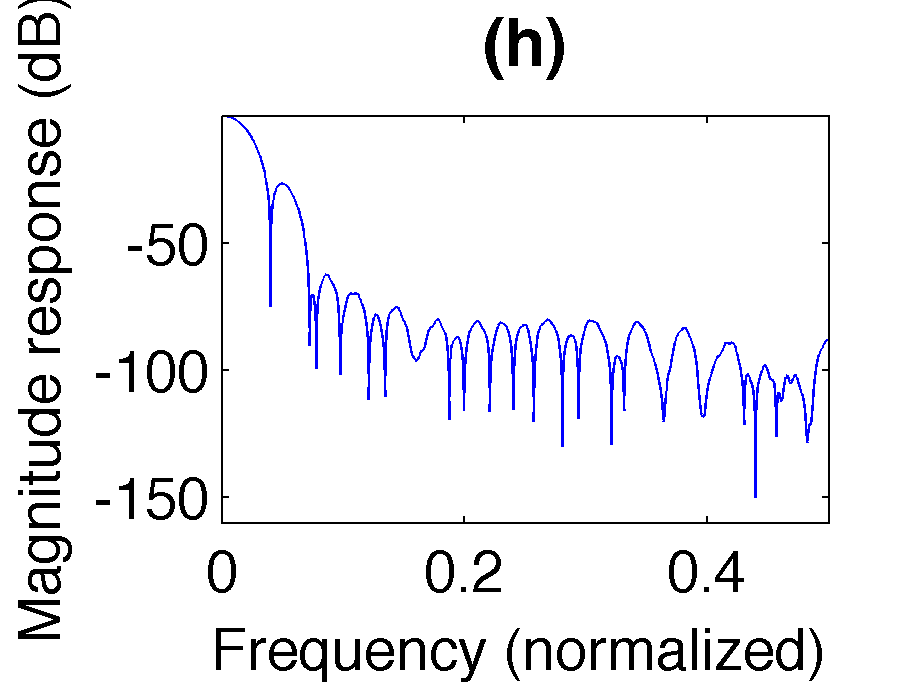}
\includegraphics[width=0.15\textwidth]{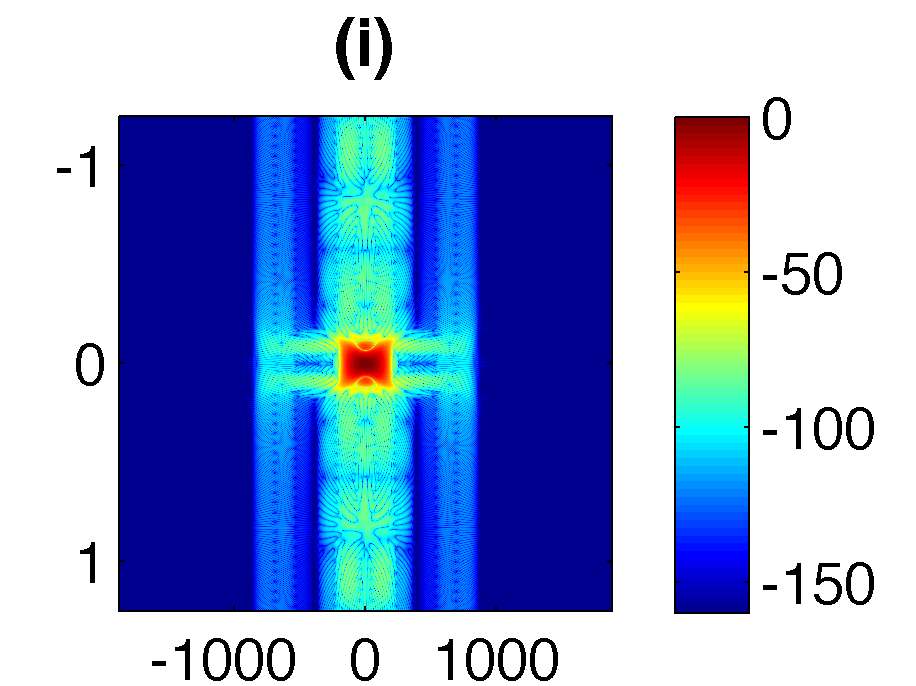}\\

\includegraphics[width=0.15\textwidth]{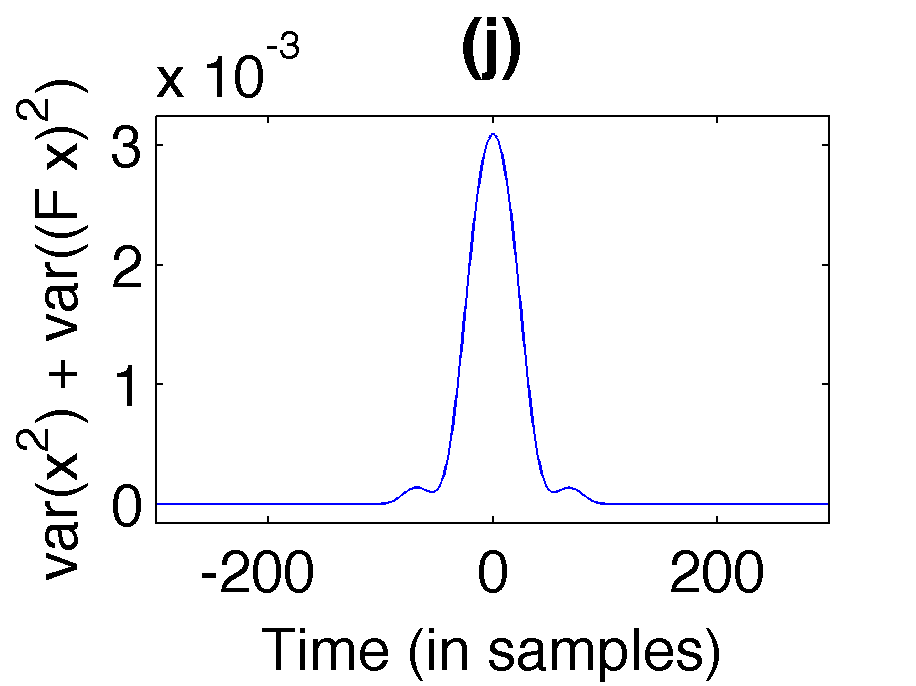}
\includegraphics[width=0.15\textwidth]{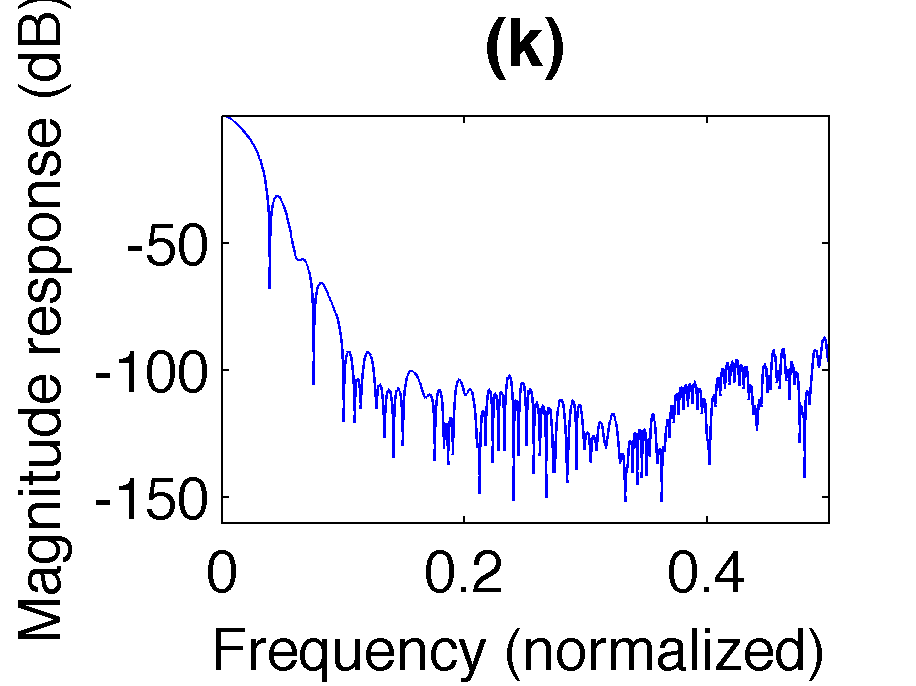}
\includegraphics[width=0.15\textwidth]{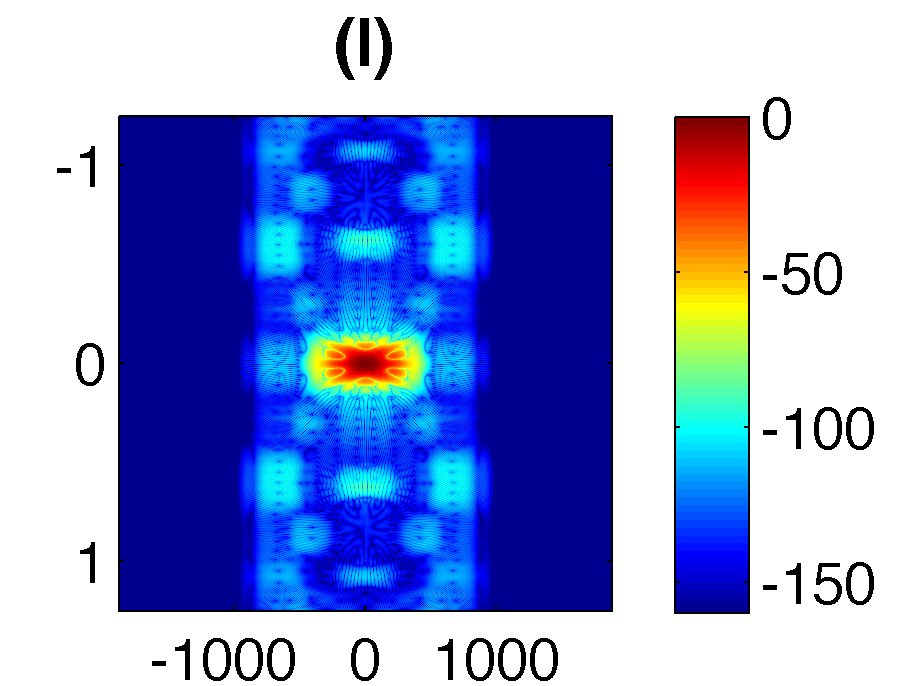}\\

\includegraphics[width=0.15\textwidth]{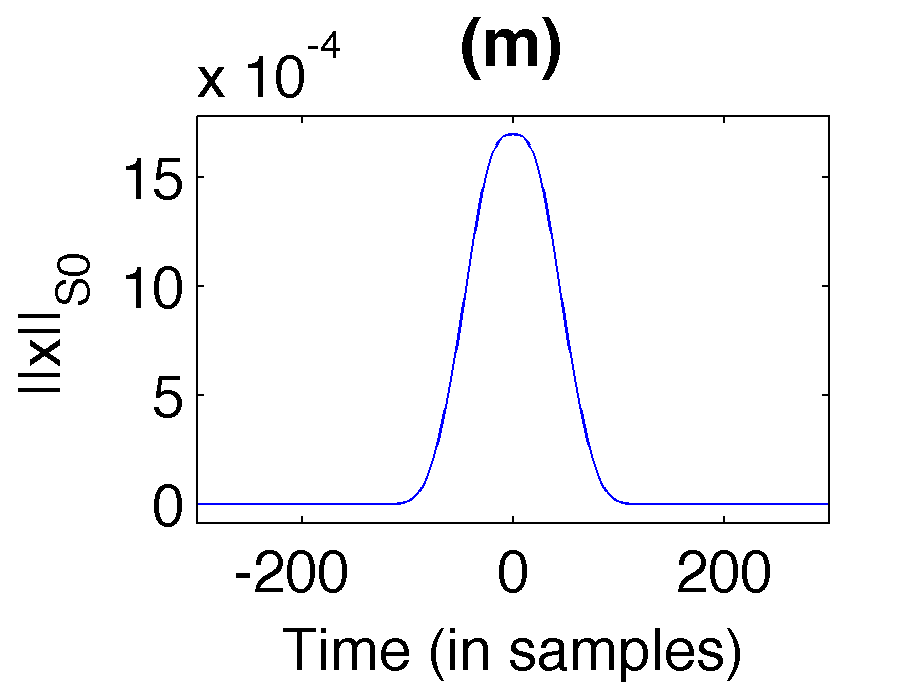}
\includegraphics[width=0.15\textwidth]{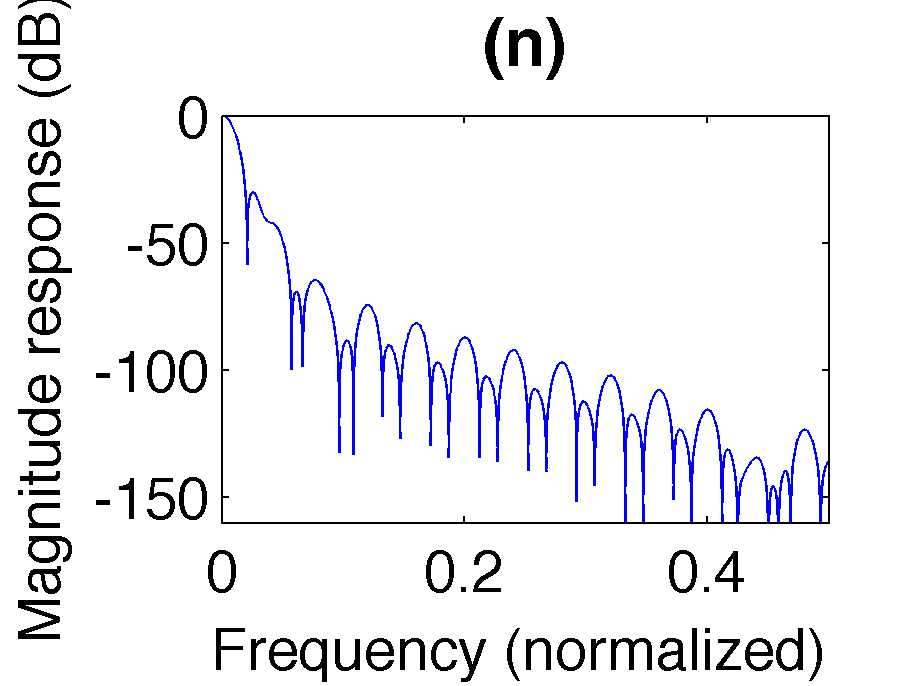}
\includegraphics[width=0.15\textwidth]{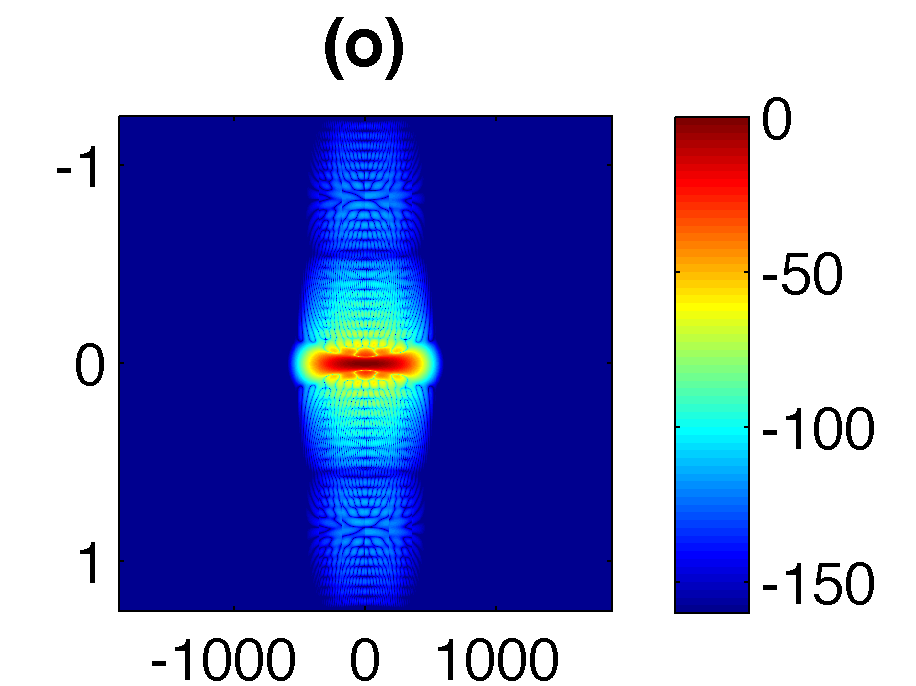}\\

\includegraphics[width=0.15\textwidth]{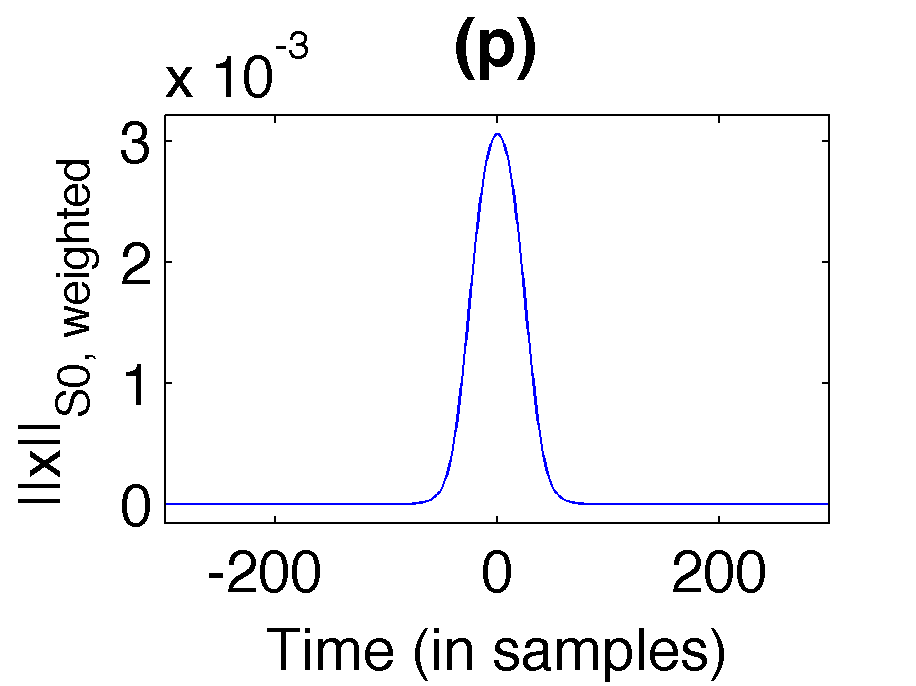}
\includegraphics[width=0.15\textwidth]{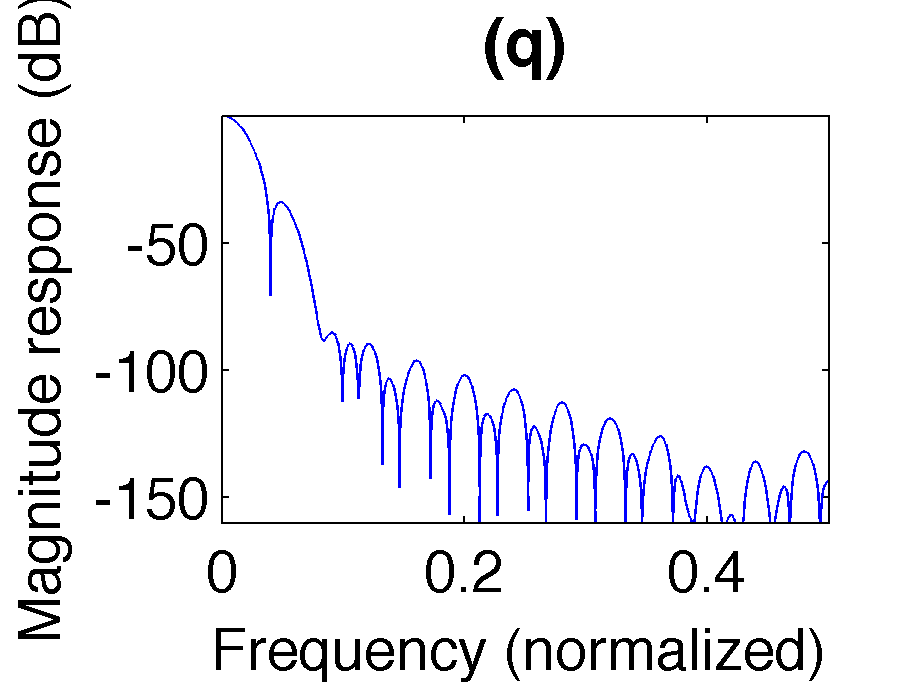}
\includegraphics[width=0.15\textwidth]{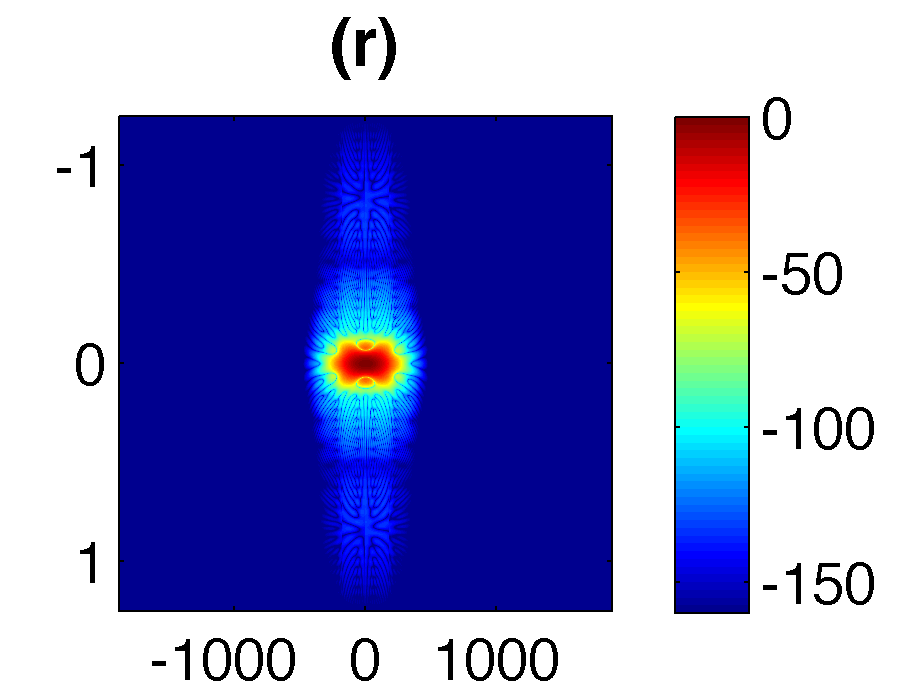}\\
\par\end{centering}
\caption{Results of time-frequency optimization for Exp. 1. The first column shows the time representation, the second the frequency representation and the third the ambiguity function of the window. From top to bottom: canonical dual, gradient, variance, energy variance, $S_0$ norm, weighted $S_0$ norm.}\label{fig:Experiment1-duals}
\end{figure}

\begin{table*}[thb] 
\begin{small}
\begin{center}\begin{tabular}{ |c|ccccccccccc|} 
 \hline 
 & $\frac{\|\nabla x\|_2}{10^{4}}$ & $\frac{\|\nabla \mathcal{F} x\|_2}{10^{4}}$ & $\frac{\text{var}(|x|)}{100}$ & $\frac{\text{var}(|\mathcal{F}x|)}{100}$ & $\frac{\sqrt{\text{var}(|x|^2)}}{10^3}$ & $\frac{\sqrt{\text{var}(|\mathcal{F}x|^2)}}{10^{3}}$ & $\|x\|_{S_0}$ & $\|x\|_{S_0,w}$ & $\frac{\|x\|_{1}}{100}$ & $\frac{\|\mathcal{F}x\|_{1}}{10}$ & $\frac{\|x\|_{2}}{10}$\\ 
 \hline 
 $x_{\text{can}}$  & $  \mathbf{0.2074} $ & $   3.2155 $ & $  40.0898 $ & $   0.3523 $ & $  17.7331 $ & $ \mathbf{1.1433} $ & $   1.3883 $ & $   1.4859 $ & $   \mathbf{1.3638} $ & $   1.8670 $ & $ \mathbf{0.5422} $ \\ 
 $x_{\nabla}$ & $   0.6638 $ & $   1.4988 $ & $   8.2516 $ & $   0.8612 $ & $   8.2638 $ & $   3.6602 $ & $   1.4294 $ & $   0.7527 $ & $   1.9182 $ & $   1.6855 $ & $   1.1069 $ \\ 
 $x_{\text{var}}$ & $   0.7941 $ & $   1.5037 $ & $  \mathbf{4.2567} $ & $   1.6165 $ & $   8.2908 $ & $   4.3791 $ & $   1.5032 $ & $   0.7163 $ & $   2.0872 $ & $   1.6718 $ & $   1.2612 $ \\ 
 $x_{\text{envar}}$ & $   0.6662 $ & $  \mathbf{1.4978} $ & $   8.2108 $ & $   0.9996 $ & $   \mathbf{8.2580} $ & $   3.6734 $ & $   1.4310 $ & $   0.7534 $ & $   1.9213 $ & $   1.6852 $ & $   1.1096 $ \\ 
 $x_{S_0}$ & $   0.2796 $ & $   2.1551 $ & $  18.1214 $ & $   \mathbf{0.2434} $ & $  11.8833 $ & $   1.5418 $ & $ \mathbf{1.2881} $ & $   0.9455 $ & $   1.4952 $ & $   1.7094 $ & $   0.6397 $ \\ 
 $x_{S_0,w}$ & $   0.6498 $ & $   1.5361 $ & $   5.8810 $ & $   0.8149 $ & $   8.4695 $ & $   3.5830 $ & $   1.4092 $ & $  \mathbf{0.6866} $ & $   1.9418 $ & $   \mathbf{1.6681} $ & $   1.0987 $ \\ 
 \hline 
 \end{tabular}\end{center} 
  \end{small}
 \caption{Time and frequency concentration measures for the solution dual windows (Exp. 1). Columns are the various criteria, rows indicate the different solutions windows. Note that the criteria have been scaled with appropriate powers of $10$ to improve readability of the table. Best results are indicated in bold.}\label{tab:ex1-crit1}
\end{table*}

\begin{table}[thb] 
\begin{small}
\begin{center}\begin{tabular}{ |c|ccccc|} 
 \hline 
 & $-3$~dB-W & ML-W & 1/pr-W & SL-A & SL-D\\ 
 \hline 
 $x_{\text{can}}$ & $  27.5000 $ & $   \mathbf{1.2292} $ & $ 295.8398 $ & $  16.3689 $ & $ 106.0177 $ \\ 
 $x_{\nabla}$    & $   8.5000 $ & $   3.9792 $ & $ 295.6575 $ & $  31.5146 $ & $ \mathbf{124.6016} $ \\ 
 $x_{\text{var}}$  & $   \mathbf{8.1667} $ & $   4.0208 $ & $ \mathbf{304.5363} $ & $  26.6322 $ & $  70.3673 $ \\
 $x_{\text{envar}}$ & $   8.5000 $ & $   3.9792 $ & $ 295.6575 $ & $  31.5040 $ & $  76.8473 $ \\ 
 $x_{S_0}$        & $  16.1667 $ & $   2.1042 $ & $ 293.9675 $ & $  30.0840 $ & $ 109.6085 $ \\ 
 $x_{S_0,w}$     & $   8.8333 $ & $   4.0208 $ & $ 281.5524 $ & $  \mathbf{33.9670} $ & $ 120.2623 $ \\ 
 \hline 
 \end{tabular}\end{center} 
 \end{small}
 \caption{Window quality measures for the solution dual windows. Rows indicate the different solutions windows, columns are the various criteria (from left to right): $-3$~dB width in time (in percent of $L_h$), main lobe width in frequency (in percent of the full frequency range) and the reciprocal value of their product (providing a measure of joint TF concentration), side lobe attenuation in frequency (in dB) and side lobe decay (in dB, measured as the amplitude difference of the maximum sidelobe and the largest of the $3$ final side lobes below the Nyquist frequency). Best results are indicated in bold.}\label{tab:ex1-crit2}
\end{table} 

\textbf{Experiment 2 - Controlling the time-frequency trade-off:} 

It is well known that no function can be arbitrarily concentrated in both the time and the frequency domain simultaneously. When choosing a dual window to a given Gabor frame the concentration is further limited by the duality conditions, the shape and the quality of the given frame. Oversimplified, a badly conditioned Gabor frame (with large frame bound ratio $B/A$), admits only badly concentrated duals. However, even if the canonical dual window is well concentrated overall, applications might benefit from the improvement of time concentration versus frequency concentration and vice-versa. To see this, just recall that the TF shape of the synthesis window limits the precision of TF processing.

The following experiment demonstrates the surprising flexibility that the set of dual windows allows when choosing the appropriate TF concentration trade-off. The system parameters are the same as in Exp. 1: $L_g = 240$, $a = 50$, $M=300$ and dual window support less or equal to $L_h = 600$. However, to provide a more diverse set of examples, we exchanged the Tukey window for an Itersine window. Based on the good results in Exp. 1, we selected the time and frequency gradient priors to control the TF spread and optimize 
\[
\mathop{\operatorname{arg~min}}\limits _{x \in \mathcal{C}_{\text{dual}}\cap \mathcal{C}_{\text{supp}}} \lambda_1 \|\nabla \mathcal{F} x\|_2^2 + \lambda_2\|\nabla x\|_2^2
\]
for varying $\lambda_1,\lambda_2 \in\RR^+$, therefore balancing both concentration measures against one another. Recall that $\|\nabla \mathcal{F} x\|_2^2$ leads to concentration in time, while $\|\nabla x\|_2^2$ promotes concentration in frequency. 

The results, presented in Figure~\ref{fig:Experiment2-duals} and Table~\ref{tab:ex2-crit2}, demonstrate the large amount of freedom when choosing the TF concentration trade-off. It also shows that extreme demands on either time or frequency concentration come at the cost of other properties. In this particular example, time concentration comes at the cost of worse sidelobe attenuation, while strong demands on frequency concentration inhibit quick frequency decay. Despite this, all four solution windows behave as expected and show reasonable to very good overall TF concentration.

\begin{figure}[!thp]
\begin{centering}

\includegraphics[width=0.15\textwidth]{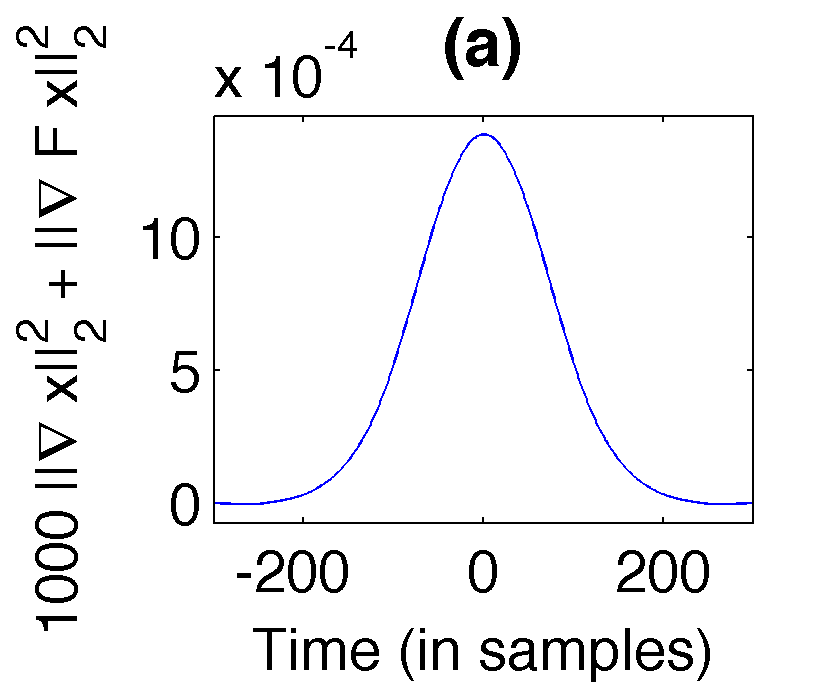}
\includegraphics[width=0.15\textwidth]{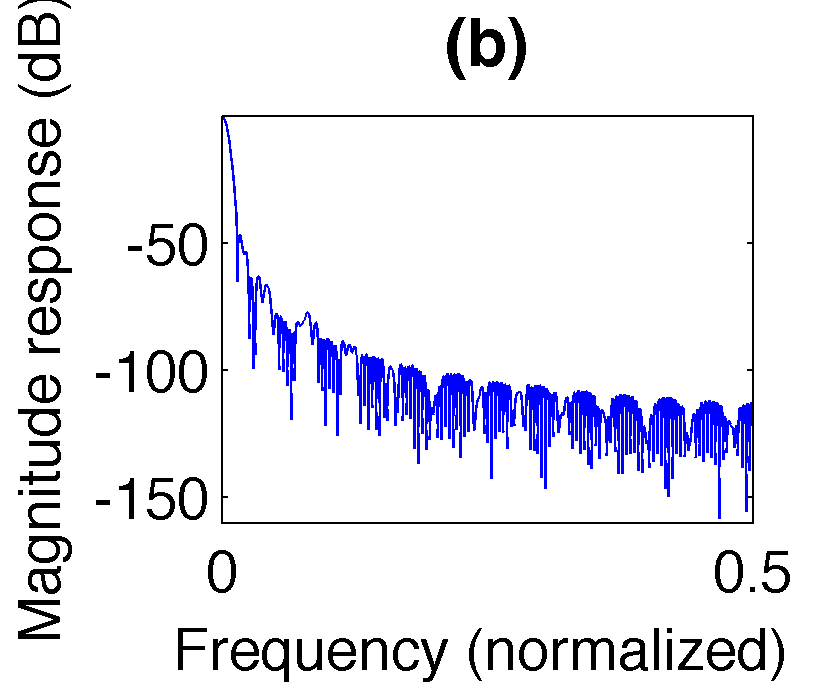}
\includegraphics[width=0.15\textwidth]{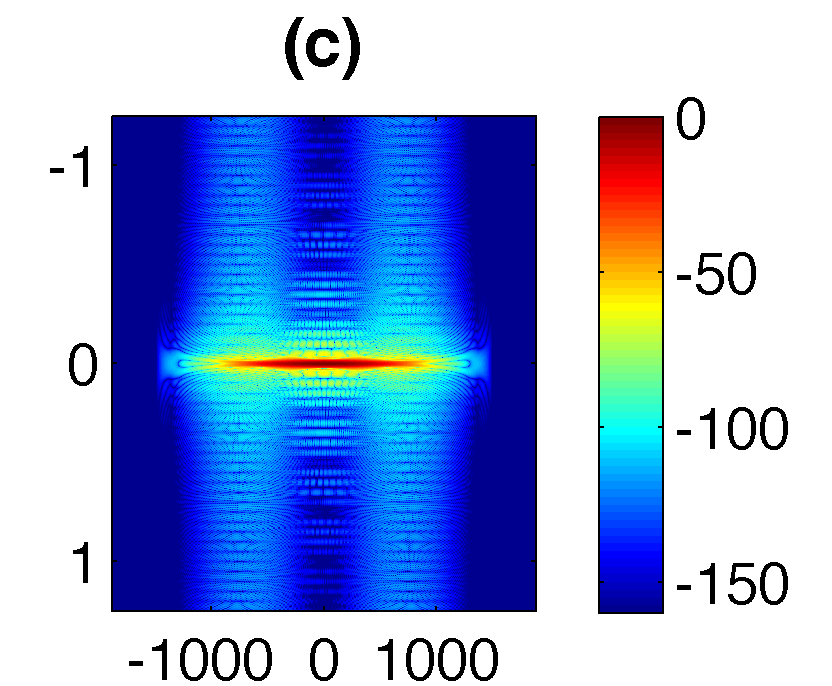}\\

\includegraphics[width=0.15\textwidth]{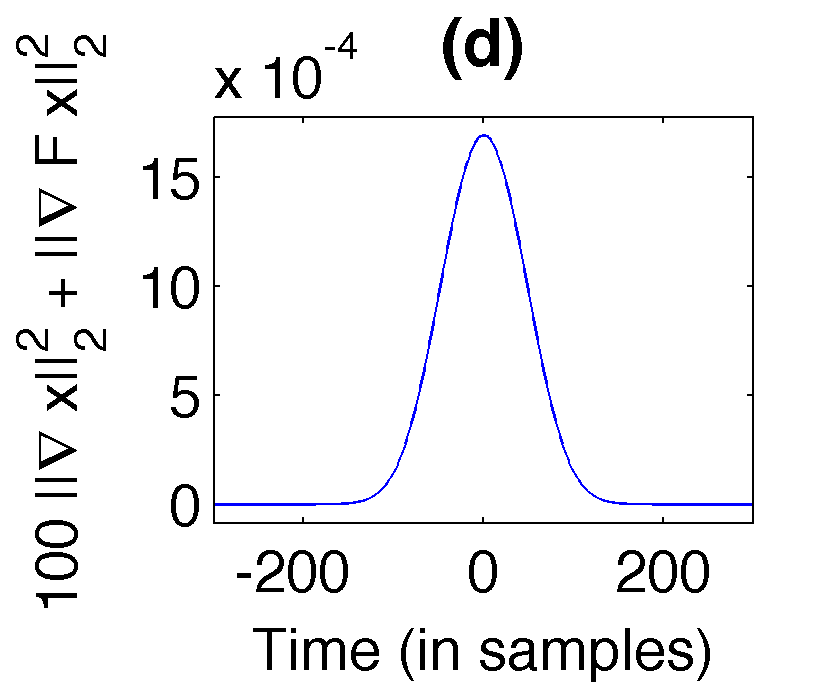}
\includegraphics[width=0.15\textwidth]{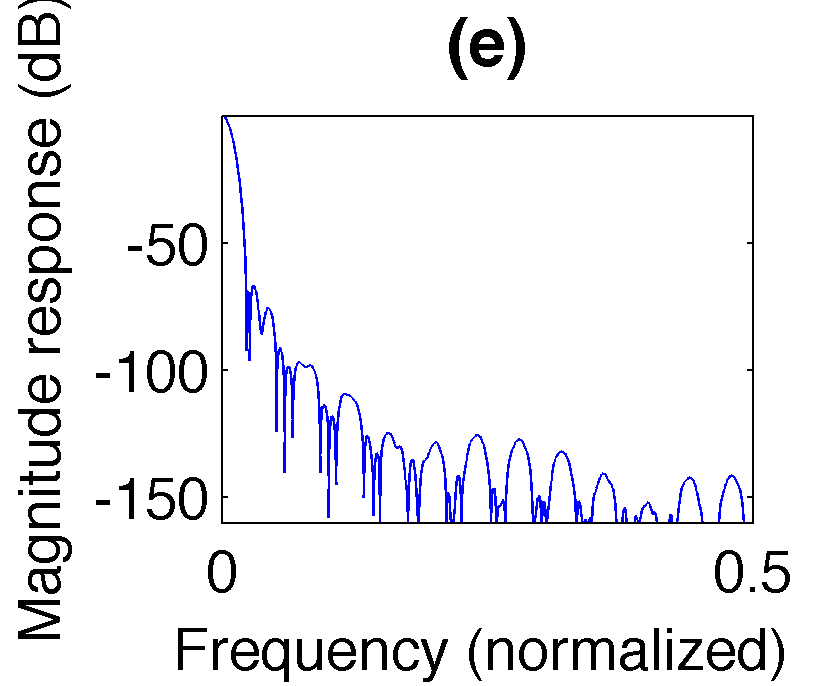}
\includegraphics[width=0.15\textwidth]{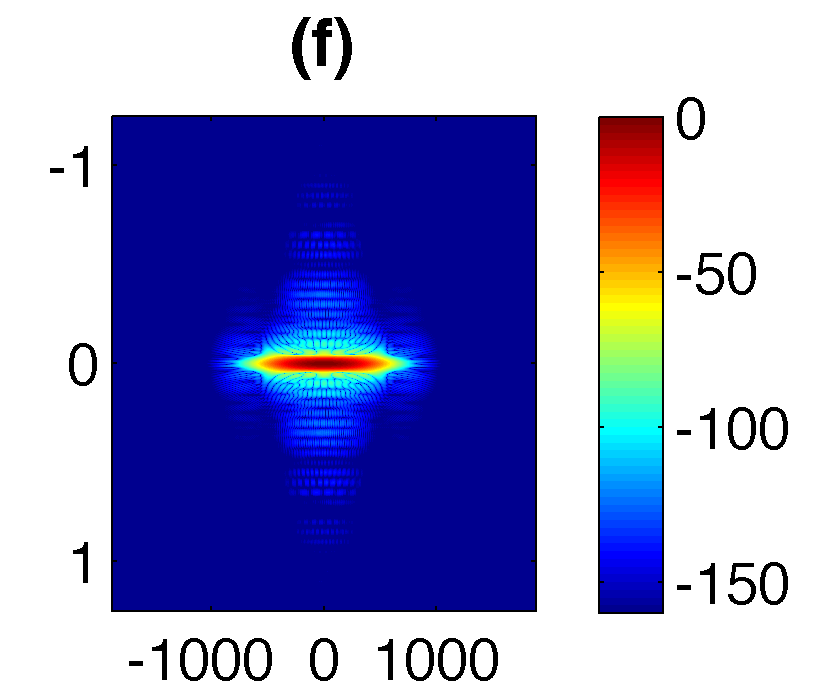}\\

\includegraphics[width=0.15\textwidth]{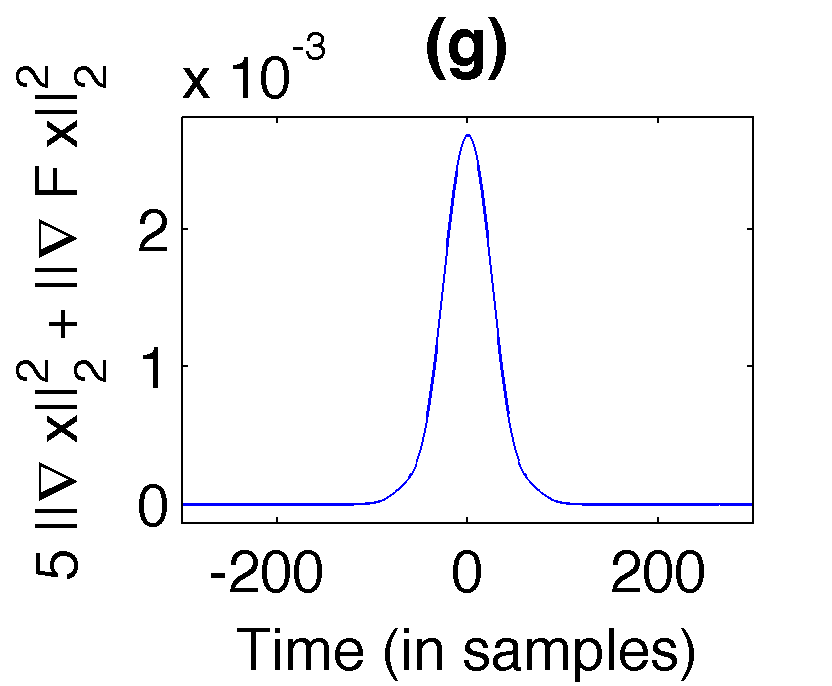}
\includegraphics[width=0.15\textwidth]{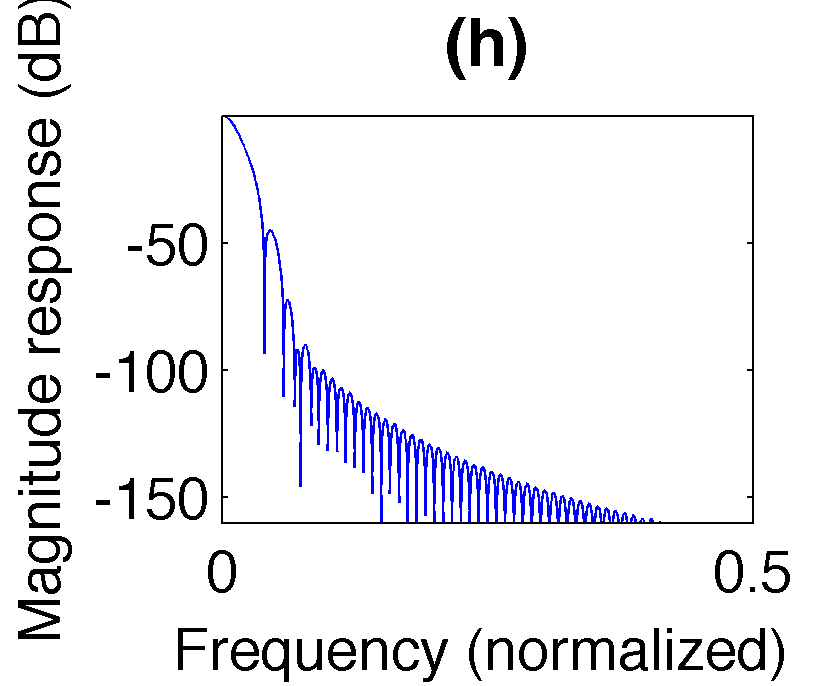}
\includegraphics[width=0.15\textwidth]{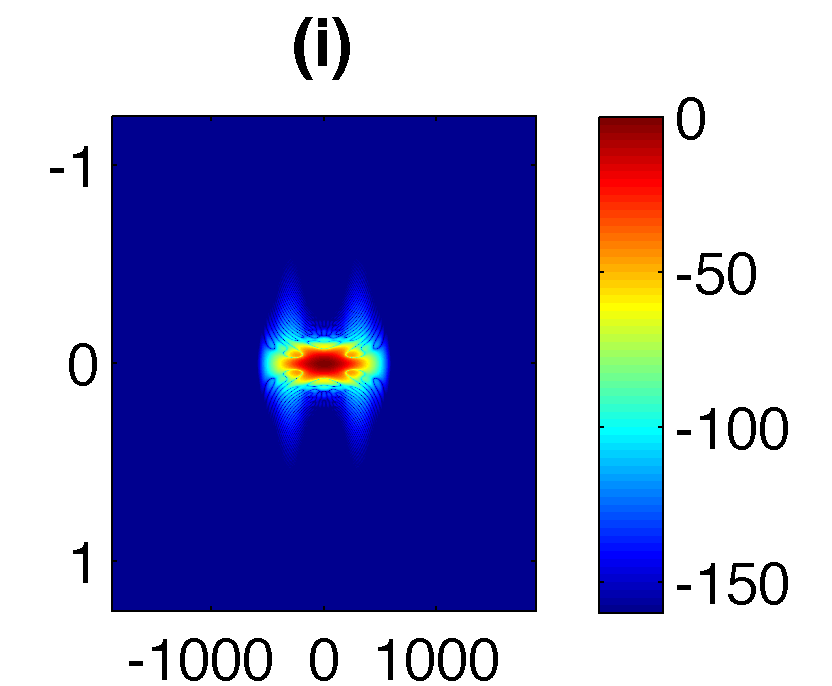}\\

\includegraphics[width=0.15\textwidth]{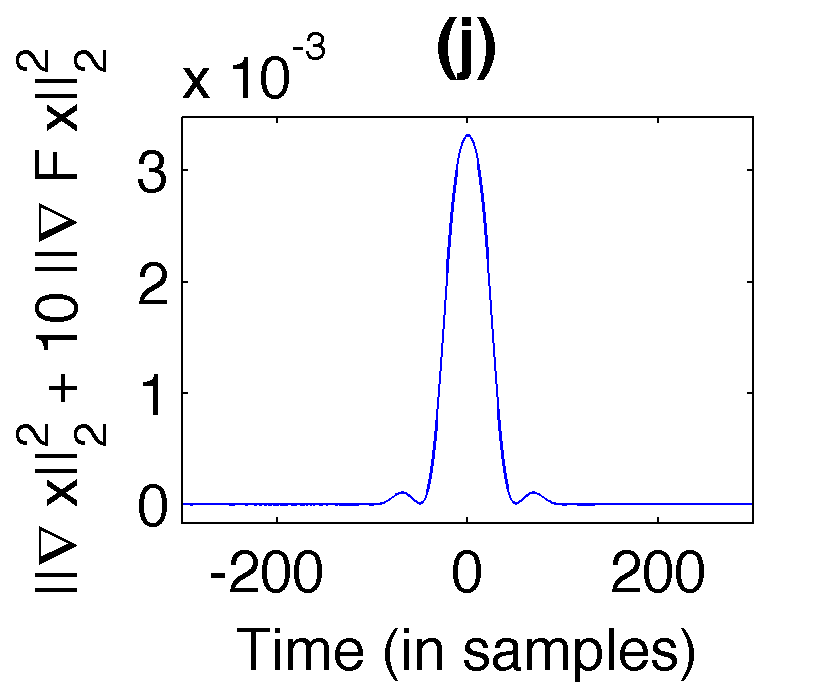}
\includegraphics[width=0.15\textwidth]{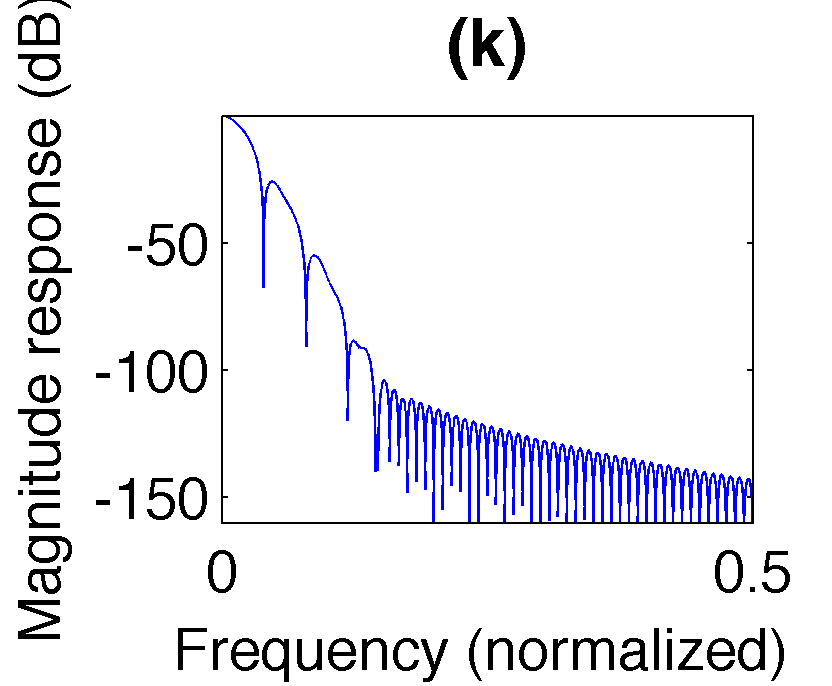}
\includegraphics[width=0.15\textwidth]{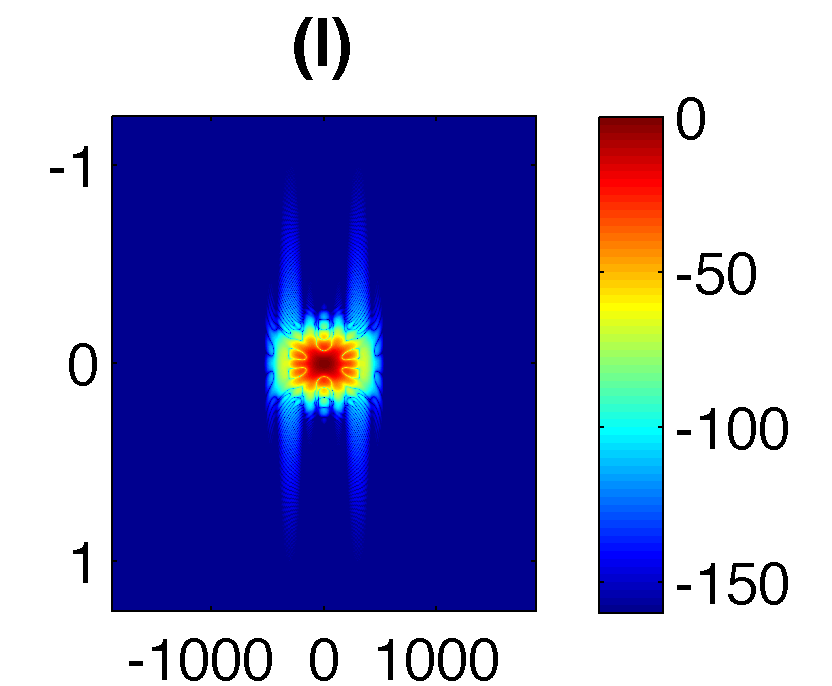}\\

\end{centering}
\caption{Results for the TF trade-off experiment optimizing $\lambda_1\|\nabla \mathcal{F} x\|_2^2 + \lambda_2\|\nabla x\|_2^2$. (a)-(c) $\lambda_1 = 1000$, $\lambda_2 = 1$. (d)-(f) $\lambda_1 = 100$, $\lambda_2 = 1$. (g)-(i) $\lambda_1 = 5$, $\lambda_2 = 1$. (j)-(l)  $\lambda_1 = 1$, $\lambda_2 = 10$.}\label{fig:Experiment2-duals}
\end{figure}

\begin{table}[thb] 
\begin{small}
\begin{center}\begin{tabular}{ |c|ccccc|} 
 \hline 
 & $-3$~dB-W & ML-W & 1/pr-W & SL-A & SL-D \\ 
 \hline 
 $x_{1/10}$    & $ 8.1667 $ & $   3.9375 $ & $ 621.9631 $ & $  25.7860 $ & $ 153.8656 $  \\ 
 $x_{5/1}$    & $  9.8333 $ & $   4.0208 $ & $ 505.8400 $ & $  45.0251 $ & $ 161.5928 $  \\ 
 $x_{100/1}$ & $ 18.5000 $ & $   2.3125 $ & $ 467.4945 $ & $  66.7578 $ & $ 103.8376 $  \\ 
 $x_{1000/1}$& $  28.1667 $ & $   1.4792 $ & $ 480.0400 $ & $  46.6063 $ & $  71.2169 $ \\ 
\hline 
 \end{tabular}\end{center} 
 
 \begin{center}\begin{tabular}{ |c|cc|} 
 \hline 
 & $\frac{\|\nabla x\|_2}{10^{4}}$ & $\frac{\|\nabla \mathcal{F} x\|_2}{10^{4}}$\\ 
 \hline 
 $x_{1/10}$    & $   0.7887 $ & $   1.5032 $ \\ 
 $x_{5/1}$    & $   0.5005 $ & $   1.7009 $ \\ 
 $x_{100/1}$  & $   0.2365 $ & $   2.6391 $ \\ 
 $x_{1000/1}$ & $   0.1545 $ & $   4.1818 $ \\ 
\hline 
 \end{tabular}\end{center} 
 \end{small}
 
 \caption{Window quality measures and prior values for the solution windows of the TF trade-off experiment. The subscript refers to the ratio $\lambda_1/\lambda_2$ of regularization parameters.}\label{tab:ex2-crit2}
\end{table} 

\textbf{Experiment 3 - Smooth dual windows with short support:} 
So far, we have only considered weak support constraints, therefore allowing for a wide variety of dual window choices.
Sometimes a stricter constraint is desired, e.g. to reduce delay in real-time application or to minimize the required number
of multiplication operations. Also, for fixed redundancy, reducing the time step and the number of frequency channels is computationally
cheaper than an increase of the number of frequency channels (amounting to a longer FFT length). Therefore, we consider
a non-painless setup, i.e. a system with few frequency channels: $M<L_g$.

The construction of dual Gabor windows with short support has been attempted previously,
e.g. by the truncation method in~\cite{st98-8}, cf. Section~\ref{sec:design}. However, the solutions obtained, 
are often badly localized in frequency. This is a result of the truncation method
yielding to nonsmooth solutions, i.e. solutions with ``jumps'' (discontinuity-like
behavior) in time. By solving the optimization problem Equation~\eqref{eq:minimizedualsupp} with suitably chosen 
priors $f_i$, better results can be obtained, showing reasonable smoothness and 
therefore frequency concentration. 

We start from a Gabor system $\mathcal{G}(g,30,60)$ with redundancy $2$. 
The analysis window $g$ is chosen as a Nuttall window of length $L_{g}=120$ samples\footnote{The Nuttall window~\cite{nu81-1} 
is a compactly supported $4$-term window of the form $\chi_{[-0.5,0.5]}\sum_{k=0}^3 c_k \cos(2k\pi \cdot)$, with $c_0=0.355768,\ c_1=0.487396,\ c_2 = 0.144232$ and $
c_3 = 0.012604$. It has very good side lobe attenuation and decay properties, but is less popular than the wide-spread Hann and Blackman windows. However, out of those three windows, it best approximates the Gaussian and optimizes all the joint TF concentration measures discussed in Sec.~\ref{sec:prox} (see webpage).}. 

We desire a dual window $h$ with the same support as $g$, i.e. $L_h = L_g = 120$.
Furthermore, we aim to achieve localization and smoothness by selecting
the priors \mbox{$f_{1}=\|\cdot\|_{1}$},$f_{2}=\|\mathcal{F}(\cdot)\|_{1}$,
$f_{3}=\|\nabla(\cdot)\|_{2}^{2}$ and $f_{4}=\|\nabla\mathcal{F}(\cdot)\|_{2}^{2}$.
Here, $f_3,\ f_4$ have been chosen to induce smoothness and localization, while $f_1,\ f_2$ 
have been added to improve the shape of the window, in particular they serve as a counter to 
the solution's tendency to have multiple peaks. This is unwanted as it leads to windows 
with ambiguous temporal or frequency position. Heuristically, minimizing the $l^{1}$-norm pushes
all big coefficients to similar values, therefore achieving the suppression of
multiple significant peaks.

The results in Figure~\ref{fig:Experiments-FIR}(c)(d) show the
optimal dual window with regards to the regularization parameters
$\lambda_{1}=\lambda_{2}=0.001$ and $\lambda_{3}=\lambda_{4}=1$, chosen experimentally
to provide a good result. As reference, we included the least-squares solution provided by
the truncation method, see Figure~\ref{fig:Experiments-FIR}(e)(f). At closer examination, we see that 
the improved sidelobe decay of the optimized dual window comes at the cost of $5$~dB of side lobe attenuation,
compared to the least squares solution, cf. Table~\ref{tab:ex3-crit2}.

\begin{figure}[!thp]
\begin{centering}
\includegraphics[width=0.23\textwidth]{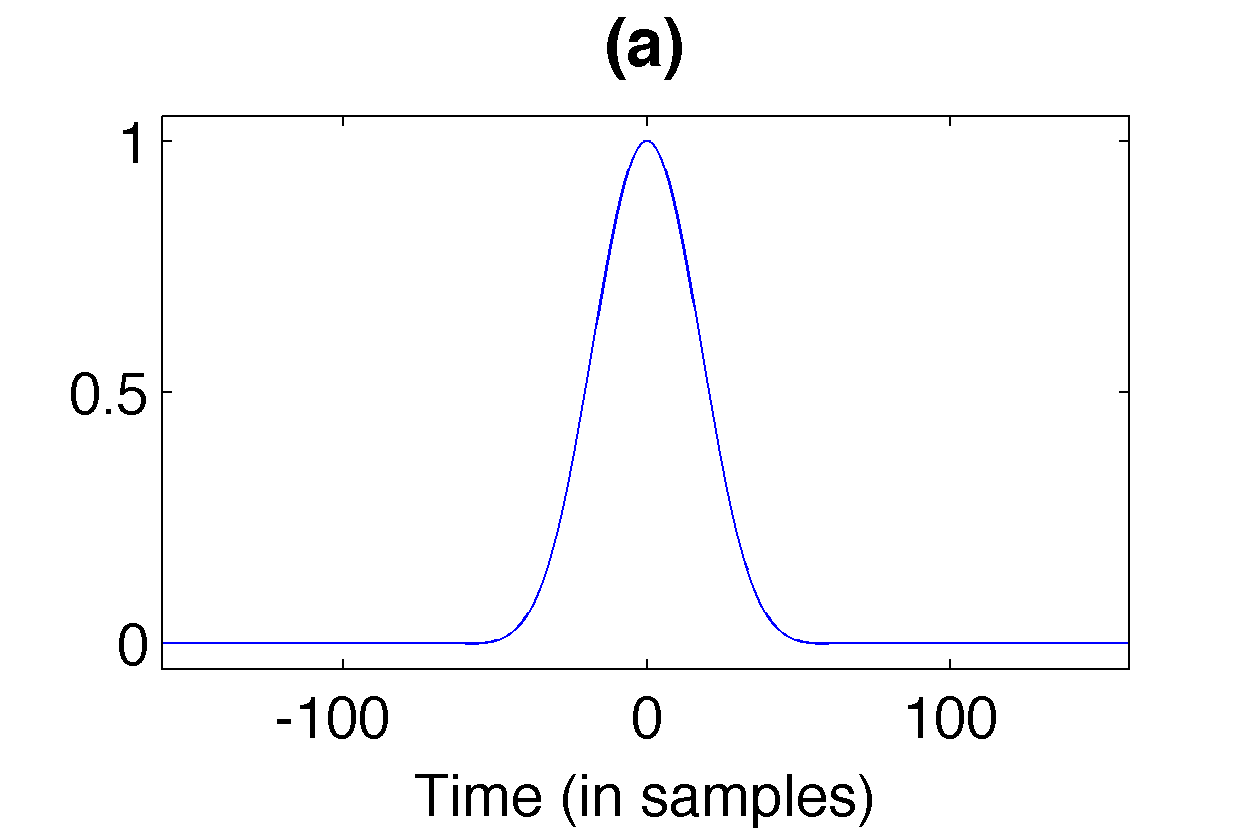}
\includegraphics[width=0.23\textwidth]{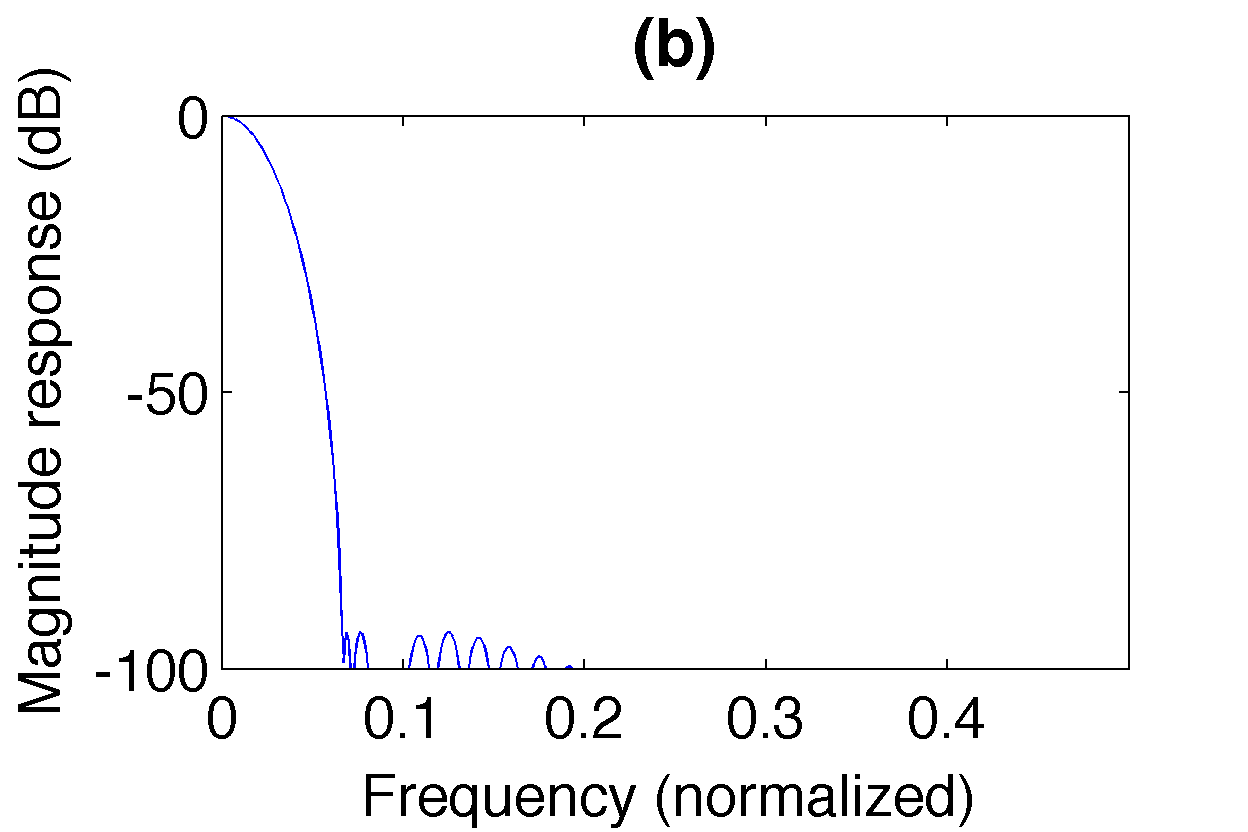} 
\par\end{centering}
\begin{centering}
\includegraphics[width=0.23\textwidth]{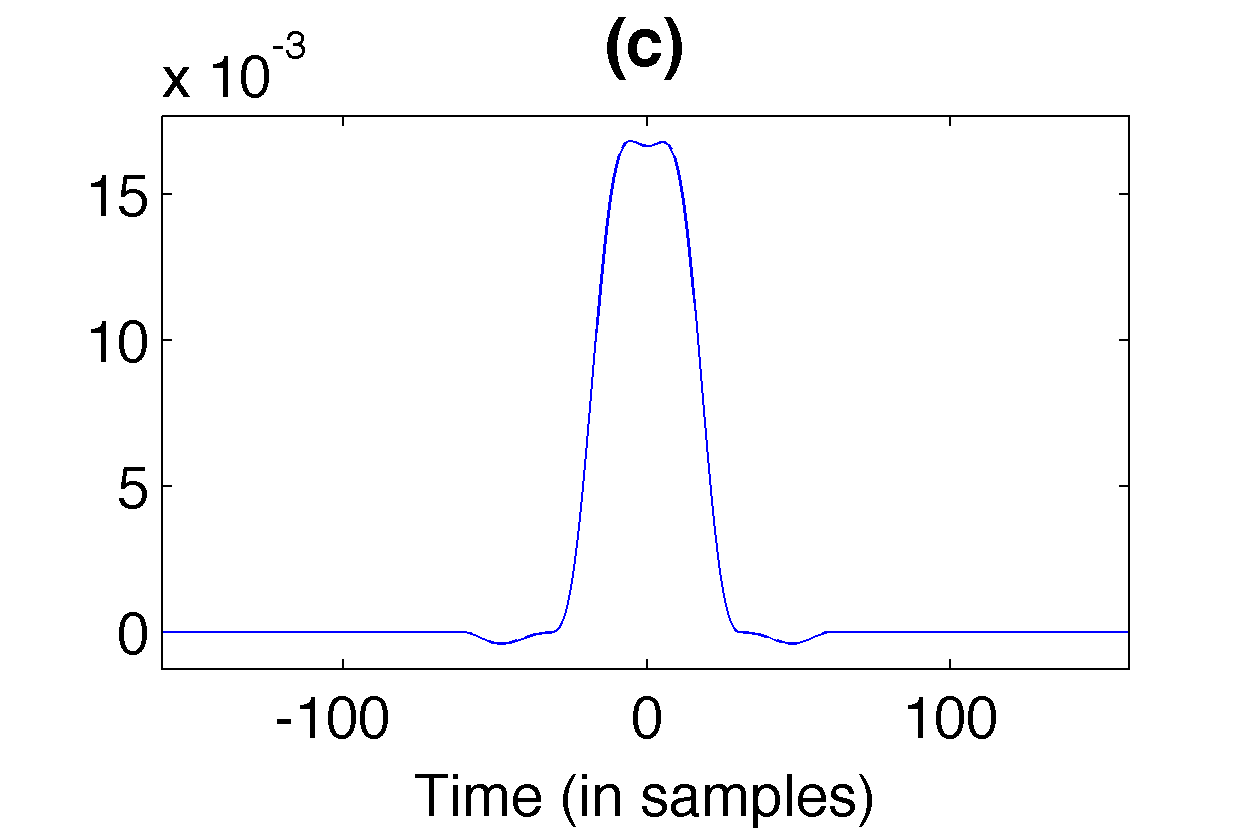}
\includegraphics[width=0.23\textwidth]{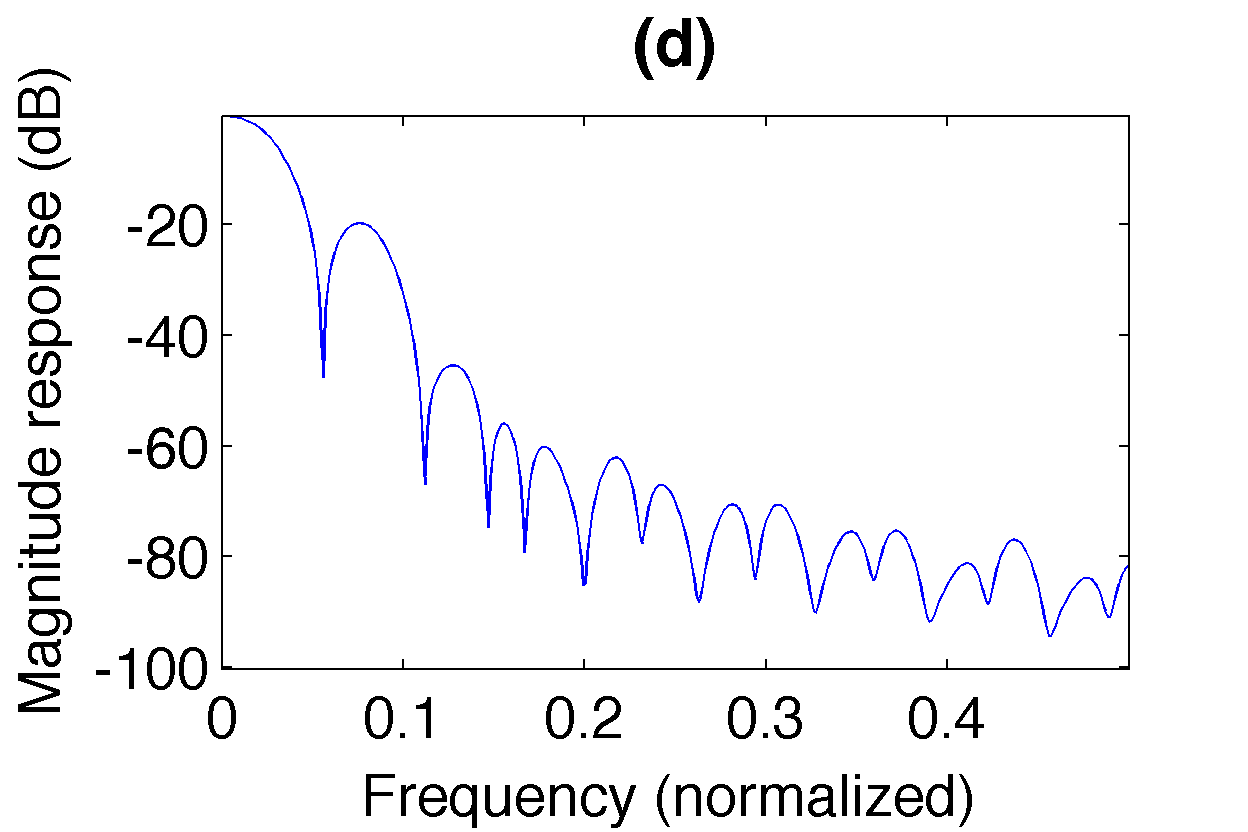} 
\par\end{centering}
\begin{centering}
\includegraphics[width=0.23\textwidth]{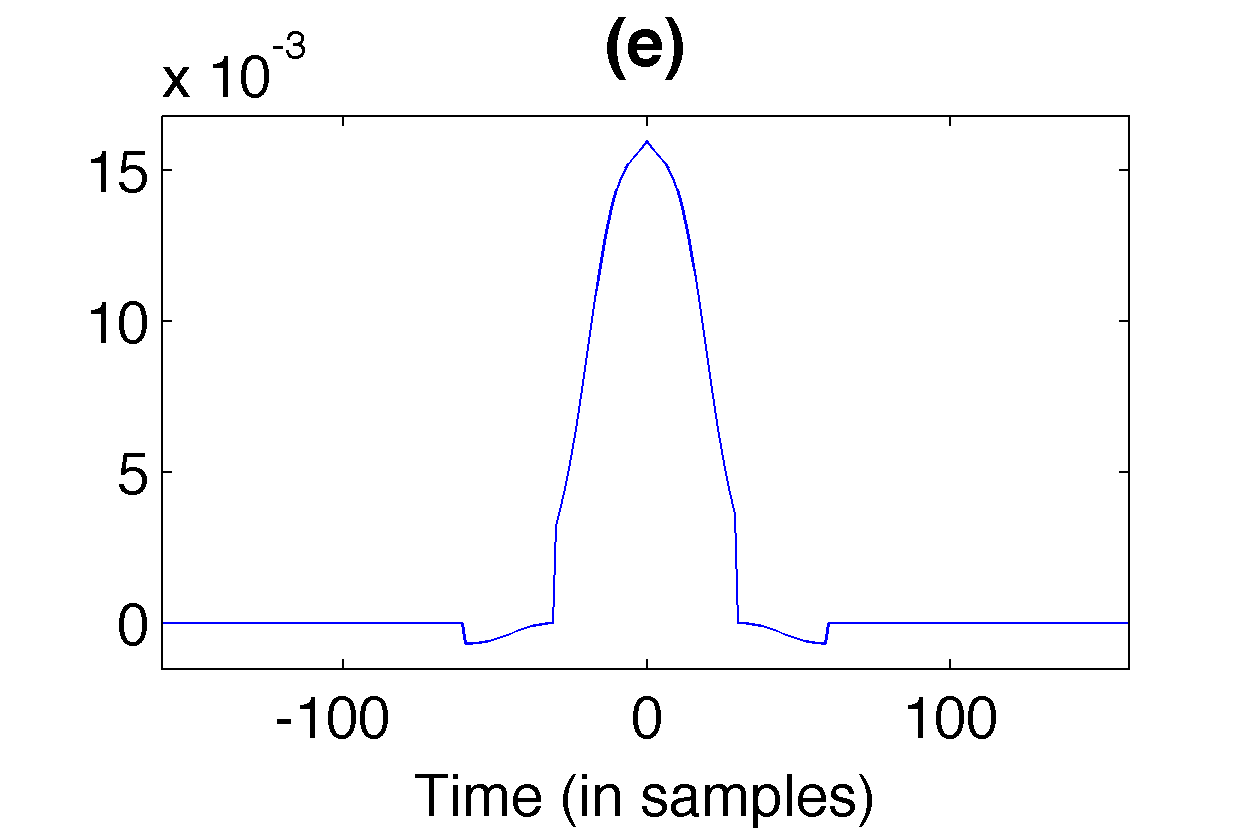}
\includegraphics[width=0.23\textwidth]{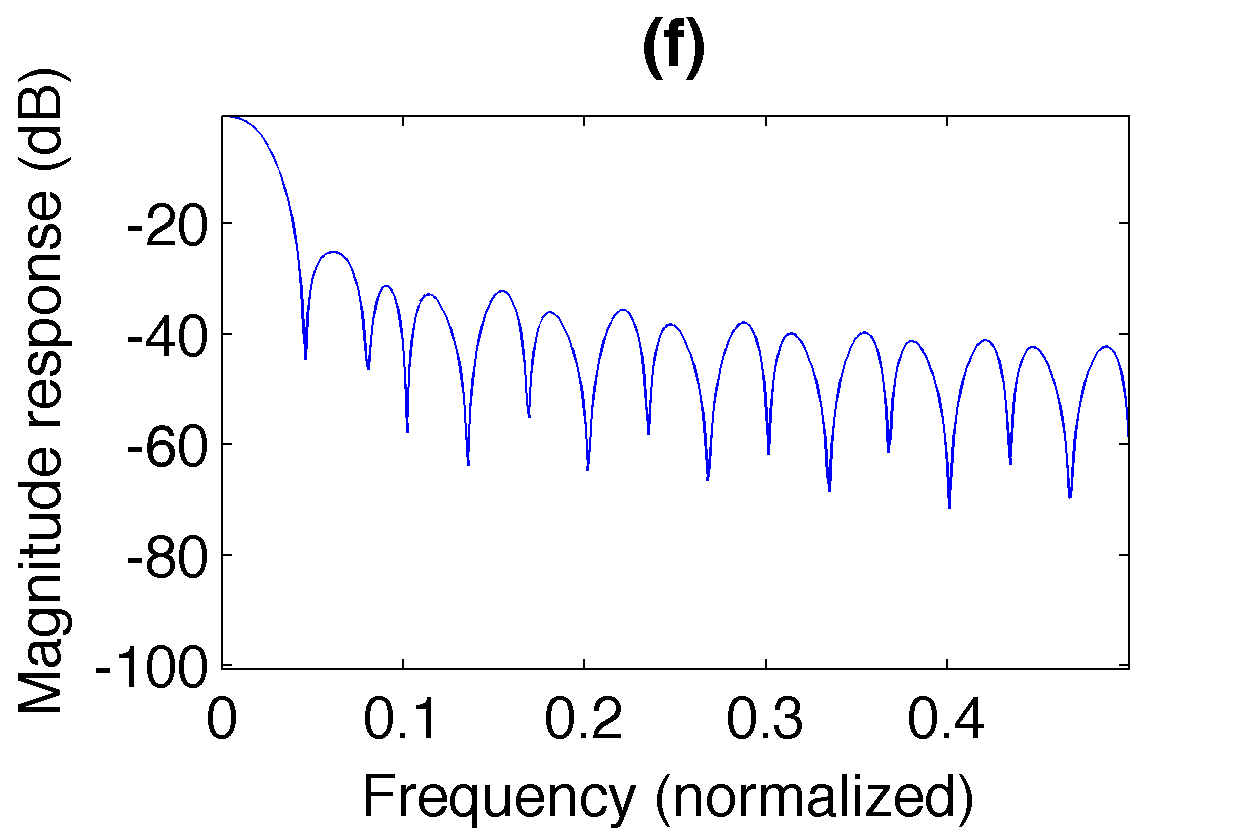} 
\par\end{centering}

\caption{Experiments. (a)(b) Analysis window in time and frequency.
(c)(d) Optimized synthesis window in time and frequency. (e)(f) Truncation method result in time and frequency.}\label{fig:Experiments-FIR}
\end{figure}

\begin{table}[thb] 
\begin{small}
\begin{center}\begin{tabular}{ |c|ccccc|} 
 \hline 
 & $-3$~dB-W & ML-W & 1/pr-W & SL-A & SL-D \\ 
 \hline 
 $x_{\text{trunc}}$ & $  34.1667 $ & $   \mathbf{4.6875} $ & $ \mathbf{124.8780} $ & $  \mathbf{24.6729} $ & $  20.3697 $  \\ 
 $x_{\text{opt}}$ & $  \mathbf{29.1667} $ & $   5.7292 $ & $ 119.6883 $ & $  19.4349 $ & $  \mathbf{69.4470} $  \\ 
 \hline 
 \end{tabular}\end{center} 
 \begin{center}\begin{tabular}{ |c|cccc|} 
 \hline 
 & $\frac{\|\nabla x\|_2}{100}$ & $\frac{\|\nabla \mathcal{F} x\|_2}{100}$ & $ \|x\|_1 $ & $ \|\mathcal{F}x\|_1 $\\ 
 \hline 
 $x_{\text{trunc}}$ &  $   0.6156 $ & $   2.8254 $ & $   \mathbf{0.9092} $ & $   0.6774 $ \\ 
 $x_{\text{opt}}$ &  $   \mathbf{0.5579} $ & $   \mathbf{2.3705} $ & $   0.9278 $ & $   \mathbf{0.6188} $ \\ 
 \hline 
 \end{tabular}\end{center} 
 \end{small}
 \caption{Window quality measures and prior values for the solution windows of the non painless experiment. The subscript 'opt' refers to optimization method and the subscript 'trunc' to truncation method.}\label{tab:ex3-crit2}
\end{table} 

In the setup above, the canonical dual window would have very long, quite possibly infinite
support. To guarantee compact support on $L_{h}=L_{g}$ for the canonical dual, going to the painless case would increase the number
of frequency channels to $M\geq120$ and 
increase the redundancy twofold, the latter being an unwanted side effect.
Alternatively, we could decide to keep the parameters $a = 30$, $M = 60$ fixed, but decrease 
the window size to $L_{g}\leq60$ for a painless case setup. However, this construction provides 
a system with a more than $8$ times larger frame bound ratio.
Consequently, the resulting canonical dual window, shown in Figure~\ref{fig:Experiments-painless}, 
shows bad frequency behavior and multiple significant peaks in time. In contrast,
the method proposed in this manuscript allows the construction of nicely shaped,
compactly supported dual Gabor windows at low redundancies, without
the strong restrictions of the painless case.\\

\begin{figure}[!thp]
\begin{centering}
\includegraphics[width=0.23\textwidth]{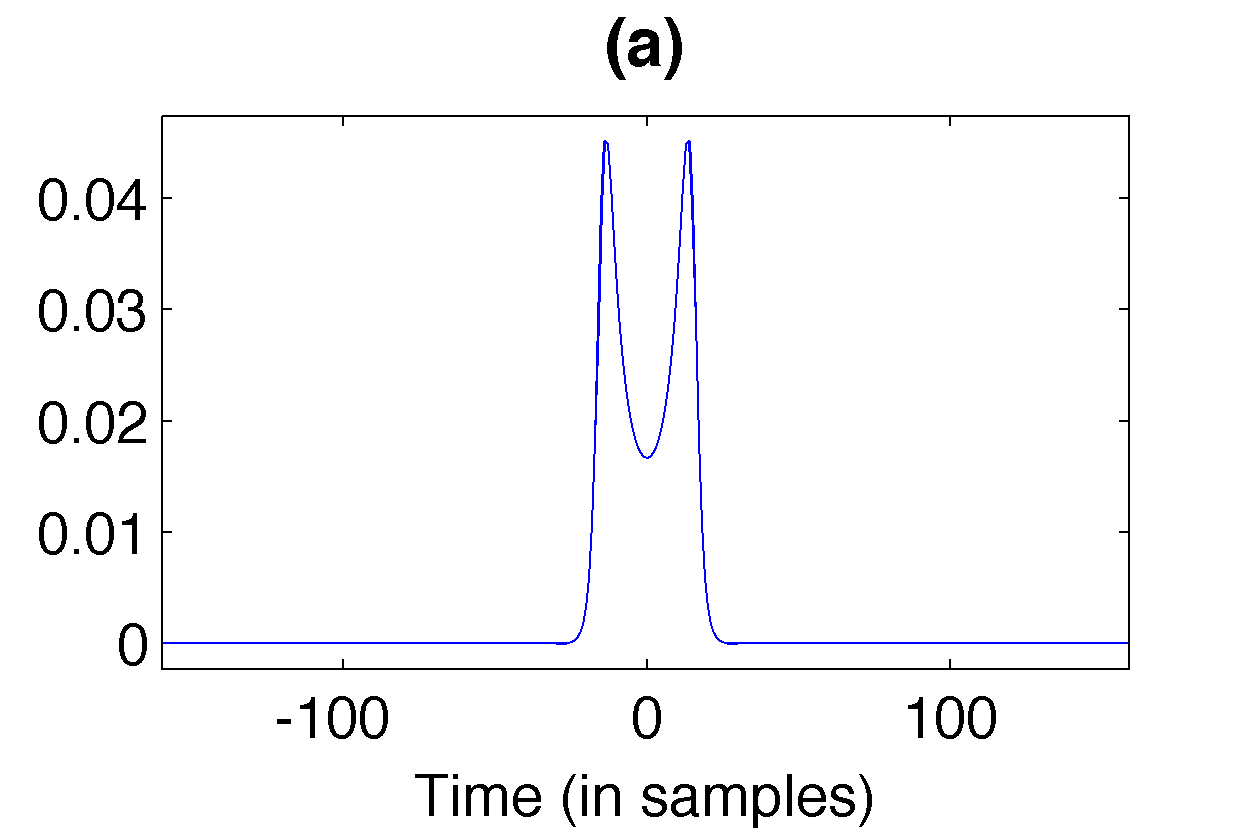}\includegraphics[width=0.23\textwidth]{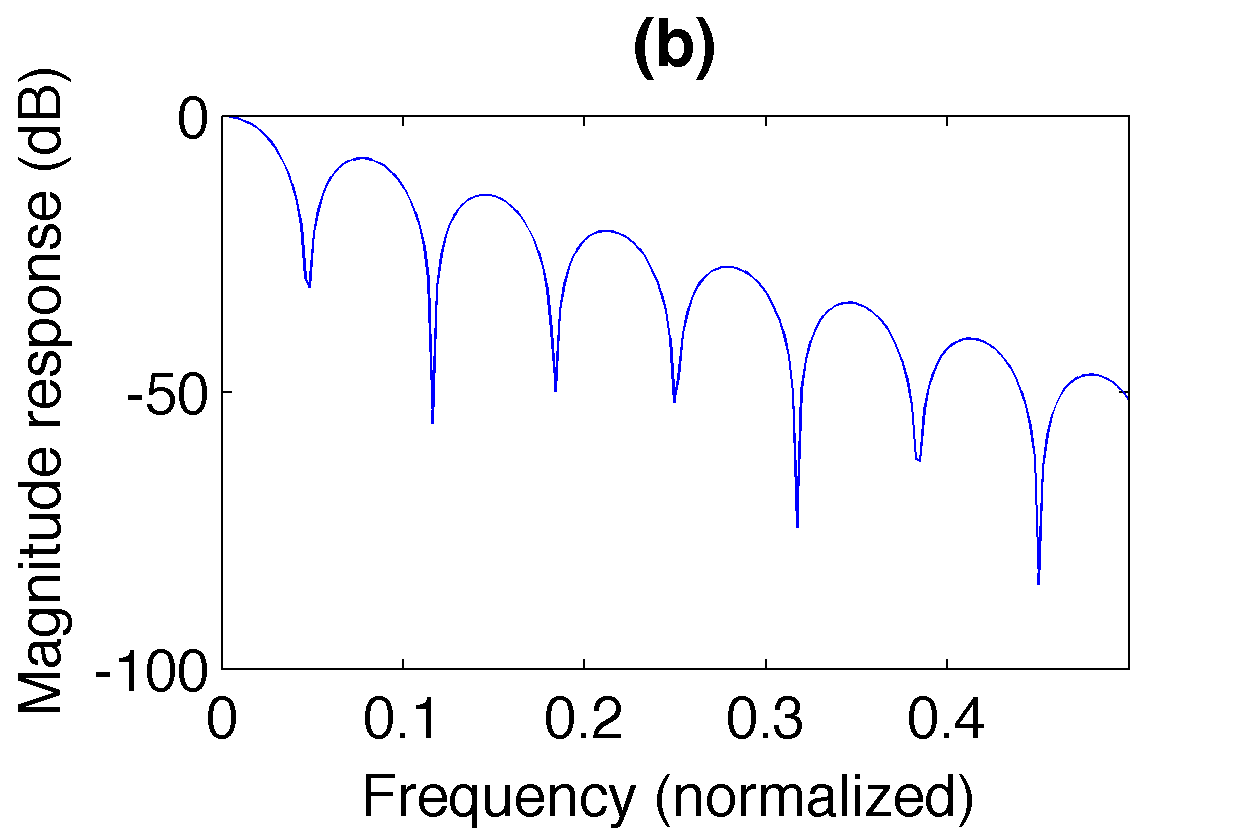} 
\par\end{centering}

\caption{Half-overlap painless case construction
($\mathcal{G}(g,30,60)$, $L_{g}=60$): Canonical dual window in time
(a) and in frequency (b).}\label{fig:Experiments-painless}
\end{figure}

\textbf{Excursion - Improvement of classical tight frame constructions:} 

Sometime significant simplifications of applications or algorithms are achieved, whenever tight frames are used, e.g. the SALSA algorithm for $\ell^1$ regularized least-squares~\cite{Afonso2010}.
Instead of computing the dual window to a previously selected analysis window, we now attempt the computation of a \emph{Gabor Parseval frame}\footnote{A Parseval frame is a tight frame with frame bound $1$.} with good properties, given the parameters $a,M$. A system $\mathcal G(h,a,M)$ is a Parseval frame, if it is a dual frame to itself. We consider only Parseval frames, since every tight frame is Parseval up to a scaling factor. The Gabor Parseval windows $h$ for parameters $a,M$ are characterized by the nonlinear equation system
\begin{equation*}
\frac{M}{a}\left\langle h,h[\cdot-nM]e^{2\pi im\cdot/a}\right\rangle =\delta[n]\delta[m],\label{eq_wexler_raz_tight}
\end{equation*}
obtained by setting $g=h$ in the WR equations~\eqref{eq_wexler_raz}. The solution set associated to this equation system is not convex. Indeed, little is known about the set beyond it being pathwise connected~\cite{lawi05} subset of $\{h\in\ell_2(\ZZ)~:~ \|h\|_2^2 = a/M\}$. Even worse, since we only consider real-valued solution windows, the connectedness property might be violated. Although no convergence guarantees can be given, we will provide a heuristic optimization scheme that provided good experimental results.

In order to implement our proposed scheme, we need a method to compute the projection onto the set of Parseval windows:
\begin{equation*}
P_{\mathcal{C}_{\text{Pars}}}(y)=\mathop{\operatorname{arg~min}}\limits_{x\in \mathcal{C}_{\text{Pars}}} \|x-y\|_2.
\end{equation*}
It can be shown~\cite{janssen2002characterization} that this projection can be computed via the formula
\begin{equation}
P_{\mathcal{C}_{\text{Pars}}}(y)=\bd{S}_{y,a,M}^{-1/2}y,
\end{equation}
where $\bd{S}_{y,a,M}=G^\ast_{y,a,M}G_{y,a,M}$ is the frame operator with respect to $\mathcal G(y,a,M)$, cf.~\cite{janssen2002characterization}. The result of $\bd{S}_{y,a,M}^{-1/2}y$ is called the \emph{canonical tight window} and efficient algorithms for its computation are freely available, e.g. in the LTFAT toolbox~\cite{ltfatnote015}, see also~\cite{jast02}. The canonical tight window always forms a Parseval frame.

We attempt to solve the tight problem with PPXA, simply replacing the projection on the dual set by projection onto the set of Parseval windows. At convergence, a final projection onto the tight set ensures the tightness property of the result. If a support constraint is desired, the final projection is based on a POCS-based algorithm. However, there is no convergence guarantee of this final step, even if $\mathcal{C}_{\text{Pars}}$ and $\mathcal{C}_{\text{supp}}$ are not disjoint. We have no theoretical guarantee to find a good solution to the problem with this method, since both the PPXA and POCS steps lack a convergence guarantee. Nevertheless, our approach is not blindly random. PPXA is a generalization of the Douglas-Rachford algorithm~\cite{combettes2007douglas}. The latter algorithm, when applied to non-convex, lower, semi-continuous functions has been proved to converge to a stationary point, one consequence of a more general minimization scheme presented by Attouch et al.~\cite{
attouch2010proximal}. Unfortunately, the indicator function of the set of real-valued Parseval windows is most likely not semi-continuous and therefore not subject to Attouch's result.

Since the problem is not convex, good starting value and timestep choices are crucial. We have obtained good results and dependable convergence when choosing a starting window that is not too far from what we aim for, i.e. it already has a good frame bound ratio $B/A$ for the Gabor parameters $a,M$ and shows the properties we wish to promote in the tight window, e.g. TF concentration. Note that, especially for frames with small redundancy, it has been observed that a trade-off between localization and smoothness in TF exists between the analysis and dual windows. Therefore, low redundancy Parseval frame windows provide, in comparison, suboptimal TF concentration.

As starting setup, we choose a Gabor system $\mathcal{G}(g,30,60)$ with an Itersine window of length $L_g = 60$. For this half-overlap, redundancy $2$ situation, the Itersine window forms a tight, painless frame with better joint TF concentration than other widely used constructions for redundancy $2$ tight frames, such as the cosine window $\chi_{[-0.5,0.5]}\cos(\pi\cdot)$ or rectangular window $\chi_{[-0.5,0.5]}$. We now attempt the construction of a Gabor Parseval frame with redundancy $2$, using a window function that further improves the TF concentration of the Itersine window.

To gain some design freedom, we allow the tight window $g_t$ to have support length $L_{g_t}\leq 360$. As in the earlier experiments, gradient priors are used to promote a window that is smooth and well-localized in both domains, leading (formally) to the optimization problem 
\[
 \mathop{\operatorname{arg~min}}\limits _{x \in \mathcal{C}_{\text{Pars}} \cap \mathcal{C}_{\text{supp}}} \lambda_1 \|\nabla \mathcal{F} x\|_2^2 + \lambda_2 \|\nabla x\|_2^2.
\]
For easier comparison, we tuned the result to have roughly the same visual concentration in time. The result shown was obtained for the regularization parameters $\lambda_1 = 1$, $\lambda_2 = 5$ and shows improved decay and side lobe attenuation, when compared to the Itersine, see Figure~\ref{fig:Experiments-tight} and Table~\ref{tab:Ex-tight}. 

\begin{figure}[!thp]
\begin{centering}
\includegraphics[width=0.23\textwidth]{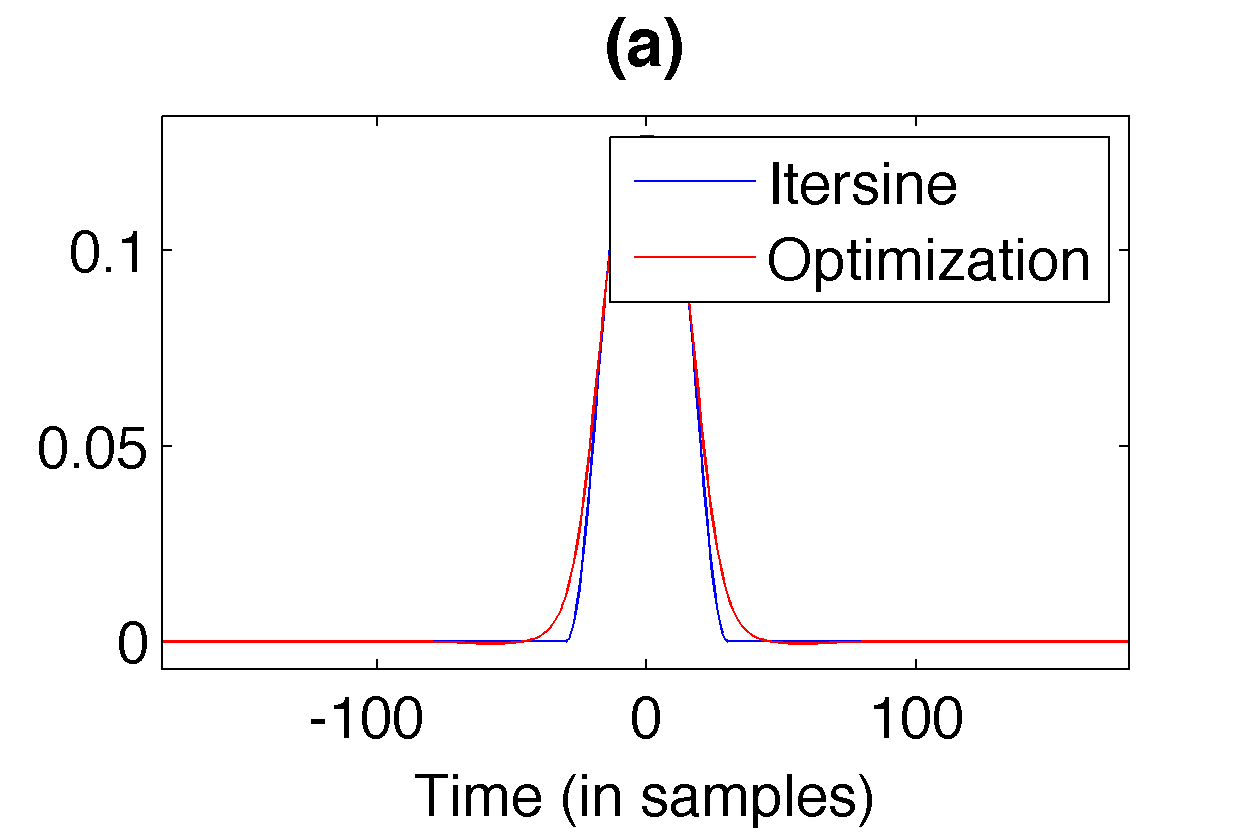}
\includegraphics[width=0.23\textwidth]{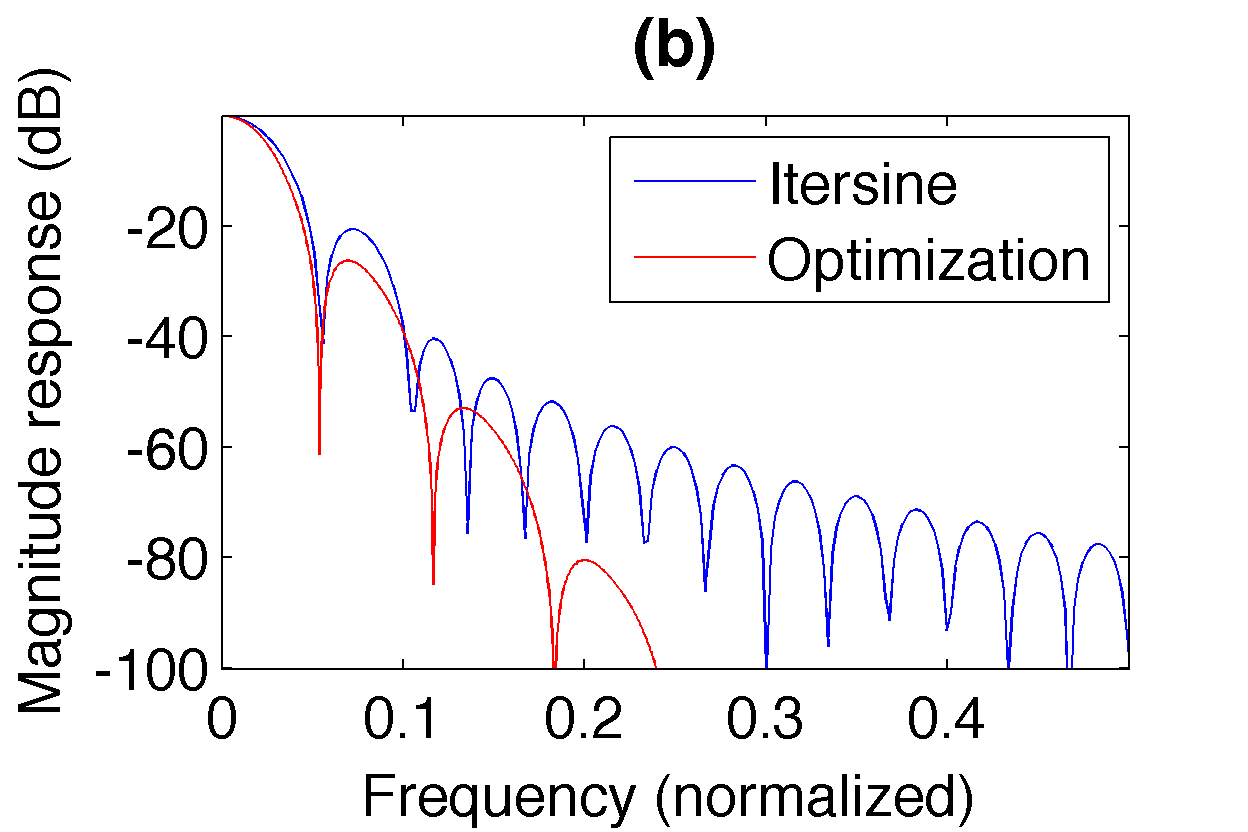}
\par

\end{centering}
\caption{$\mathcal{G}(g,30,60)$, $L_{g}=60$, compactly supported tight windows: $L_g=360$:  Itersine (gray). Result of optimization (black).}\label{fig:Experiments-tight}
\end{figure}

\begin{table}[thb] 
\begin{small}
\begin{center}\begin{tabular}{ |c|ccccc|} 
 \hline 
 & $-3$~dB-W & ML-W & 1/pr-W & SL-A & SL-D \\ 
 \hline 
$x_\text{iter}$  & $  10.2778 $ & $   5.5449 $ & $ 175.4726 $ & $  20.5885 $ & $  78.1983 $  \\ 
$x_\text{opt}$ &$  10.2778 $ & $   \mathbf{5.4167} $ & $ \mathbf{179.6258} $ & $  \mathbf{26.1916} $ & $ \mathbf{126.9964} $ \\ 
 \hline 
 \end{tabular}\end{center}
 \begin{center}\begin{tabular}{ |c|cc|} 
 \hline 
 &  $\frac{\|\nabla x\|_2}{100}$ & $\frac{\|\nabla \mathcal{F} x\|_2}{100}$\\ 
 \hline 
$x_\text{iter}$  &  $   4.1096 $ & $   \mathbf{5.5811} $ \\ 
$x_\text{opt}$ & $   \mathbf{3.5207} $ & $   6.0137 $ \\ 
 \hline 
 \end{tabular}\end{center}
 \end{small}
 \caption{Window quality measures and gradient criteria for the tight window experiment. The subscript indicates the window: $x_\text{iter}$: Itersine window, $x_\text{opt}$: Optimized tight window. Best values indicated in bold.}\label{tab:Ex-tight}
\end{table} 

Although the heuristic tight frame optimization has provided promising results, there is no guarantee for the results' optimality. Whenever the tight frame property is not essential, the construction of a pair of dual frames with good TF concentration might be preferable. For the same parameters as before, we choose a Nuttall window of length $L_{g} = 120$ to construct a frame. This window is slightly broader in time than the Itersine window used before, but provides very good decay and concentration in frequency, see Figure~\ref{fig:beat_itersine}. For the dual window, we consider $L_h \leq 360$ as in the tight case and the optimization problem
\[
 \mathop{\operatorname{arg~min}}\limits _{x \in \mathcal{C}_{\text{dual}} \cap \mathcal{C}_{\text{supp}}} \lambda_1 \|\nabla \mathcal{F} x\|_2^2 +  \lambda_2\|\nabla x\|_2^2.
\]
Again, the regularization parameters have been tuned to provide similar concentration in time. The dual window shown in Figure~\ref{fig:beat_itersine} was obtained for the regularization parameters $\lambda_1 = 0.1,\ \lambda_2 = 1$. Both the Nuttall window and its optimized dual show improved joint TF concentration over the tight Itersine window, see also Table~\ref{tab:Ex-tight2}. In terms of joint TF localization, the values presented in Table~\ref{tab:Ex-tight} and Table~\ref{tab:Ex-tight2} suggest that all $3$ prototypes constructed in this excursion show a considerable improvement upon the Itersine window and consequently over other widely used redundancy $2$ tight Gabor frame constructions.

\begin{figure}[!thp]

\begin{centering}
\includegraphics[width=0.23\textwidth]{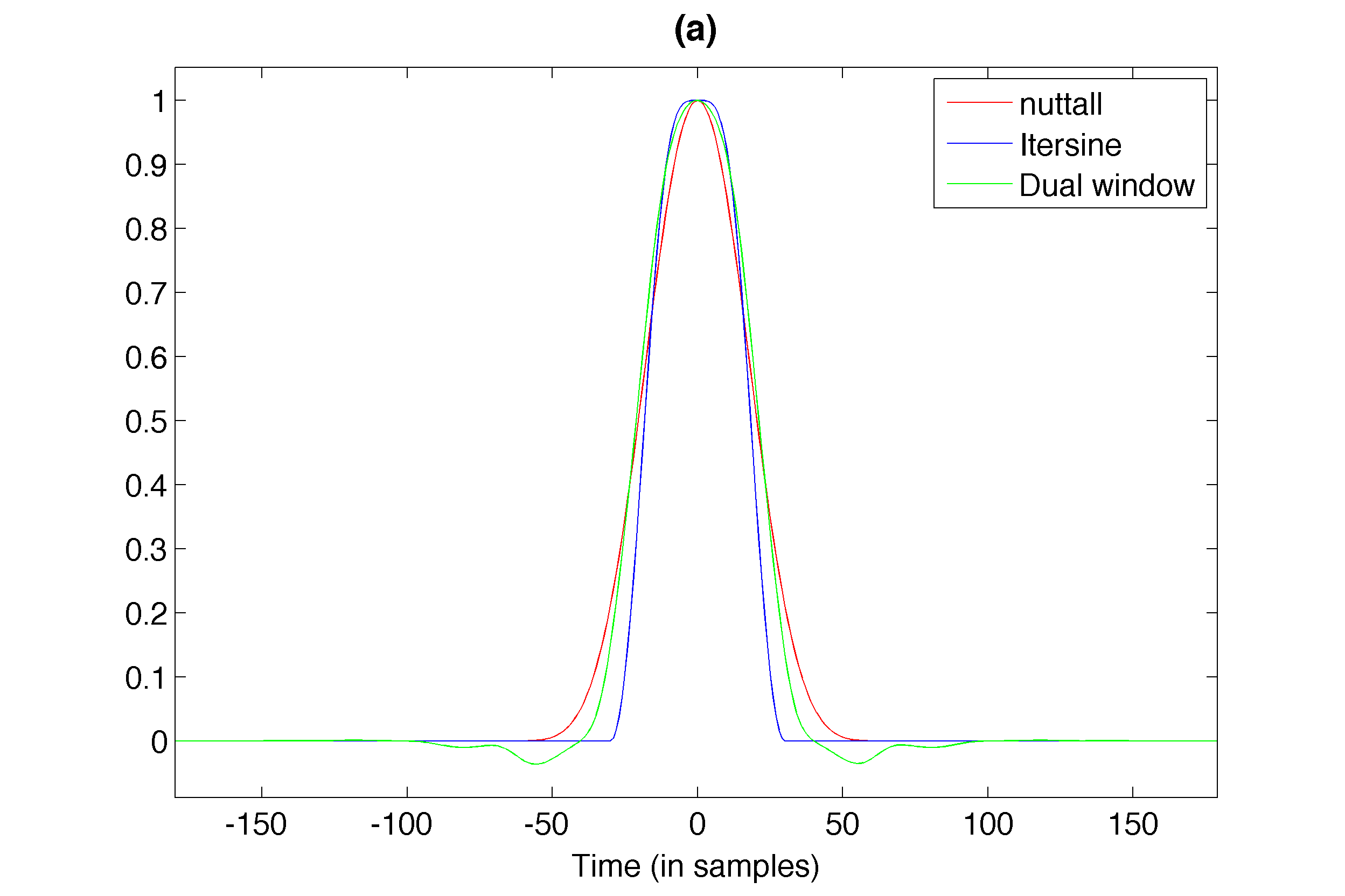}
\includegraphics[width=0.23\textwidth]{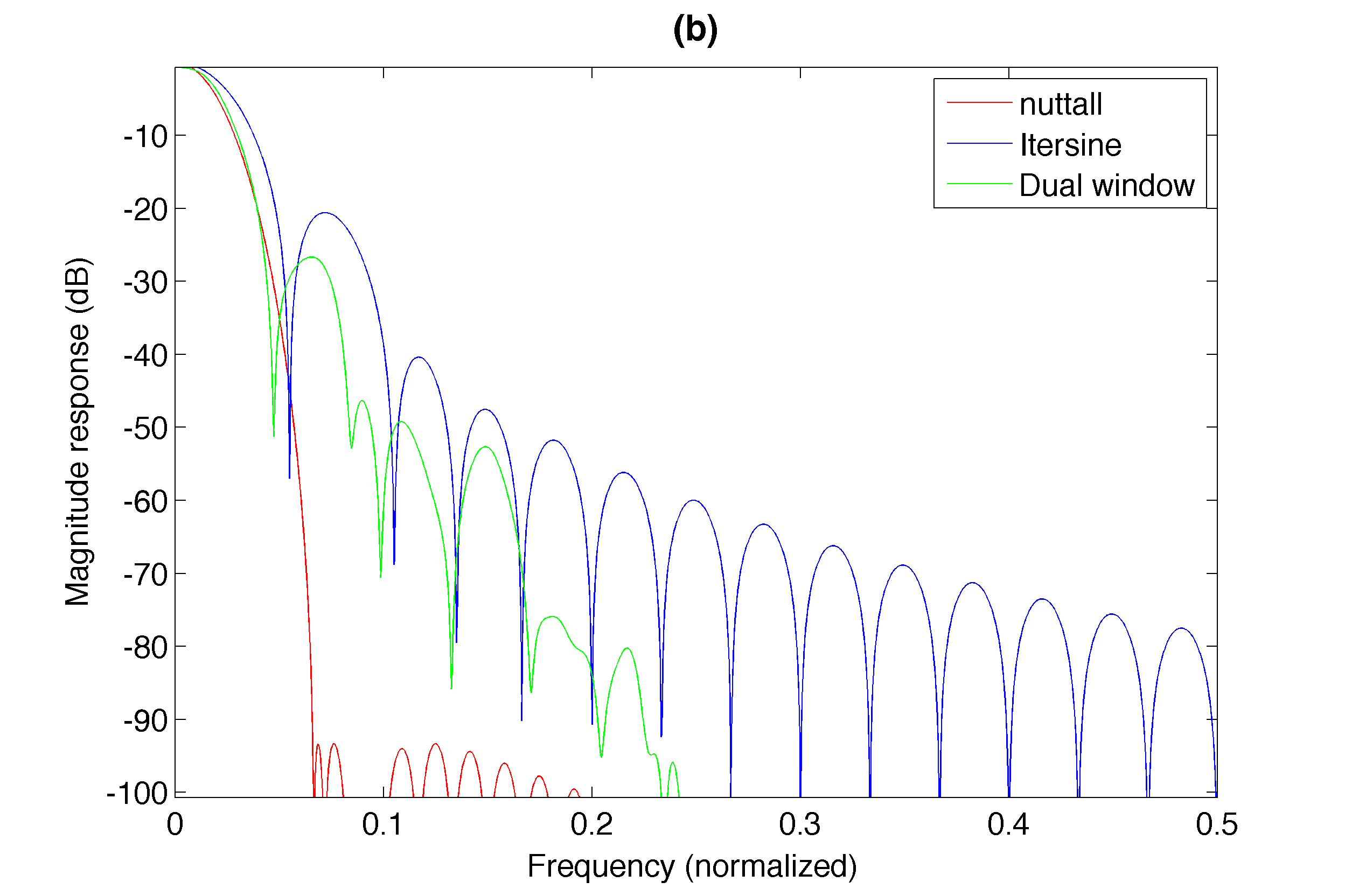}
%
%

\end{centering}
\caption{Construction of a pair of dual FIR windows both improving Itersine TF concentration
($\mathcal{G}(g,30,60)$, Nuttall $L_{g_0}=120$ (gray), Itersine $L_{g}=60$ (gray dotted), Dual window $L_{h}=120$ (black)): Left: Windows in time. Right: Windows in frequency.\label{fig:beat_itersine}}
\end{figure}

\begin{table}[thb] 
\begin{center}\begin{tabular}{ |c|ccccc|} 
 \hline 
 & $-3$~dB-W & ML-W & 1/pr-W & SL-A & SL-D \\ 
 \hline 
$x_\text{iter}$  & $   \mathbf{10.2778} $ & $   5.5093 $ & $ \mathbf{176.6069} $ & $  20.5903 $ & $  78.1948 $ \\ 
$x_\text{nutt}$ & $  11.3889 $ & $   6.7130 $ & $ 130.7990 $ & $  \mathbf{93.3292} $ & $  55.7800 $ \\ 
$x_\text{opt}$ & $  11.9444 $ & $   \mathbf{4.7685} $ & $ 175.5701 $ & $  26.0021 $ & $  \mathbf{91.1066} $ \\ 
 \hline 
 \end{tabular}\end{center} 
 \begin{center}\begin{tabular}{ |c|cc|} 
 \hline 
 &  $\frac{\|\nabla x\|_2}{100}$ & $\frac{\|\nabla \mathcal{F} x\|_2}{100}$\\ 
 \hline 
$x_\text{iter}$   & $   4.3278 $ & $   \mathbf{8.4892} $ \\ 
$x_\text{nutt}$ &  $   3.1292 $ & $  10.5630 $ \\ 
$x_\text{opt}$ &  $   \mathbf{3.0412} $ & $   9.1353 $ \\ 
 \hline 
 \end{tabular}\end{center} 
 \caption{Window quality measures and gradient criteria for the tight window experiment. The subscript indicates the window: $x_\text{iter}$: Itersine window, $x_\text{nutt}$: Nuttall window, $x_\text{dual}$: Optimized dual window.  Best values indicated in bold. Minor differences to Table VI are due to different assumed length $L$. They amount to either sampling issues (ML-W) or a scaling factor ($\|\nabla \mathcal{F} x\|_2$).}\label{tab:Ex-tight2}
\end{table}


\section{Conclusion}

In this contribution, we have proposed a convex optimization framework for the 
computation of optimized Gabor dual windows. The presented method is based on the 
observation that the set of dual windows for a fixed Gabor filterbank can be described 
as the solution set to the linear Wexler-Raz equations. Furthermore, we exploit the facts 
that support constraints can be expressed as linear equations and that compactly supported 
Gabor windows are dual independent of the underlying signal length.

The resulting scheme enables the computation of alternative dual windows, with the possibility 
to optimize a wide variety of criteria, freely chosen by the user. Although the complexity varies
with the selected priors, results are usually obtained efficiently using the provided open-source 
implementation, supplied on the associated webpage \url{https://lts2.epfl.ch/rrp/gdwuco/}.

We provided several demonstrations of the method's capability in the context of joint time-frequency concentration 
optimization. In particular, we constructed dual windows that considerably improve time-frequency concentration
over the widely used canonical dual, windows that provide alternative ratios between time and frequency concentration 
and windows that combine short support, smoothness and frequency concentration. Finally, we showed that we can, heuristically,
apply the proposed framework to compute well-concentrated tight windows, improving previously known explicit tight frame
constructions.

The results show that our method can be applied in various situations to construct dual
frames with properties more relevant for application than minimal $\ell^{2}$-norm. 

Future work will concern the extension of the presented scheme to more general systems, e.g.
Gabor systems with complex-valued window on nonseparable sampling sets~\cite{hosowi13}
and nonstationary Gabor frames~\cite{nsdgt10}.


\section*{Acknowledgement} 
We thank the anonymous reviewers for their constructive comments that helped us improve the paper. 

This work was supported by the Austrian Science Fund (FWF)
START-project FLAME (``Frames and Linear Operators for Acoustical
Modeling and Parameter Estimation''; Y 551-N13).

\section*{Bibliography}
\bibliographystyle{elsarticle-num}
\bibliography{biblio}

\end{document}